\documentclass[12pt]{amsart}
\usepackage[T1]{fontenc}
\usepackage[utf8]{inputenc}

\usepackage{amstext}
\usepackage{amsmath, amssymb, amsthm}
\usepackage{amssymb}
\usepackage[a4paper, total={6in, 9in}, centering]{geometry}
\usepackage{enumerate}
\usepackage{xcolor}
\usepackage{comment}
\usepackage{hyperref}
\usepackage{graphicx}

\newtheorem{theorem}{Theorem}[section]
\newtheorem{lemma}[theorem]{Lemma}
\newtheorem{proposition}[theorem]{Proposition}
\newtheorem{corollary}[theorem]{Corollary}

\newtheorem{definition}[theorem]{Definition}

\newtheorem{conjecture}[theorem]{Conjecture}

\DeclareMathOperator{\graph}{graph}
\DeclareMathOperator{\trace}{tr}
\DeclareMathOperator{\diam}{diam}
\DeclareMathOperator{\loc}{loc}
\DeclareMathOperator{\dist}{dist}

\DeclareMathOperator{\Div}{div}

\title[Rotational symmetry of ancient solutions]{Rotational symmetry of ancient solutions to fully nonlinear curvature flows}

\author{A. Cogo}
\address{Eberhard Karls Universit\"at T\"ubingen, Fachbereich Mathematik, Auf der Morgenstelle 10, 72076 T\"{u}bingen, Germany}\email{albachiara.cogo@uni-tuebingen.de}

\author{S. Lynch}
\address{Department of Mathematics, Imperial College London, London SW7 2AZ, United Kingdom}
\email{stephen.lynch@imperial.ac.uk} 

\author{O. Vi\v{c}\'{a}nek Mart\'{i}nez}
\address{Eberhard Karls Universit\"at T\"ubingen, Fachbereich Mathematik, Auf der Morgenstelle 10, 72076 T\"{u}bingen, Germany}
\email{olivia.vicanek-martinez@math.uni-tuebingen.de}

\begin{document}

\begin{abstract}
We address the classification of ancient solutions to fully nonlinear curvature flows for hypersurfaces. Under natural conditions on the speed of motion we classify ancient solutions which are convex, noncollapsing, uniformly two-convex and noncompact. There are exactly two possibilities---every such solution is either a self-similarly shrinking cylinder, or else is a rotationally symmetric translating soliton. For a large class of flows this yields a complete classification of the blow-up limits that can arise at a singularity of a solution which is compact, embedded and two-convex. 
\end{abstract}
	
\maketitle

\tableofcontents

\pagebreak
\section{Introduction}\label{intro}

The study of singularity formation in geometric flows began with the fundamental work of Hamilton \cite{Hamilton_3D}. There it was shown that positively curved three-manifolds become singular in finite time under the Ricci flow, and are deformed to a constant curvature space in the process. This is an example of a phenomenon which has since been observed in many other situations---close to a singularity, solutions to geometric flows tend to become highly symmetric. This allows for a very precise description of the geometric structure of singularities; an essentially complete picture has now been established e.g. for the three-dimensional Ricci flow and for two-convex mean curvature flow. In these settings it has been possible to extend the flow in question beyond singularities using a surgery procedure (see \cite{Perelman1,Perelman2} and \cite{Huisken--Sinestrari_Surgery, Brendle--Huisken_Surgery, Haslhofer--Kleiner_Surgery}), enabling profound insights into questions in geometry and topology. 

A solution of a geometric flow is called ancient if it exists on a time interval of the form $(-\infty, T]$. Ancient solutions arise as dilations of singularities, so their classification is of central importance in understanding singularity formation. In this paper we are concerned with the classification of ancient solutions to fully nonlinear geometric flows. Consider a family of hypersurfaces $M_t \subset \mathbb{R}^{n+1}$, $t \in (-\infty, T]$, which move with pointwise normal speed equal to a symmetric function of their principal curvatures. Under natural assumptions on the speed function, we show that if $M_t$ is convex, noncollapsing, uniformly two-convex and noncompact then it is rotationally symmetric. Moreover, $M_t$ is either a self-similarly shrinking round $\mathbb{R}\times S^{n-1}$, or else moves by translation in a constant direction (i.e. it is a `bowl soliton'). 

Brendle and Choi proved the corresponding statement for ancient solutions of the mean curvature flow \cite{Brendle_Choi_a, Brendle_Choi_b} (see also \cite{ADS,ADS_uniqueness} and \cite{CHH_MCnbhd, CHHW}). One principal difficulty in our more general setting is that we may not appeal to Huisken's monotonicity formula. Huisken's formula asserts that, up to a certain nonlinear change of variables, the mean curvature flow is the gradient flow for the Gaussian area functional. The vast majority of the flows with which we are concerned lack an analogous interpretation. One of our main contributions is to reinterpret a number of deep phenomena, usually derived from the monotonicity formula, as consequences of mild concavity properties for the speed of motion. 

In certain situations our results yield a complete classification of the possible ancient solutions that can arise by dilating a singularity. A key example we have in mind is the flow which moves $M_t$ with speed
    \begin{equation}\label{2-harmonic mean}
    \gamma(\lambda) = \bigg(\sum_{i<j} \frac{1}{\lambda_i + \lambda_j - 2\kappa}\bigg)^{-1},
    \end{equation}
where the $\lambda_i$ are the principal curvatures. This flow was first considered by Brendle and Huisken \cite{Brendle_Huisken}. They developed a surgery procedure for continuing the flow past singularities, and thereby proved a remarkable classification result for compact Riemannian manifolds. Namely, if $(X,g)$ has boundary satisfying $\lambda_1 + \lambda_2 > 2\kappa$ and its curvature satisfies $R_{ikik} + R_{jkjk} \geq -\kappa^2$, then $X$ is diffeomorphic to a standard ball or 1-handlebody. 

The surgery procedure implemented by Brendle--Huisken relies on a priori cylindrical and gradient estimates (cf. \cite{Huisken--Sinestrari_Surgery}), which yield a detailed picture of regions of large curvature in the solution. Our results go further, showing that any blow-up at a singularity of the flow is rotationally symmetric, and either a round sphere, cylinder or bowl soliton (see Corollary~\ref{singularities_Riemannian} for a precise statement). 

A key point in \cite{Brendle_Huisken} is that the flow with speed \eqref{2-harmonic mean} preserves two-convexity in an arbitrary ambient space. This is not true of the mean curvature flow, for example. Fully nonlinear flows were also employed to study the global geometry of Riemannian manifolds in the earlier works \cite{Andrews_Riemannian, Andrews_Sphere}. Each of the papers we have mentioned employs a flow which is tailored to the specific geometric problem at hand. This highlights the necessity of developing a theory for the widest possible range of flows, in order to enable a range of geometric applications in the future. 

\subsection{Main results} Let $\Gamma \subset \mathbb{R}^n$ denote an open, symmetric (under permutations), convex cone. Fix a function $\gamma:\Gamma \to \mathbb{R}$ which is smooth, positive, symmetric, strictly increasing in each argument, and homogeneous of degree one. We assume $\gamma$ is either \textbf{convex} or \textbf{concave}. 

We also assume that $\Gamma$ contains the positive cone $\mathbb{R}^n_+ := \{\lambda : \min_i \lambda_i > 0\}$ and the cone $\{0\} \times \mathbb{R}_+^{n-1}$, and that both $\gamma|_{\mathbb{R}^n_+}$ and $\gamma|_{\{0\} \times \mathbb{R}_+^{n-1}}$ are strictly inverse-concave functions. (These conditions ensure that we have a splitting theorem \cite{Lynch_22_b} and a differential Harnack inequality \cite{Lynch_Harnack}.) Note that strict inverse-concavity is automatic if $\gamma$ is convex. 

For a hypersurface $M$ with principal curvatures $\lambda = (\lambda_1, \dots, \lambda_n)$ in $\Gamma$ we write $G = \gamma(\lambda)$. The principal curvatures are the eigenvalues of the second fundamental form $A$ with respect to the induced metric $g$. We label these so that $\lambda_1 \leq \dots \leq \lambda_n$. The outward unit normal to $M$ is denoted $\nu$.

A $G$-flow is a smooth one-parameter family of embedded hypersurfaces $M_t$ in $\mathbb{R}^{n+1}$ which move inwards with pointwise velocity $-G\nu$. We say that a $G$-flow is: 

\begin{enumerate}
\item convex if $M_t = \partial \Omega_t$, where $\Omega_t$ is convex;
\item uniformly two-convex if there is a constant $\beta > 0$ such that $\lambda_1 + \lambda_2 \geq \beta H$, where $H$ is the mean curvature;
\item noncollapsing if there is a constant $\alpha > 0$ such that every point $x \in M_t$ admits an inscribed ball of radius at least $\alpha G(x,t)$.
\end{enumerate}

We can now state our main result. 

\begin{theorem}\label{main}
Let $M_t$, $t \in (-\infty, T]$, be a convex ancient $G$-flow which is noncollapsing and uniformly two-convex. Suppose also that $M_t$ is noncompact. Then $M_t$ is either a self-similarly shrinking round $\mathbb{R}\times S^{n-1}$, or else is a rotationally symmetric translating soliton.
\end{theorem}

We say that $M_t$, $t \in (-\infty, T]$, is a translating soliton if there is a constant vector $V \in \mathbb{R}^{n+1}$ such that $M_t = M_T + (t-T)V$. Up to rigid motions of the ambient space and parabolic rescalings, there is a unique rotationally symmetric translating $G$-flow (the `bowl' soliton). This solution was constructed rigorously in \cite{rengaswami2021rotationally}.

As we mentioned above, Theorem~\ref{main} applies when $\gamma$ is the speed \eqref{2-harmonic mean} introduced in \cite{Brendle_Huisken}. It also applies to the speeds introduced in \cite{Lynch_Convexity}, and when $\gamma$ is a ratio of elementary symmetric polynomials $\sigma_k/\sigma_{k-1}$ such that $n \geq k + 2$. 

Under some further conditions on $\gamma$, Theorem~\ref{main} provides a complete classification of the possible limits that can arise when we blow up a singularity of a compact $G$-flow.

\begin{corollary}\label{singularities}
Suppose $\gamma$ is concave. Suppose also that $\Gamma$ is contained in the two-positive cone $\{\lambda : \min_{i < j} \lambda_i + \lambda_j > 0\}$, and that $\gamma$ admits a continuous extension to $\bar \Gamma$ which vanishes on $\partial \Gamma$. Consider a compact embedded hypersurface $M_0$ in $\mathbb{R}^{n+1}$ whose principal curvatures lie in $\Gamma$ (this implies strict two-convexity). Let $M_t$ denote the maximal evolution of $M_0$ which moves inwards with pointwise velocity $-G\nu$. The hypersurfaces $M_t$ form a singularity in finite time, and every blow-up of $M_t$ at the singular time is a self-similarly shrinking round $S^n$ or $\mathbb{R} \times S^{n-1}$, or else is the bowl soliton. 
\end{corollary}

Corollary~\ref{singularities} applies, for example, when $n \geq 3$, $\gamma$ is the speed \eqref{2-harmonic mean}, and $\Gamma$ is the two-positive cone.

Let us indicate how Corollary~\ref{singularities} is proven. We write $C$ for positive constant depending on $M_0$. Assuming $\gamma$ is concave, the parabolic maximum principle implies that $\min_{M_t} G > C^{-1}$ and $\max_{M_t} H/G \leq C$. Assuming that $\Gamma$ sits inside the two-positive cone and $\gamma$ vanishes on $\partial \Gamma$, the bound $H/G \leq C$ implies uniform two-convexity, and also allows us to estimate
    \[\bigg(\frac{\partial}{\partial t} - \frac{\partial \gamma}{\partial A_{ij}} \nabla_i \nabla_j\bigg) G = \frac{\partial \gamma}{\partial A_{ij}} A_i^kA_{kj} G \geq \frac{1}{C} G^3.\]
Therefore, by the maximum principle, a singularity must form in finite time. The flow $M_t$ is noncollapsing by \cite{ALM_Noncollapsing}, and satisfies a convexity estimate (implying that blow-ups are convex) by \cite{Brendle_Huisken} (see also \cite{LangfordLynch}). The gradient estimates in \cite{Brendle_Huisken} (see also \cite{lynch2020convexity}) imply that if $G(p_k, t_k) \to \infty$ then the sequence 
    \[G(p_k,t_k)(M_{t/G(p_k,t_k)^2 + t_k} - p_k)\]
subconverges in $C^\infty_{\loc}$ to an ancient $G$-flow. This ancient flow is either compact, in which case it is a sphere by \cite{Andrews_Pinching, Andrews_Euclidean}, or else is noncompact and satisfies the hypotheses of Theorem~\ref{main}. The claim follows. 

Corollary~\ref{singularities} also applies when $M_t$ is the boundary of a region of a compact Riemannian $(n+1)$-manifold, provided the ambient space satisfies a natural curvature condition. In particular, it applies in the setting of \cite{Brendle_Huisken}. 

\begin{corollary}\label{singularities_Riemannian}
Let $\gamma: \Gamma \to \mathbb{R}$ be as in the statement of Corollary~\ref{singularities}. Additionally, we require that $\sup_{\Gamma} |\frac{\partial \gamma}{\partial \lambda_i}| < \infty$ for each $1 \leq i \leq n$. Let $N^{n+1}$ be a compact Riemannian manifold, and suppose the curvature tensor $R$ of $N$ satisfies
    \[(R(e_1,e_{n+1},e_1,e_{n+1})+\kappa^2, \dots, R(e_n,e_{n+1},e_n,e_{n+1})+\kappa^2) \in \bar \Gamma\]
for every orthonormal frame $\{e_i\}_{i=1}^{n+1}$, where $\kappa \geq 0$ is a constant. Let $M_0 = \partial \Omega_0$ be a smooth hypersurface in $N$ whose principal curvatures satisfy
    \[(\lambda_1 - \kappa, \dots, \lambda_n - \kappa) \in \Gamma,\]
and let $M_t$ denote the maximal evolution of $M_0$ which moves inwards with pointwise velocity $-G_\kappa \nu$, where $G_\kappa :=\gamma(\lambda_1 - \kappa, \dots, \lambda_n-\kappa)$. The hypersurfaces $M_t$ form a singularity in finite time, and (after pulling back by the exponential map) every blow-up of $M_t$ at the singular time is a self-similarly shrinking round $S^n$ or $\mathbb{R} \times S^{n-1}$, or else is the bowl soliton, in $\mathbb{R}^{n+1}$.
\end{corollary}

The proof of Corollary~\ref{singularities_Riemannian} is very similar to that of Corollary~\ref{singularities}. All of the adjustments that need to be made can be readily extracted from \cite{Brendle_Huisken}. 

Based on Theorem~1 and the works \cite{ADS, ADS_uniqueness}, we are led to the following

\begin{conjecture}\label{3-convex}
Up to parabolic rescalings and rigid motions of the ambient space, there is a unique convex ancient $G$-flow which is noncollapsing, uniformly two-convex and compact. This solution is the ancient `oval' constructed in \cite{Lu} (see also \cite{Risa--Sinestrari_ovals}).
\end{conjecture}

We note that there is an ongoing program \cite{Zhu, CHH, Du-Haslhofer_a, Du-Haslhofer_b, CDDHS} aimed at classifying noncollapsing convex ancient solutions to mean curvature flow which lie in $\mathbb{R}^4$, or are uniformly three-convex in $\mathbb{R}^{n+1}$. The present work initiates a corresponding program for fully nonlinear flows. 

While all of our results concern noncollapsing ancient solutions, we would like to point out that there are many interesting phenomena to be studied in the collapsing case. Collapsing ancient solutions to curvature flows appear to exhibit far less symmetry, in general, than their noncollapsing cousins \cite{BLT_polytope}. A first step towards understanding this rich class of solutions was recently taken in \cite{Langford--Rengaswami}, which constructed collapsing ancient `pancake' solutions in great generality. 

\subsection{Overview of the proof} Let $M_t$ denote a noncompact, convex ancient $G$-flow which is noncollapsing and uniformly two-convex. We write $\bar M_\tau$ for the rescaled family of hypersurfaces $e^{\tau/2} M_{-e^{-\tau}}$, which move with velocity $-(G - \frac{1}{2}\langle x, \nu\rangle) \nu$.

In Section~\ref{preliminaries section} we collect preliminary results to be used throughout the paper. We argue in particular that, unless $M_t$ is strictly convex, it splits off a line and hence is a shrinking cylinder by a result in \cite{LangfordLynch}. So for the remainder of this discussion let us assume $M_t$ is strictly convex. The claim is that $M_t$ must then be the bowl soliton. 

In Section~\ref{blow-down section} we begin studying the behaviour of the hypersurfaces $M_t$ at very early times. We show that as $\tau \to -\infty$, up to a choice of ambient Euclidean coordinates, the family $\bar M_\tau$ converges in $C^\infty_{\loc}$ to the cylinder $\Sigma = \mathbb{R}\times S^{n-1}((2\gamma(0,1,\dots,1))^{1/2})$. In the case of mean curvature flow, this property would typically be established using Huisken's monotonicity formula, but no such tool is available for fully nonlinear flows; the majority of these certainly lack any kind of variational structure. In \cite{Wang}, Wang found another approach which does not rely on the monotonicity formula. The idea is to extract a (subsequential) blow-down limit, and then translate it infinitely far along its asymptotic cone. The limit of the translated solutions is again an ancient solution, which splits off a line and encloses the blow-down. Moreover, both the translated solution and the blow-down form a singularity at the origin at the same time. Wang then exploits the fact that lower-dimensional slices through a convex mean curvature flow satisfy an avoidance principle to conclude that the translated solution actually coincides with the blow-down. In particular, the blow-down splits off a line. Since it is also uniformly two-convex, the blow-down is a cylinder by \cite{Huisken_Convex} (this also follows from \cite{Huisken--Sinestrari_ancient, Haslhofer--Hershkovits}).

Unfortunately, for flows by other curvature functions, there is no avoidance principle for lower-dimensional slices. The validity of such a principle for the mean curvature flow traces back to the fact that the graphical mean curvature operator is the Laplacian (and in particular is isotropic in the Hessian) at points with horizontal tangent planes. This does not turn out to be a major issue---the translated solution and blow-down can still be shown to coincide via a barrier argument (this argument appeared in \cite{Lynch_22_b}, where it was used by the second author to classify Type I convex ancient solutions to a large class of fully nonlinear flows). In short, we are able to adapt Xu-Jia Wang's approach to prove subsequential convergence of the family $\bar M_\tau$ to a cylinder of radius $(2\gamma(0,1,\dots,1))^{1/2}$ whose axis passes through the origin. This leaves open the possibility that $\bar M_\tau$ might slowly rotate, so that along different subsequences we might see cylinders with different axes arising in the limit. We rule out this possibility by showing that the asymptotic cone of $\bar M_\tau$ is independent of $\tau$. The subsequential convergence can then be upgraded to full convergence to $\Sigma$. 

The next step is to prove a sharp asymptotic estimate for the rate at which $\bar M_{\tau}$ converges to $\Sigma$. For this we make use of certain barrier solutions, which are constructed in Section~\ref{barrier section}. For each sufficiently large $a > 0$ we construct an embedded disc $\Sigma_a$ which is strictly convex, rotationally symmetric, and generates a self-similarly shrinking $G$-flow. Equivalently, each $\Sigma_a$ is a stationary solution of the rescaled flow. We also establish rather precise asymptotics for these solutions as $a \to \infty$, which allow us to use them as barriers for $\bar M_\tau$. To describe these asymptotics let us express $\Sigma_a$ as the set of points of the form 
    \[(z,\Psi_a(z)\theta), \qquad z \in [L_0, a], \; \theta \in S^{n-1},\]
where $\Psi_a$ is a nonnegative concave function. We show that $L_0$ can be fixed independently of $a$. The tip of $\Sigma_a$ is located at $(a, 0)$, and we have $\Psi_a(z)^2 < 2\gamma(0,1,\dots,1)$, meaning $\Sigma_a$ lies inside the region bounded by $\Sigma$. In addition, there is a constant $C = C(n,\gamma)$ such that, for every bounded $L$, when $a$ is sufficiently large we have
    \[\Psi_a(z)^2 \leq 2\gamma(0,1,\dots,1)\bigg(1-\bigg(1-C\frac{\log a}{a^2}\bigg)\frac{z^2 - C}{a^2}\bigg).\]
for $L_0 \leq z \leq L$. We also have the lower bound
    \[\Psi_a(z)^2 \geq 2\gamma(0,1,\dots,1)\bigg(1 - \frac{z^2}{a^2}\bigg)\]
for $L_0 \leq z \leq a$. Analogous self-similarly shrinking solutions to the mean curvature flow were constructed in \cite{ADS} and played a central role in \cite{Brendle_Choi_a, Brendle_Choi_b, ADS_uniqueness}. Our construction is similar to that in \cite{ADS}, but is made substantially more difficult by the fact that the profile function $\Psi_a$ now solves a fully nonlinear ODE. 

By the results of Section~\ref{blow-down section}, we may express $\bar M_\tau$ as the graph of a function $u(z, \theta, \tau)$ over the $z$-axis, and we have $u \to 0$ in $C^\infty_{\loc}$ as $\tau \to -\infty$. In Section~\ref{asymptotics section} we prove the following sharp asymptotic estimate for $u$: for each $L < \infty$, we have
    \begin{equation}
    \label{asymptotic intro}
    \sup_{|z| \leq L} |u(\cdot,\tau)| \leq O(e^{-\tau/2}).
    \end{equation}
This is achieved by performing a spectral decomposition of $u$ with respect to the linearised rescaled $G$-flow about the cylinder $\Sigma$, and determining that this decomposition is dominated by unstable modes as $\tau \to -\infty$. The argument is one of the key steps in our analysis, so let us describe it in more detail. 

We show that $u$ satisfies
    \begin{equation}\label{linearised rescaled flow intro}
    \frac{\partial u}{\partial \tau} = a\frac{\partial^2 u}{\partial z^2} + \frac{1}{2(n-1)} \Delta_{S^{n-1}} u - \frac{1}{2} z \frac{\partial u}{\partial z} + u + E,
    \end{equation}
where $a:= \dot \gamma^1(0,1,\dots,1)$ and $E = O(|u|^2 + |\nabla u|^2 + |\nabla^2 u|^2)$. For each large $-\tau$ this holds in a large but bounded subset of space. The linear elliptic operator on the right-hand side of \eqref{linearised rescaled flow intro} admits a complete set of $L^2$-eigenfunctions with respect to the measure $e^{-z^2/4a}\,dzd\theta$, and its eigenvalues are $1-\frac{k}{2} - \frac{1}{2(n-1)}l(l+n-2)$, where $k, l \geq 0$ are integers. After cutting off in space we decompose $u$ into pieces lying in the positive, neutral and negative eigenspaces. These evolve as one would expect for a solution to the linearised equation, up to error terms depending on $E$ and the cutoff function. We show that all of these errors decay rapidly as $\tau \to -\infty$ by employing the hypersurfaces $\Sigma_a$ as barriers. Here the asymptotics for $\Sigma_a$ established in Section~\ref{barrier section} play a decisive role. A standard argument (due to Merle and Zaag) shows that either positive or neutral modes dominate as $\tau \to -\infty$. The latter possibility is ruled out using the fact that $\bar M_\tau$ is noncompact---this forces $u(z,\theta,\tau)$ to be monotone in $z$, whereas no nonzero linear combination of zero modes has this property. The estimate \eqref{asymptotic intro} follows. 

The proof of \eqref{asymptotic intro} is broadly similar to the spectral analysis of convex ancient solutions to the mean curvature flow carried out in \cite{Brendle_Choi_a, Brendle_Choi_b} and \cite{ADS, ADS_uniqueness}. In these works sufficient decay of errors could be proven by using Huisken's monotonicity formula to derive a powerful reverse-Poincar\'{e} inequality for the profile function $u$. Our approach using barriers is instead informed by Brendle's work on ancient $\kappa$-solutions to the 3-dimensional Ricci flow \cite{Brendle}. We note however that error terms involving the Hessian $|\nabla^2 u|$ do not arise when considering the mean curvature flow or Ricci flow, as these equations are quasilinear, whereas for fully nonlinear equations such higher-order errors are inevitable. 

In our setting we are able to control these higher-order errors, roughly speaking, as follows. Using a gradient bound furnished by the barriers, we can appeal to the interior $C^2$-estimate for radial graph solutions due to Brendle--Huisken \cite{Brendle_Huisken}, followed by Krylov's $C^{2,\alpha}$-estimate for convex parabolic PDE \cite{Krylov82}, and then use Schauder estimates to get uniform bounds on higher derivatives. Next we use interpolation inequalities and interior estimates for linear parabolic equations to obtain an improved bound for the supremum of $|u| + |\nabla u| + |\nabla^2 u|$. Crucially, this bound is only slightly worse than linear in the $L^2$-norm of $u$ with respect to the measure $e^{-z^2/4a}dzd\theta$ (see Lemma~\ref{error estimates}). This is sufficient to carry out the Merle--Zaag argument. Note that all of the estimates just described hold in a subset of space which grows at a definite rate as $\tau \to -\infty$. 

With the sharp asymptotic \eqref{asymptotic intro} in hand, the remainder of the proof of Theorem~\ref{main} is quite similar to its counterpart for solutions of mean curvature flow \cite{Brendle_Choi_a, Brendle_Choi_b}. In Section~\ref{neck improvement section} we prove the Neck Improvement Theorem for $G$-flows (cf. \cite[Theorem~4.4]{Brendle_Choi_a} and \cite[Theorem~4.4]{Brendle_Choi_b}). Using this, we show that $M_t$ is rotationally symmetric (Section~\ref{rotational symmetry section}) and is therefore a translating soliton (Section~\ref{translating section}). This completes the proof of Theorem~\ref{main}. 

The arguments in Section~\ref{rotational symmetry section} and Section~\ref{translating section} make use of a differential Harnack inequality for noncompact convex solutions to fully nonlinear flows, which was recently proven in \cite{Lynch_Harnack}. For solutions of mean curvature flow this inequality is due to Hamilton \cite{Hamilton_Harnack_MCF}, and for compact convex solutions to fully nonlinear flows it is due to Andrews \cite{Andrews_Harnack}. The Harnack inequality lets us show that near its `tip', our solution looks like a bowl soliton. This step also relies on work of Bourni--Langford \cite{Bourni--Langford}, who showed that every convex, noncollapsing, uniformly two-convex $G$-translator is the bowl soliton.

 
\section{Preliminaries}\label{preliminaries section}

In this section we collect auxhiliary results which will be needed in the proof of Theorem~\ref{main}. 

We may view $\gamma$ as a symmetric function defined in $\Gamma$, or as an $O(n)$-invariant function on the space of symmetric matrices with eigenvalues in $\Gamma$. We alternate between these viewpoints. We write $\dot \gamma^i(\lambda)$ and $\ddot \gamma^{ij}(\lambda)$ for derivatives with respect to eigenvalues, so that 
    \[\frac{d}{ds}\bigg|_{s=0} \gamma(\lambda + s \mu) = \dot \gamma^i(\lambda)\mu_i, \qquad \frac{d}{ds}\bigg|_{s=0} \dot \gamma^i(\lambda + s\mu) = \ddot \gamma^{ij}(\lambda)\mu_j,\]
and write $\dot \gamma^{ij}(A)$ and $\ddot \gamma^{ij,kl}(A)$ for derivatives with respect to matrix entries, so that 
    \[\frac{d}{ds}\bigg|_{s=0} \gamma(A + s B) = \dot \gamma^{ij}(A)B_{ij}, \qquad \frac{d}{ds}\bigg|_{s=0} \dot \gamma^{ij}(A + sB) = \ddot \gamma^{ij,kl}(A)B_{kl}.\]
If $A$ is a diagonal matrix with entries $\lambda$, $\dot\gamma^{ij}(A)$ is also diagonal, with entries $\dot\gamma^i(\lambda)$.

Recall that we assume $\gamma$ is convex or concave. We also assume 
    \begin{equation}\label{inverse-concavity}
        \text{the functions } \gamma|_{\mathbb{R}^n_+} \text{ and } \gamma|_{\{0\}\times\mathbb{R}^{n-1}_+} \text{ are strictly inverse-concave.}
    \end{equation}
That is, 
    \[\lambda \mapsto -\gamma(\lambda_1^{-1},\dots,\lambda_n^{-1})\]
is strictly concave for $\lambda \in \mathbb{R}_+^n$, and 
    \[\mu \mapsto -\gamma(0, \mu_1^{-1},\dots,\mu_{n-1}^{-1})\]
is strictly concave for $\mu \in \mathbb{R}_+^{n-1}$. When $\gamma$ is convex the hypotheses \eqref{inverse-concavity} are redundant.

We often use the notation
    \[F(\lambda,\mu) = \gamma(\lambda, \mu, \dots, \mu).\]

\subsection{Uniform ellipticity} Let $M_t$, $t \in I$, denote a $G$-flow. We say that $\gamma$ is uniformly elliptic on $M_t$ if there is a constant $\delta > 0$, independent of time, such that 
    \begin{equation}\label{unif parabolicity}
        \inf_{M_t} \dist \left(\tfrac{\lambda}{|\lambda|}, \partial \Gamma\right) \geq \delta.
    \end{equation}
This implies that for any symmetric $(0,2)$-tensor $B$ we have
    \[C^{-1} |B| \leq \dot \gamma^{ij}(A)B_{ij} \leq C |B|\]
and 
    \[|\ddot\gamma^{ij,kl}(A)B_{ij}B_{kl}| \leq CG^{-1}|B|^2,\]
where $C$ is a positive constant depending only on $n$, $\gamma$ and $\delta$.

Recall that we assume $\Gamma$ contains the cones $\mathbb{R}_+^n$ and $\{0\} \times \mathbb{R}_+^{n-1}$. Given these hypotheses, weak convexity and uniform two-convexity imply uniform ellipticity. That is, if $\lambda_1 \geq 0$ and $\lambda_1 + \lambda_2 \geq \beta H$ at every point on a $G$-flow, then there is a constant $\delta = \delta(n,\gamma,\beta)$ such that \eqref{unif parabolicity} holds for $t \in I$. 

\begin{lemma}\label{lemma uniform parabolicity}
Let $M_t$ be a convex, uniformly two-convex $G$-flow. Then $\gamma$ is uniformly elliptic on $M_t$.
\end{lemma}
\begin{proof}
If not, let $(x_k,t_k)$ be a sequence such that $\lambda(x_k,t_k)/|\lambda(x_k,t_k)| \to \partial \Gamma$. Given that $\Gamma$ contains $\mathbb{R}_+^n$ and $\{0\} \times \mathbb{R}_+^{n-1}$, it must be the case that both $\lambda_1(x_k,t_k)$ and $\lambda_2(x_k,t_k)$ tend to zero. But convexity implies
    \[\frac{1}{n} H(x_k,t_k)^2 \leq |\lambda(x_k,t_k)|^2 \leq H(x_k,t_k)^2,\]
so we have $\lambda_1(x_k,t_k) + \lambda_2(x_k,t_k)/H(x_k,t_k) \to 0$, which contradicts uniform two-convexity. 
\end{proof}

\subsection{Compactness of noncollapsing convex ancient solutions} The following interior curvature estimate was proven in \cite[Proposition 5.1]{Brendle_Huisken} for the speed $\gamma(\lambda) = (\sum_{i<j}(\lambda_i+\lambda_j)^{-1})^{-1}$. However the same proof works for any $\gamma$ satisfying our hypotheses (see \cite{Lynch_22_b} for further discussion).

\begin{lemma}\label{interior curvature}
Let $M_t = \partial\Omega_t, t \in [-T, 0],$ be a convex $G$-flow which is $\delta$-uniformly parabolic. If $B(0,r)$ is contained in $\Omega_0$, then
    \[\sup_{B(0,R) \times [- T/2 , 0]}   |A| \leq C r^{-3}  R^3 \Big(r^{-1} + T^{-1/2}\Big),\]
where $C$ depends only on $n$, $\gamma$, $\delta$ and $R/r$. 
\end{lemma}  

As a consequence of Lemma \ref{interior curvature} and the avoidance principle, noncollapsing $G$-flows satisfy the following curvature estimate (cf. \cite{Haslhofer--Kleiner_MC}). For the proof we refer to \cite[Theorem~4.17]{lynch2020convexity} (note however that the argument there simplifies considerably when $M_t$ is convex). 

Here and throughout the paper we only consider ancient $G$-flows which exist for $t \in (-\infty,0]$. This is purely for convenience, and can always be arranged by shifting the time variable.

\begin{proposition}\label{lower curvature}
Let $M_t = \partial \Omega_t$, $t \in (-\infty,0]$, be a convex ancient $G$-flow which is $\delta$-uniformly parabolic and $\alpha$-noncollapsing. Suppose $x_0 \in M_{t_0}$. We then have 
	\[C^{-1} G(x_0,t_0) \leq G(x,t) \leq C G(x_0,t_0),\] 
for $(x,t) \in B(x_0, LG(x_0,t_0)^{-1})\times[t_0-L^2G(x_0,t_0)^{-2},t_0]$, where the constant $C$ depends only on $n$, $\gamma$, $\delta$, $\alpha$ and $L$. 
\end{proposition}

As a consequence of the curvature estimate we have the following compactness theorem for sequences of convex ancient $G$-flows which are uniformly parabolic and noncollapsing.
	
\begin{proposition}\label{compactness}
Let $M_t^k$, $t\in(-\infty,0]$, be a sequence of convex ancient $G$-flows. Suppose each $M_t^k$ is $\delta$-uniformly parabolic and $\alpha$-noncollapsing, where $\delta$ and $\alpha$ are independent of $k$. Suppose in addition that there exists an $R < \infty$ such that $M_0^k \cap B(0,R)$ is nonempty for every $k$ and 
    \[\liminf_{k \to \infty} \sup_{M_0^k \cap B(0,R)} (G_k + G_k^{-1})>0.\] 
Then there is a convex ancient $G$-flow $M_t$, $t\in(-\infty,0]$, such that after passing to a subsequence
    \[M_t^k \to M_t \;\; \text{in} \;\; C^\infty_{\loc}(\mathbb{R}^{n+1}\times(-\infty,0]).\]
\end{proposition}
\begin{proof}
The assumptions imply that, for each $k$, there is a point $p_k$ satisfying $|p_k| \leq R$ and a ball  
    \[B(p_k,r) \subset \Omega_0^k,\]
where $r>0$ is independent of $k$. We may therefore apply the interior curvature bound of Lemma \ref{interior curvature} to conclude that, in every compact subset of spacetime, $|A_k|$ is bounded independently of $k$. Moreover, by Proposition~\ref{lower curvature}, in every compact subset of spacetime $G_k$ is bounded from below by a positive constant which is independent of $k$. 
    
With the uniform upper bound for $|A_k|$, we can deduce uniform $C^{2,\alpha}$-estimates in compact subsets of spacetime by expressing $M_t^k$ locally as a graph and (since $\gamma$ is convex or concave) appealing to the H\"{o}lder estimate for for concave/convex parabolic PDE proven in \cite{Krylov82}. Using the upper bound for $|A_k|$ and the uniform lower bound for $G_k$, we can bootstrap using the Schauder estimates to get uniform derivative bounds up to any order in compact subsets of spacetime. Since the hypersurfaces $M_t^k$ are radial graphs over $\partial B(p_k, r)$, a standard application of the Arzel\'{a}--Ascoli Theorem shows that the flows subconverge in $C^\infty_{\loc}$.
\end{proof}

\subsection{Neck-cap decomposition} We have the following consequence of the splitting theorem in \cite{Lynch_22_b}.

\begin{lemma}\label{splitting}
Let $M_t$, $t \in (-\infty, 0]$, be a convex, noncollapsing, uniformly two-convex ancient $G$-flow. If there is a time $t \leq 0$ and a point in $M_t$ at which $\lambda_1 = 0$, then $M_t$ is a self-similarly shrinking $\mathbb{R} \times S^{n-1}$.
\end{lemma}
\begin{proof}
Suppose $\lambda_1$ vanishes at some point in spacetime. By the splitting theorem \cite[Proposition~A.2]{Lynch_22_b}, 
    \[M_t = \mathbb{R} \times M_t^{\perp}\]
for $t \leq 0$, where $M_t^\perp$ is a family of hypersurfaces in $\mathbb{R}^n$. If $n-1 \geq 2$ then $M_t^\perp$ is uniformly convex, and hence compact by \cite{HamiltonPinched}, so Theorem~1.7 in \cite{LangfordLynch} (see also \cite{Risa--Sinestrari}) implies that $M_t^\perp$ is a self-similarly shrinking sphere. If $n-1 = 1$, since we are assuming noncollapsing, the classification result in \cite{BLT} implies that $M_t^\perp$ is a self-similarly shrinking circle. This completes the proof. 
\end{proof}

Let $M_t$ be a $G$-flow. For $\bar x \in M_{\bar t}$ we define 
\begin{align*}
    \hat{\mathcal{P}}(\bar{x},\bar{t},L,\theta):=\Big\lbrace (x,t) : \, &x\in B_{g(\bar t)}(\bar x(t), \gamma(0,1,\dots,1)LG(\Bar{x},\Bar{t})^{-1}),\\
    &\; t\in[\Bar{t}-\gamma(0,1,\dots,1)^2 \theta G(\Bar{x},\Bar{t})^{-2}, \bar{t}]\Big\rbrace,
\end{align*}
where $\bar x(t)$ is the unique curve satisfying 
    \[\frac{d}{dt} \bar x(t) = - G(\bar x(t), t) \nu(\bar x(t), t), \qquad \bar x(\bar t) = \bar x.\]

\begin{definition}
A point $(\bar{x},\bar{t})$ is said to lie at the center of an $\varepsilon$-neck if, after shifting $(\bar x, \bar t)$ to $(0,0)$, parabolically rescaling to arrange $G(0, 0) = (\gamma(0,1,\dots,1)/2)^{1/2}$, and applying a rigid motion, the region $\hat{\mathcal{P}}(\bar x, \bar t, 100,100)$ is a normal graph over $\mathbb{R} \times S^{n-1}((2\gamma(0,1,\dots,1)(1-t))^{1/2})$, and the $C^{10}$-norm of the height function is at most $\varepsilon$. 
\end{definition}

Let $M_t$ be a noncompact, strictly convex ancient $G$-flow. Then $M_t$ has a single end at each time. The following result asserts that if $M_t$ is also noncollapsing and uniformly two-convex, then this end consists of necks. For solutions of the mean curvature flow this was proven in \cite[Proposition~3.4]{Haslhofer--Kleiner_Surgery}. 

\begin{lemma}\label{neck-cap}
Let $M_t$, $t \in (-\infty, 0]$, be a convex ancient $G$-flow which is $\alpha$-noncollapsing and $\beta$-uniformly two-convex. For every $\varepsilon_1$, $\varepsilon_2 > 0$ there is a positive constant $R = R(n,\gamma,\alpha,\beta,\varepsilon_1,\varepsilon_2)$ with the following property. If $x_1$ and $x_2$ are points in $M_0$, neither of which lies at the center of an $\varepsilon_1$-neck, then one of the following holds:
\begin{enumerate}
    \item $\max\{G(x_1, 0), G(x_2,0)\}|x_1 - x_2| \leq R$.
    \item $M_t$ is compact and bounds a region of diameter at most $(1+\varepsilon_2)|x_1 - x_2|$.
\end{enumerate}
\end{lemma}
\begin{proof}
Suppose the claim is false for some $\varepsilon_1$ and $\varepsilon_2$. Then there is a sequence of convex ancient $G$-flows $M_t^k = \partial\Omega_t^k$, $t \in (-\infty,0]$, which are $\alpha$-noncollapsing and $\beta$-uniformly two-convex, and have the following property. There are points $x_1^k$ and $x_2^k$ in $M_0^k$, neither of which lies at the center of an $\varepsilon_1$-neck, such that
    \[\max\{G(x_1, 0), G(x_2,0)\}|x_1 - x_2| \to \infty\]
and 
    \[\diam(\Omega_0^k) \geq (1+\varepsilon_2)|x_1 - x_2|.\]
After performing a parabolic rescaling, we may assume that $|x_1^k - x_2^k| = 1$. We then have $\max\{G(x_1^k, 0), G(x_2^k,0)\} \to \infty$. Consequently, after passing to a subsequence, the sets $\bar \Omega_0^k$ Hausdorff-converge to a closed convex set of dimension at most $n$ (this follows from Lemma~\ref{interior curvature}). Let us denote this set by $K$.  

We claim that the dimension of $K$ is 1. If $K$ has dimension at least 2 then there is a `cross' in $K$ consisting of intervals $\overline{y_1y_2}$ and $\overline{y_3y_4}$ which intersect transversally at a point $y_0$. Approximate the points $y_i$ by sequences $y_i^k$ in $M_0^k$, and let $T_k$ denote the convex hull of the set $\{y_i^k\}$. We have $T_k \subset \bar \Omega_0^k$. After performing a parabolic rescaling to arrange $G(y_0^k, 0) = 1$, we may apply Proposition~\ref{compactness} to extract a limiting $G$-flow which is uniformly two-convex for $t \in (-\infty, 0]$. But after rescaling the sets $\partial T_k$ converge to a 2-plane which makes interior contact with the limiting flow at time $t = 0$. This contradicts uniform two-convexity. So we conclude that $K$ has dimension at most 1. Moreover, since $|x_1^k - x_2^k| = 1$, $K$ is not a point. Therefore, the dimension of $K$ is 1. 

If $x_1^k$ and $x_2^k$ converge to the endpoints of $K$ then for sufficiently large $k$ we have $\diam(\Omega_0^k) < 1+\varepsilon_2$, which is impossible by hypothesis. Therefore, we may assume without loss of generality that $x_2^k$ converges to an interior point of $K$. Let us now parabolically rescale so that $G(x_2^k, 0) = 1$ and extract a limiting $G$-flow using Proposition~\ref{compactness}. Arguing as above, we see that the limit makes interior contact with a line at time $t = 0$. Therefore, by Lemma~\ref{splitting}, $x_2^k$ lies at the center of an $\varepsilon_1$-neck for all large $k$. This is a contradiction. 
\end{proof}

Lemma~\ref{neck-cap} implies that we have bounded curvature on bounded time intervals. 

\begin{lemma}\label{curvature bound}
Let $M_t = \partial \Omega_t$, $t \in (-\infty, 0]$, be a noncompact, convex ancient $G$-flow which is noncollapsing and uniformly two-convex. Then for each $T < \infty$, $\sup_{M_t} G$ is bounded from above independently of $t \in [-T, 0]$. 
\end{lemma}
\begin{proof}
Let $B$ be an open ball in $\Omega_0$. Then $B$ is contained in $\Omega_t$ for all $t \leq 0$. Suppose there is a sequence of times $t_k \in [-T, 0]$ and points $x_k \in M_{t_k}$ such that $G(x_k, t_k) \to \infty$. Then $|x_k| \to \infty$, so for any $\varepsilon > 0$ we may apply Lemma~\ref{neck-cap} to conclude that $(x_k,t_k)$ lies at the center of an $\varepsilon$-neck for all large $k$. Since $G(x_k,t_k) \to \infty$, the radius of the neck must tend to zero as $k \to \infty$. On the other hand, by convexity and noncompactness, $\Omega_{t_k}$ contains a ray emanating from the center of $B$. As we slide $B$ along this ray it must remain inside $\Omega_{t_k}$ and eventually pass through the neck containing $x_k$. This is impossible if the neck is too thin, so we have reached a contradiction. 
\end{proof}

\subsection{Differential Harnack inequality} Every strictly convex ancient $G$-flow which is noncollapsing and uniformly two-convex satisfies the following differential Harnack inequality. In addition, if equality is ever attained then the $G$-flow is a translating soliton. These results require the inverse-concavity hypotheses \eqref{inverse-concavity}. 

\begin{proposition}\label{differential harnack}
Let $M_t$ be a strictly convex ancient $G$-flow which is uniformly two-convex and noncollapsing. We then have
    \begin{equation}\label{Harnack V}
    \partial_t G + 2\langle \nabla G, V\rangle + A(V,V) \geq 0
    \end{equation}
at every point $x \in M_t$ and for every $V \in T_x M_t$. Equivalently,
    \begin{equation}\label{Harnack scalar}
    \partial_t G - A^{-1}(\nabla G, \nabla G) \geq 0
    \end{equation}
for every $x \in M_t$. The last inequality is strict at every point in spacetime unless $M_t$ is a translating soliton. 
\end{proposition}
\begin{proof}
We first observe that \eqref{Harnack V} is equivalent to \eqref{Harnack scalar}; indeed, we have 
    \[2\langle \nabla G, V\rangle + A(V,V) \geq - A^{-1}(\nabla G, \nabla G)\]
with equality exactly when $V = - A^{-1}(\nabla G)$. When $M_t$ is compact, \eqref{Harnack scalar} was proven in \cite{Andrews_Harnack}. When $M_t$ is noncompact we have bounded curvature on bounded time intervals by Lemma~\ref{curvature bound}. Therefore, Theorem~1.2 in \cite{Lynch_Harnack} gives the inequality \eqref{Harnack V}. If equality is attained in \eqref{Harnack scalar} then $M_t$ is a translating soliton by \cite[Corollary~1.3]{Lynch_Harnack}. 
\end{proof}

As a consequence of the Harnack inequality we obtain a curvature bound which is valid for all times.

\begin{lemma}
Let $M_t$ be a strictly convex ancient $G$-flow which is uniformly two-convex and noncollapsing. Then $\sup_{M_t} G$ is bounded from above independently of $t$. 
\end{lemma}
\begin{proof}
If $\lambda_1$ vanishes somewhere then $M_t$ is a shrinking cylinder by Lemma~\ref{splitting}, in which case the claim is immediate, so assume $A > 0$. 

We have $\sup_{M_0} G < \infty$. This is immediate if $M_0$ is compact, and follows from Lemma~\ref{curvature bound} if $M_t$ is noncompact. On the other hand, the Harnack inequality says that $G$ is nondecreasing at each point in spacetime. With this the claim is proven.
\end{proof}

\subsection{Expansions for hypersurfaces close to a cylinder} Let 
    \[\Sigma = \mathbb{R}\times S^{n-1}(r)\]
and consider a smooth function $u:\Sigma\to\mathbb{R}$. We study the hypersurface
    \[M = \graph u = \{x + u(x)\nu_\Sigma(x) : x \in \Sigma\},\]
where $\nu_\Sigma$ is the outward unit normal to $\Sigma$. We assume $M$ is close to $\Sigma$ in the sense that
    \[\sup_{\Sigma} \, (r^{-1}|u| + |\nabla^\Sigma u| + r|\nabla^2_\Sigma u|)\]
is small, and compute expansions for certain geometric quantities on $M$ in terms of their counterparts on $\Sigma$. We write $a = b + O(c)$ when $|a - b| \leq K|c|$ for a positive constant $K$ depending only on $n$ and $\gamma$.

Let $g$ and $A$ denote the induced metric and second fundamental form of $M$. Let $g_\Sigma$ and $A_\Sigma$ denote the induced metric and second fundamental form of $\Sigma$. We will not distinguish between tensors on $M$ and their pullbacks to $\Sigma$ via the map $x \mapsto x + u(x)\nu_\Sigma(x)$.

The proof of the following lemma is straightforward, so we omit it. 

\begin{lemma}\label{Lemma: expansions}
We have the expansions
\begin{align*}
    g &= g_\Sigma + 2 u A_\Sigma + O(r^{-2}|u|^2 + |\nabla^\Sigma u|^2)\\
    \nu &= \nu_\Sigma - \nabla^\Sigma u + O(r^{-2}|u|^2 + |\nabla^\Sigma u|^2)\\
    A &= A_\Sigma  - \nabla^2_\Sigma u + u A^2_\Sigma  + O(r^{-3}|u|^2 + r^{-1}|\nabla^\Sigma u|^2).
\end{align*}
\end{lemma}

From now on we work with a fixed orthonormal frame of eigenvectors for $A_\Sigma$ with respect to the metric $g_\Sigma$, denoted $\{e_i\}$. Moreover, we assume $e_1$ is orthogonal to $S^{n-1}$. This frame will typically fail to be orthonormal with respect to $g$. The following lemma establishes that we can, in a controlled fashion, smoothly transform $\{e_i\}$ to a frame which is $g$-orthonormal. We will find it convenient to use this transformed frame when we expand $G$ in Lemma~\ref{expansion for G} below. 

\begin{lemma}\label{frame construction}
There is a smooth field of endomorphisms $\alpha: T\Sigma \to T\Sigma$ such that $\{\alpha(e_i)\}$ is orthonormal with respect to $g$, and
    \[\alpha = I - u A_\Sigma + O(r^{-2}|u|^2 + |\nabla^\Sigma u|^2),\]
where $I : T\Sigma \to T\Sigma$ is the identity. 
\end{lemma}

\begin{proof}
With respect to $\{e_i\}$, $g$ has entries
    \[g_{ij} = \delta_{ij} + 2 u (A_\Sigma)_{ij} + u^2 (A^2_\Sigma)_{ij} + \nabla^\Sigma_i u \nabla^\Sigma_j u.\]
For $s \in [0,1]$ we define 
    \[h_{ij}(s) := \delta_{ij} + 2 u (A_\Sigma)_{ij} + u^2 (A^2_\Sigma)_{ij} + s \nabla^\Sigma_i u \nabla^\Sigma_j u.\]
We then define $\alpha_i^k(s)$ to be the solution of the initial value problem
    \[h_{jk}(s) \frac{d}{ds} \alpha_i^k(s) = -\frac{1}{2}\alpha_i^m(s) \nabla^\Sigma_m u \nabla^\Sigma_j u, \qquad (\delta_k^i + u (A_\Sigma)_k^i)\alpha_j^k(0) = \delta^i_j.\]
This ensures that 
    \[ \frac{d}{ds}(h_{kl}(s)\alpha_i^k(s)\alpha_j^l(s)) = 0  \qquad \text{and} \qquad h_{kl}(0)\alpha_i^k(0)\alpha_j^l(0) = \delta_{ij},\]
and hence $\alpha_i^k(1)$ a smooth field of matrices on $\Sigma$ satisfying 
    \[\delta_{ij} = h_{kl}(1)\alpha_i^k(1) \alpha_j^l(1) = g_{kl}\alpha_i^k(1) \alpha_j^l(1).\]
That is, the frame $\{\alpha(1)(e_i)\}$ is $g$-orthonormal. Straightforward computations show that 
    \[\alpha_i^k(1) = \alpha_i^k(0) + O(|\nabla^\Sigma u|^2) = \delta_i^k - u (A_\Sigma)_i^k + O(r^{-2}|u|^2 + |\nabla^\Sigma u|^2).\]
Therefore, $\alpha(1)$ is a smooth field of endomorphisms with the properties claimed. 
\end{proof}

Let $\lambda$ and $\lambda_\Sigma$ denote the principal curvatures of $M$ and $\Sigma$, respectively. Recall that we assume $\Gamma$ contains $\{0\}\times\mathbb{R}_+^{n-1}$, and hence $\lambda_\Sigma \in \Gamma$. We define $G_\Sigma(x) = \gamma(\lambda_\Sigma(x))$. We may also write $G_\Sigma = \gamma(A_\Sigma)$, where on the right we evaluate $\gamma$ at the matrix with entries $(A_\Sigma)_{ij} = A_\Sigma(e_i, e_j)$. Suppose $\lambda$ lies in $\Gamma$ and define $G(x) = \gamma(\lambda(x))$. We also define $\tilde A_{ij} = A_{kl}\alpha^k_i\alpha^l_j$, so that $\tilde A_{ij}$ is the matrix representation of $A$ with respect to the $g$-orthonormal frame $\{\alpha(e_i)\}$. We may then express $G = \gamma(\tilde A)$. 

\begin{lemma}\label{expansion for G}
We have
    \[G = {G_\Sigma} - \dot \gamma^{ij}(A_\Sigma) (\nabla^2_\Sigma u)_{ij} - \dot \gamma^{ij}(A_{\Sigma})(A_\Sigma^2)_{ij} u  + O(r^{-3}|u|^2 + r^{-1}|\nabla^\Sigma u|^2 + r|\nabla^2_\Sigma u|^2).\]
This may also be written as 
    \begin{align*}G &= G_{\Sigma} - \dot\gamma^1(0,1,\dots,1) \frac{\partial^2 u}{\partial z^2} - \frac{1}{n-1}\frac{\gamma(0,1,\dots,1)}{r^2} r^2 \Delta_{S^{n-1}_r} u - \frac{\gamma(0,1,\dots,1)}{r^2} u\\
    &+ O(r^{-3}|u|^2 + r^{-1}|\nabla^\Sigma u|^2 + r|\nabla^2_\Sigma u|^2),
    \end{align*}
where $z$ denotes a Euclidean coordinate on the $x_1$-axis.
\end{lemma}
\begin{proof}
The function
    \[h(t) := \gamma(A_\Sigma  + t(\tilde A - A_\Sigma))\]
is well defined and smooth in $t$. Performing a Taylor expansion at $t = 0$ yields
    \[|h(t) - G_\Sigma - \dot\gamma^{ij}(A_\Sigma)(\tilde A - A_\Sigma)_{ij}| \leq \frac{1}{2}\bigg(\sup_{t \in [0,1]} |h''(t)|\bigg) t^2\]
for $t \in [0, 1]$. Evaluating at $t = 1$, and estimating the right-hand side, we obtain 
    \[|G - G_\Sigma - \dot\gamma^{ij}(A_\Sigma)(\tilde A - A_\Sigma)_{ij}| \leq O(r|\tilde A-A|^2).\]
Now we use Lemma~\ref{Lemma: expansions} and Lemma~\ref{frame construction} to deduce that
    \[\tilde A - A = -\nabla^2_\Sigma u - A^2_\Sigma u+
    O(r^{-3}|u|^2 + r^{-1}|\nabla^\Sigma u|^2 + r|\nabla^2_\Sigma u|^2),\]
and so obtain
    \begin{align*}
    |G - G_\Sigma + \dot \gamma^{ij}(A_\Sigma)((\nabla^2_\Sigma u)_{ij} + (A^2_\Sigma)_{ij} u)| =O(r^{-3}|u|^2 + r^{-1}|\nabla^\Sigma u|^2 + r|\nabla^2_\Sigma u|^2).
    \end{align*}
This proves the first claim. 

Since $A_{\Sigma}$ is diagonal with respect to $\{e_i\}$, $\dot\gamma^{ij}(A_\Sigma)$ is also diagonal, and we have
    \begin{align*}
    \dot \gamma^{ij}(A_\Sigma) (\nabla^2_\Sigma u)_{ij} &= \dot\gamma^1(0,r^{-1}, \dots, r^{-1}) \frac{\partial^2 u}{\partial z^2}\\
    &+ \frac{1}{n-1} \sum_{i=2}^n \dot\gamma^i(0,r^{-1}, \dots, r^{-1}) \Delta_{S^{n-1}_r} u + O(|\nabla^\Sigma u|^2).
    \end{align*}
Here we have also used the symmetry of $\gamma$ in its arguments to express
    \[\dot\gamma^j(0,r^{-1}, \dots, r^{-1}) = \frac{1}{n-1} \sum_{i=2}^n \dot\gamma^i(0,r^{-1}, \dots, r^{-1})\]
for $j \geq 2$. Since $\gamma$ is homogeneous of degree one we have 
    \[\dot\gamma^i(0,r^{-1}, \dots, r^{-1}) = \dot\gamma^i(0,1, \dots, 1),\]
and hence
    \[\sum_{i=2}^n \dot\gamma^i(0,r^{-1}, \dots, r^{-1}) = \gamma(0,1,\dots,1).\]
The same reasoning yields
    \[\dot \gamma^{ij}(A_\Sigma)(A^2_\Sigma)_{ij} u = r^{-2} \sum_{i=2}^n\dot\gamma^i(0,r^{-1},\dots,r^{-1}) u = r^{-2}\gamma(0,1,\dots,1) u.\]
The second claim follows.
\end{proof}

Let $S$ be a symmetric $(0,2)$-tensor on $M$. After fixing a $g$-orthonormal frame, so that $A$ and $S$ can be identified with symmetric matrices, we define
\[\trace_\gamma(S) := \frac{d}{dt}\bigg|_{t=0} \gamma(A + tS).\]
The right-hand side is independent of the chosen frame, so $\trace_\gamma(S)$ is a well defined smooth function on $M$.  

\begin{lemma}\label{trace expansion}
Given a symmetric $(0,2)$-tensor $S$ on $M$, we have 
    \[\trace_\gamma(S) = \dot \gamma^{ij} (A_\Sigma)S_{ij} + O(r^{-1}|u| + |\nabla^\Sigma u| + r|\nabla^2_\Sigma u|)|S|.\]
In particular, when $S = \nabla^2 f$ for some smooth function $f$ we have
    \begin{align*}
    \trace_\gamma(\nabla^2 f) &=  \dot\gamma^1(0,1,\dots,1) \frac{\partial^2 f}{\partial z^2} + \frac{1}{n-1}\frac{\gamma(0,1,\dots,1)}{r^2} r^2 \Delta_{S^{n-1}_r} f\\
    &+ O(r^{-1}|u| + |\nabla^\Sigma u| + r|\nabla^2_\Sigma u|)|\nabla^2 f| + O(|\nabla^\Sigma u|)|\nabla f|
    \end{align*}
where $z$ denotes a Euclidean coordinate on the $x_1$-axis.
\end{lemma}
\begin{proof}
We let $\tilde S_{ij} = S_{kl}\alpha^k_i\alpha_j^l$, so that $\tilde S_{ij}$ is the matrix representation of $S$ with respect to the $g$-orthonormal frame $\{\alpha(e_i)\}$. We may then compute
\[\trace_\gamma(S) = \frac{d}{dt}\bigg|_{t=0} \gamma(\tilde A + t\tilde S)=\dot \gamma^{ij}(\tilde A)\tilde S_{ij}.\]
We define an auxiliary function 
    \[h(t) := \dot\gamma^{ij}(A_\Sigma + t(\tilde A - A_\Sigma))\tilde S_{ij},\]
which is smooth in $t$. We have
    \[|\trace_\gamma(S) - \dot\gamma^{ij}(A_\Sigma)\tilde S_{ij}| = |h(1) - h(0)| \leq \sup_{t\in [0,1]} |h'(t)| = O(r|\tilde A - A|)|\tilde S|,\]
and hence, by Lemma~\ref{frame construction},
    \[|\trace_\gamma(S) - \dot \gamma^{ij}(A_\Sigma)S_{ij}| \leq O(r^{-1}|u| + |\nabla^\Sigma u| + r|\nabla^2_\Sigma u|)|S|.\]
The proves the first claim. The second follows as in the proof of Lemma~\ref{expansion for G}.
\end{proof}

\subsection{Asymptotic cones} Let $\Omega$ be an open convex subset of $\mathbb{R}^{n+1}$. We define the asymptotic cone $C(\Omega)$ of $\Omega$ to be the Hausdorff-limit of the sequence $k^{-1} \bar \Omega$. When $\Omega$ is bounded, $C(\Omega) = \{0\}$. When $\Omega$ is unbounded, $C(\Omega)$ is a closed cone. If we translate $\Omega$ so that it contains the origin, then $C(\Omega)$ is precisely the union of all the rays emanating from the origin which are contained in $\bar \Omega$. 

\begin{lemma}\label{constant cone}
Let $M_t = \partial\Omega_t$, $t \in (-\infty, 0]$, be a noncompact, convex ancient $G$-flow. The asymptotic cone $C(\Omega_t)$ is independent of $t$. 
\end{lemma}
\begin{proof}
Since the hypersurfaces $M_t$ evolve inwards, we have $\Omega_t \subset \Omega_s$ whenever $s < t$. It follows that $C(\Omega_t) \subset C(\Omega_s)$. In particular, $C(\Omega_0) \subset C(\Omega_t)$ for $t \leq 0$.

Now consider an arbitrary time $\bar t < 0$, and an arbitrary ray $R$ contained in $C(\Omega_{\bar t})$. We may assume without loss of generality that $0 \in \Omega_0$. We then have $C(\Omega_{\bar t}) \subset \bar \Omega_{\bar t}$, and hence $R \subset \Omega_{\bar t}$. Let $\tilde t$ denote the supremum of all the times $t$ for which $R \subset \Omega_{t}$. Let $B$ be an open ball with center at $0$ such that $B \subset \Omega_0$. Sliding $B$ along $R$ and appealing to the avoidance principle, we see that unless $\tilde t = 0$, there is a short time period following $\tilde t$ in which $R \subset \Omega_t$. So $\tilde t = 0$, and hence $R \subset \bar \Omega_0$. It follows that $R \subset C(\Omega_0)$. Since $R$ was arbitrary, we conclude that $C(\Omega_{\bar t}) \subset \Omega_0$.

We have shown that 
    \[C(\Omega_t) \subset C(\Omega_0) \subset C(\Omega_t)\]
for every $t \leq 0$. This proves the claim. 
\end{proof}


 \section{The parabolic blow-down limit}\label{blow-down section}

Let $M_t = \partial \Omega_t$, $t \in (-\infty,0]$, be a convex ancient $G$-flow which is noncollapsing and uniformly two-convex. In this section we study the possible parabolic blow-down limits of $M_t$. We show that the only possibilities are self-similarly shrinking sphere and cylinder solutions. 

We first prove that the region $\Omega_t$ contains a shrinking ball with radius on the order of $(-t)^{1/2}$ (cf. \cite[Theorem 2.2]{Wang} and \cite[Proposition 4.3]{BLL}). Without loss of generality we may assume $0 \in M_0$.

\begin{lemma}\label{Paraboloid}
There is a constant $\eta > 0$ such that $B(0, \eta (-t)^{1/2}) \subset \Omega_t$ whenever $-t$ is sufficiently large. 
\end{lemma}
\begin{proof}
Supposing the claim is false, we can find a sequence of times $t_k \to -\infty$ such that the rescaled flows 
    \[\tilde M_t^k := (-t_k)^{-1/2} M_{(-t_k)t}\]
satisfy $\dist(0, \tilde M_{-1}^k) \to 0$. The function $f_k(t):=\dist(0,\tilde M_t^k)$ is locally Lipschitz, and for almost every $t$ we have
    \[\frac{df_k}{dt}(t)\leq -\sup\{\tilde G_k(x,t):|x|=f_k(t)\}.\]
Therefore, since 
    \[ -\int^{0}_{-1} \frac{df_k}{dt}(t)\,dt = \dist(0, \tilde M^k_{-1}) \to 0,\]
there are sequences $s_k \in [-1,0]$ and $x_k \in \tilde M_{s_k}^k$ such that $|x_k| = f_k(s_k) \to 0$ and $\tilde G_k (x_k,s_k) \to 0$. By the noncollapsing property, $\tilde M^k_{s_k}$ admits an inscribed ball at $x_k$ whose radius becomes arbitrarily large as $k \to \infty$. In particular, given any $R >0$, $\tilde M^k_{s_k}$ encloses the ball $B(y_k,R)$ for all large $k$, where $y_k := x_k - R\tilde \nu_k(x_k,s_k)$ and $\tilde\nu_k$ is the outward unit normal to $\tilde M^k_t$. By the avoidance principle, $\tilde M^k_0$ encloses $B(y_k, (R^2 - 2\gamma(1,\dots,1))^{1/2})$ for all large $k$. Choosing $R^2 = 1+2\gamma(1,\dots,1)$, we conclude that $B(y_k, 1)$ is contained in the region bounded by $\tilde M_{t}^k$ for all $t \leq 0$ and large $k$. This contradicts Lemma~\ref{interior curvature}, since 
    \[\tilde G_k(0,0) = (-t_k)^{1/2}G(0,0) \to \infty.\] 
\end{proof}

Let us fix an arbitrary sequence $a_k \to 0$. We show that the rescaled flows $M^k_t := a_k M_{a_k^{-2}t}$ converge smoothly along a subequence.

\begin{lemma}
Possibly after passing to a subsequence, the rescaled flows $M^k_t$ converge in $C^{\infty}_{\loc}(\mathbb{R}^{n+1}\times(-\infty,0))$ to a limiting flow $\hat M_t = \partial \hat \Omega_t$. 
\end{lemma}
\begin{proof}
Consider a fixed $s<0$. As a consequence of Lemma~\ref{Paraboloid}, we have $B(0, \eta(-t)^{1/2}) \subset \Omega_t^k$ when $t < s$ and $k$ is sufficiently large depending on $s$. 
    
Using this fact we will show that 
    \begin{equation}\label{blow-down speed positive}
        \inf_{B(0,R)} G_k(\cdot,s) >0
    \end{equation}
for every $R<\infty$. If this is not the case then, by Proposition~\ref{lower curvature}, $G_k \to 0$ uniformly in every compact subset of $\mathbb{R}^{n+1}\times(-\infty,s]$. Using the noncollapsing assumption we see that $\Omega_{s}^k$ converges to a halfspace in the Hausdorff topology. But since $G_k \to 0$ locally uniformly in $\mathbb{R}^{n+1}\times(-\infty,s]$, $\Omega^k_t$ converges to the same halfspace for every $t \leq s$. This contradicts $B(0,\eta(-t)^{1/2}) \subset \Omega^k_t$.

To complete the argument we use \eqref{blow-down speed positive} and Proposition~\ref{compactness} to establish subconvergence in $C^{\infty}_{\loc}(\mathbb{R}^{n+1} \times (-\infty, s])$ for each $s < 0$. The claim follows by sending $s \to 0$ and passing to a diagonal subsequence. 
\end{proof}

We have the following characterisation of $\hat M_t$ in case it is compact. (We do not make use of this result, it is only included for the sake of completeness.)

\begin{lemma}\label{blow-down=sphere}
Suppose $\hat M_t$ is compact. Then $M_t$ is a self-similarly shrinking sphere. In particular, $\hat M_t$ is a self-similarly shrinking sphere. 
\end{lemma}
\begin{proof}
Since $\hat M_t$ is compact, the splitting theorem \cite[Proposition~A.2]{Lynch_22_b} implies $\hat \lambda_1>0$ on $\hat M_t$ for every $t \in (-\infty,0)$. In particular, there is a positive constant $\varepsilon$ such that $\hat \lambda_1 \geq \varepsilon \hat H$ on $\hat M_{-1}$. Therefore, for every sufficiently large $k$, $\lambda_1 \geq \tfrac{\varepsilon}{2} H$ on the hypersurface $M_{-a_k^{-2}}$. If $\gamma$ is concave and inverse-concave, this inequality is preserved by the flow \cite{Andrews_Pinching}, and hence $\lambda_1 \geq \tfrac{\varepsilon}{2} H$ on $M_t$ for all $t \leq 0$. Using \cite[Theorem~1.7]{LangfordLynch} (see also \cite{Risa--Sinestrari}) we conclude that $M_t$ is a shrinking sphere solution.
 
We proceed similarly if $\gamma$ is convex rather than concave, but work with $\lambda_1/G$ in place of $\lambda_1/H$. A maximum principle argument shows that positive lower bounds for this ratio are preserved by the flow (see \cite{LangfordThesis}).
\end{proof}

Suppose now that $\hat M_t$ is noncompact. Applied to the solution $\hat M_t$, Proposition~4.5 in \cite{Lynch_22_b} yields the following statement. 

\begin{lemma}\label{enclosing cylinder}
Suppose that, possibly after applying a rotation, $\hat M_t$ satisfies
    \[\hat M_t \subset \{x \in \mathbb{R}^{n+1} : x_1^2 + \dots + x_{n}^2 \leq -2\gamma(0,1,\dots,1)t\}\]
for every $t < 0$. We then have
    \[\hat M_t = \{x \in \mathbb{R}^{n+1} : x_1^2 + \dots + x_{n}^2 = -2\gamma(0,1,\dots,1)t\}\]
for $t < 0$.
\end{lemma}

Using this result, we show that $\hat M_t$ is a self-similarly shrinking cylinder.

\begin{lemma}\label{blow-down=cylinder}
Suppose $\hat M_t$ is noncompact. Up to a rotation, we have
    \[\hat M_t = \{x \in \mathbb{R}^{n+1} : x_1^2 + \dots + x_{n}^2 = - 2\gamma(0,1,\dots,1)t\}\]
for $t < 0$. 
\end{lemma}
\begin{proof}
Each of the asymptotic cones $C(\hat \Omega_t)$ is equal to $C : = C(\hat \Omega_{-1})$ by Lemma~\ref{constant cone}. Since $\hat M_t$ is noncompact, $C$ is noncompact. Moreover, Lemma~\ref{neck-cap} implies that the dimension of $C$ is one, so $C$ is either a ray or a line. If $C$ is a line then $\hat M_t$ splits off a line, and the claim follows from Lemma~\ref{splitting}. Suppose then that $C$ is a ray. Applying a rotation if necessary, we may assume $C$ is contained in the $x_{n+1}$-axis.

Consider the set $\hat M_0 := \cap_{t < 0} \hat \Omega_t$. This set is closed and convex. A straightforward argument using Lemma~\ref{interior curvature} and the avoidance principle shows that 
    \[\sup_{\hat M_t \cap B(0,1)} \hat G\to \infty\]
as $t \to 0$. Therefore, using Lemma~\ref{neck-cap}, we conclude that the dimension of $\hat M_0$ is at most one. On the other hand $C \subset \hat \Omega_t$ for $t < 0$. It follows that $C \subset \hat M_0$. From this we conclude that the dimension of $\hat M_0$ is equal to one, and that every point in $C$ is reached by $\hat M_t$ as $t \to 0$ 

Let $x_j \in C$ be a sequence of points going to $\infty$, and consider the shifted flows $\hat M_t - x_j$. We know that $\hat M_t - x_j$ reaches the origin as $t \to 0$. Therefore, by the avoidance principle, for each $t < 0$ there is a point in $\hat M_t - x_j$ whose distance to the origin is at most $\sqrt{-2\gamma(1,\dots,1)t}$. Moreover, the ball $B(0,\eta(-t)^{1/2})$ is contained in $\hat \Omega_t$, and hence in $\hat \Omega_t - x_j$, for every $t < 0$. This is due to Lemma~\ref{Paraboloid}. 

Using these facts, we can appeal to Proposition~\ref{compactness} to extract a smooth limit of the shifted flows in $C^\infty_{\loc}(\mathbb{R}^{n+1}\times(-\infty,0))$. Denote the limiting flow by $\tilde M_t$. By construction, $\tilde M_t$ encloses the line through the origin parallel to $C$. Therefore, $\tilde M_t$ splits off a line and (since we rotated to align $C$ with the $x_{n+1}$-axis) Lemma~\ref{splitting} implies that
    \[\tilde M_t = \{x \in \mathbb{R}^{n+1} : x_1^2 + \dots + x_{n}^2 = -2\gamma(0,1,\dots,1)t\}.\]
But $\tilde M_t$ also encloses $\hat M_t$, so using Lemma~\ref{enclosing cylinder} we conclude that 
    \[\hat M_t = \{x \in \mathbb{R}^{n+1} : x_1^2 + \dots + x_{n}^2 = -2\gamma(0,1,\dots,1)t\}.\]
This completes the proof. 
\end{proof}

Let $\bar M_\tau$ denote the rescaled family of hypersurfaces $e^{\tau/2} M_{-e^{-\tau}}$. These move with velocity $-(G - \frac{1}{2}\langle x, \nu\rangle) \nu$. We now prove the main result of this section, which asserts that in case $M_t$ is noncompact we have full (as opposed to subsequential) convergence of $\bar M_\tau$ to a cylinder of radius $(2\gamma(0,1,\dots,1))^{1/2}$ as $\tau \to -\infty$.

\begin{proposition}\label{blow-down}
Suppose $M_t$ is noncompact. Up to a rotation, the rescaled hypersurfaces $\bar M_\tau$ converge in $C^\infty_{\loc}(\mathbb{R}^{n+1})$ to the cylinder
    \[\Sigma := \{x \in \mathbb{R}^{n+1} : x_1^2 + \dots + x_{n}^2 = 2\gamma(0,1,\dots,1)\}\]
as $\tau \to -\infty$.
\end{proposition}
\begin{proof}
Each of the asymptotic cones $C(\Omega_t)$ is equal to $C : = C(\Omega_0)$ by Lemma~\ref{constant cone}. Moreover, Lemma~\ref{neck-cap} implies that the dimension of $C$ is one. After rotating if necessary, we may assume $C$ lies in the $x_{n+1}$-axis. 

It suffices to show that, for every sequence $\tau_k \to \infty$, $\bar M_{\tau_k}$ admits a subsequence which converges in $C^\infty_{\loc}(\mathbb{R}^{n+1})$ to $\Sigma$. Lemma~\ref{blow-down=cylinder} guarantees subconvergence to a cylinder of radius $\sqrt{2\gamma(0,1,\dots,1)}$ with axis through the origin. But since $C \subset \Omega_t$ for $t \leq 0$, the axis of every such limiting cylinder is the $x_{n+1}$-axis. The claim follows.
\end{proof}


\section{A barrier construction}\label{barrier section}

In this section we prove the following statement.

\begin{proposition}\label{barrier}
There is a positive constant $L_0 = L_0(n, \gamma)$ such that, for each sufficiently large $a > 0$, there exists a function $\Psi_a : [L_0, a] \to \mathbb{R}$ with the following properties:
\begin{itemize}
    \item [(i)] $\Psi_a(z) \geq 0$ with equality at $z = a$ and $\Psi_a(z)^2 < 2\gamma(0,1,\dots,1)$ for $z \in [L_0,a]$.
    \item [(ii)] $\Psi_a$ is strictly decreasing and strictly concave. 
    \item [(iii)] The set of points 
        \[\{(z, \Psi_a(z) \theta): z \in (L_0, a], \; \theta \in S^{n-1}\}\]
    is a smooth hypersurface which solves the $G$-shrinker equation, $G = \frac{1}{2}\langle x, \nu\rangle$.
    \item[(iv)] For each $z \in [L_0, a]$ we have 
        \[\Psi_a(z)^2 \geq 2\gamma(0,1,\dots,1)\bigg(1 - \frac{z^2}{a^2}\bigg).\]
    \item [(v)] There is a positive constant $C = C(n,\gamma)$ with the following property. For any given $L > L_0$, the inequality
        \[\Psi_a(z)^2 \leq 2\gamma(0,1,\dots,1)\bigg(1-\bigg(1-C\frac{\log a}{a^2}\bigg)\frac{z^2 - C}{a^2}\bigg).\]
    holds for $z \in [L_0, L]$, provided that $a$ is sufficiently large depending on $L$.
\end{itemize}
\end{proposition}

Consider a  strictly concave smooth function $h: (0, R) \to \mathbb{R}$. The hypersurface 
    \[\{(h(r), r \theta) : r \in (0,R), \; \theta \in S^{n-1}\}\]
solves $G = \frac{1}{2}\langle x,\nu\rangle$ if and only if $h$ solves the ODE
\begin{equation*}\label{ODEh}
    \gamma\bigg(-\frac{h_{rr}}{1+h_r^2}, -\frac{h_r}{r}, \dots, -\frac{h_r}{r}\bigg) = \frac{1}{2}(h - rh_r).
\end{equation*}
For a given $a > 0$, consider the function $\psi : (0, Ra) \to \mathbb{R}$ defined by
    \[h(r) = a -\frac{1}{a}\psi(ar),\]
and set $\rho = a r$. Then $h$ solves \eqref{ODEh} if and only if $\psi$ solves
    \begin{equation}\label{shrinker psi}
        \gamma\bigg(\frac{\psi_{\rho\rho}}{1+\psi_\rho^2}, \frac{\psi_\rho}{\rho}, \dots, \frac{\psi_\rho}{\rho}\bigg) = \frac{1}{2} + \frac{1}{2a^2}(\rho \psi_\rho - \psi).
    \end{equation}
    
As a first step towards proving Proposition~\ref{barrier}, for each $a>0$ we will construct a solution of \eqref{shrinker psi} which satisfies the initial conditions $\psi(0) = 0$ and $\psi_\rho(0) = 0$. 

\begin{proposition}\label{prop: sol psi on (0,R)}
For each $a > 0$ there is a strictly convex function 
    \[\psi \in C^\infty([0, \sqrt{2F(0,1)} a))\]
which solves \eqref{shrinker psi} with the initial conditions 
    \[\psi(0) = 0, \qquad \psi_\rho(0) = 0.\]
Moreover, there is a constant $C > 0$ depending only on $n$ and $\gamma$ such that $\psi_\rho \geq C\rho$ for every $\rho \in (0,\sqrt{2F(0,1)}a)$, and $\psi_\rho \leq C\varepsilon^{-1}\rho$ for every $\rho \in (0,\sqrt{(2-\varepsilon)F(0,1)}a)$.
\end{proposition}

The solution $\psi$ in Proposition~\ref{prop: sol psi on (0,R)} is constructed as a limit of solutions to a sequence of initial value problems, in which we pose carefully chosen initial conditions at $\rho_k$, where $\rho_k$ is a sequence of positive numbers tending to 0. Some of our arguments take inspiration from \cite{rengaswami2021rotationally}, where \eqref{shrinker psi} was studied in the case $a = \infty$ (which corresponds to translating, rather than shrinking, $G$-flows). To be able to take a limit of the approximating solutions using the Arzel\'{a}--Ascoli theorem, we need to establish uniform derivative estimates. We first prove a priori estimates showing that, for a convex solution of \eqref{shrinker psi implicit}, appropriate control on the gradient implies a $C^2$-estimate and hence higher derivative bounds via bootstrapping. This reduces the entire construction to the derivation of gradient bounds for the approximating solutions, which we achieve using sub- and supersolutions and a comparison principle. 

It will be convenient to rewrite \eqref{shrinker psi}. Consider the function $F:\mathbb{R}^2_+ \to \mathbb{R}_+$ given by
    \begin{align*}
        F(x,y):=\gamma(x, y, \dots, y).
    \end{align*}
Since $\gamma$ is increasing in each argument we have $\frac{\partial F}{\partial x} > 0$, and consequently the set 
    \[\{ F(x, y) - z = 0 : (x, y, z) \in \mathbb{R}_+^3\}\]
is a smooth hypersurface. Moreover, by the implicit function theorem there is an open set $U \subset \mathbb{R}^2_+$ and a smooth function $f : U \to \mathbb{R}$ such that $F(x,y) - z = 0$  if and only if $(y,z) \in U$ and $x = f(y,z)$. In particular, $F(f(y, z), y) = z$ for all $(y,z) \in U$. We have
    \[ \frac{\partial f}{\partial y}(y,z) = - \frac{\frac{\partial F}{\partial y}(f(y,z),y)}{\frac{\partial F}{\partial x}(f(y,z),y)}, \qquad \frac{\partial f}{\partial z}(y,z) = \frac{1}{\frac{\partial F}{\partial x}(f(y,z),y)},\]
so $f$ is strictly decreasing in its first argument and strictly increasing in its second. If $\psi$ is such that
    \begin{equation}\label{psi in U}
        \bigg(\frac{\psi_\rho}{\rho}, \frac{1}{2} + \frac{1}{2a^2}(\rho \psi_\rho - \psi)\bigg) \in U,
    \end{equation}
then $\psi$ solves \eqref{shrinker psi} if and only if 
    \begin{equation}\label{shrinker psi implicit}
        \psi_{\rho \rho} = (1+\psi_\rho^2) f\bigg(\frac{\psi_\rho}{\rho}, \frac{1}{2} + \frac{1}{2a^2}(\rho \psi_\rho - \psi)\bigg).
    \end{equation}
    
Let $Q \in (0, \infty]$ denote $\lim_{x \to \infty} F(x,1)$. For example, when $\gamma(\lambda) = \lambda_1 +\dots + \lambda_n$ we have $Q = \infty$, whereas for the speed $\gamma(\lambda) = (\sum_{i < j} (\lambda_i + \lambda_j)^{-1})^{-1}$ we have $Q = \frac{1}{4(n-1)(n-2)}$. Observe that $f(y,z) < \infty$ if and only if $z/y = F(f(1,z/y),1) < Q$. Moreover, $f(y,z) > 0$ if and only if $z/y = F(f(1,z/y),1) > F(0,1)$. Therefore, $U$ is the open cone 
    \[U = \{(y,z) \in \mathbb{R}^2_+ : F(0,1) < z/y < Q\}.\]

If $\psi$ is a $C^2$ solution of \eqref{shrinker psi} such that $\psi_\rho >0$ and $\psi_{\rho\rho} >0$ then
    \[F\bigg(\frac{\rho\psi_{\rho\rho}}{\psi_\rho(1+\psi_\rho^2)},1\bigg) = \frac{\rho}{\psi_\rho} \gamma\bigg(\frac{\psi_{\rho\rho}}{1+\psi_\rho^2}, \frac{\psi_\rho}{\rho}, \dots, \frac{\psi_\rho}{\rho}\bigg) = \frac{\rho}{\psi_\rho} \bigg(\frac{1}{2} + \frac{1}{2a^2}(\rho \psi_\rho - \psi)\bigg),\]
and the left-hand side takes values in $(F(0,1), Q)$, so \eqref{psi in U} holds, and hence we are free to pass between the two forms of the equation \eqref{shrinker psi} and \eqref{shrinker psi implicit}.

\subsection{A priori estimates} The possibility that $Q < \infty$ is one reason why the construction we carry out in this section is more subtle than its counterpart for the mean curvature flow. Geometrically, if $\gamma$ is such that $Q < \infty$, a convex hypersurface on which $G$ is bounded can have principal curvatures which are arbitrarily large. This cannot happen if $\gamma$ is uniformly elliptic on the hypersurface. The following lemma provides conditions under which $\gamma$ is uniformly elliptic on a strictly convex solution of \eqref{shrinker psi}. Given a solution, let us define 
    \[\Lambda := \frac{\rho\psi_{\rho\rho}}{\psi_\rho(1+\psi_\rho^2)}.\]
Thinking of $\psi$ as the profile of a surface of rotation, $\Lambda$ is equal to the radial principal curvature divided by the principal curvature in the direction of rotation. 

\begin{lemma}\label{f bound}
Let $\psi$ be a solution of \eqref{shrinker psi} on the interval $[\rho_0, \rho_1]$ such that $\psi_\rho > 0$ and $\psi_{\rho\rho} > 0$. Suppose $\rho_1 < \sqrt{2F(0,1)}a$. There is a constant $C$ depending only on $n$ and $\gamma$ such that 
    \[\Lambda(\rho) \leq \max\{C, \Lambda(\rho_0)\}\]
for each $\rho \in [\rho_0, \rho_1]$.
\end{lemma}
\begin{proof}
Let $B(\rho):= \rho \psi_\rho^{-1}(\frac{1}{2} + \frac{1}{2a^2}(\rho \psi_\rho - \psi))$. Writing \eqref{shrinker psi} in the form \eqref{shrinker psi implicit} gives
    \[\Lambda = f(1,B).\]
We compute 
    \begin{align*}
        \rho B_\rho  = \bigg(\frac{\rho^2}{2a^2} - B\bigg) (1+\psi_\rho^2)f(1,B) + B.
    \end{align*}
Let us define $2\varepsilon_0 := 1-\frac{F(0,1)}{Q}$. If $\rho < \sqrt{2F(0,1)}a$ is such that 
    \[B(\rho) > \max\{\tfrac{F(0,1)}{1-\varepsilon_0}, F(1/\varepsilon_0, 1)\}\]
then $f(1,B(\rho)) > 1/\varepsilon_0$, and hence 
    \[\rho B_\rho < (F(0,1) - (1-\varepsilon_0)B - \varepsilon_0 B) (1+\psi_\rho^2)f(1,B) + B < 0.\]
It follows that 
    \[\max_{\rho_0 \leq \rho \leq \rho_1} B(\rho) \leq \max\{B(\rho_0), \tfrac{F(0,1)}{1-\varepsilon_0}, F(1/\varepsilon_0, 1)\}.\]
Since $\tfrac{F(0,1)}{1-\varepsilon_0}$ and $F(1/\varepsilon_0, 1)$ are both less than $Q$, we may define a finite constant 
    \[C := \max\Big\{f\left(1,\tfrac{F(0,1)}{1-\varepsilon_0}\right), f(1,F(1/\varepsilon_0, 1))\Big\}.\]
Since $f$ is increasing in its second argument,
    \[\max_{\rho_0 \leq \rho \leq \rho_1} f(1,B(\rho)) \leq \max\{C, f(1,B(\rho_0))\}.\]
The claim follows. 
\end{proof}

In addition to the upper bound in Lemma~\ref{f bound}, we have the following lower bound for $\Lambda$.

\begin{lemma}\label{f bound lower}
Let $\psi$ be a solution of \eqref{shrinker psi} on the interval $[\rho_0, \rho_1]$ such that $\psi_\rho > 0$ and $\psi_{\rho\rho} > 0$. There is a constant $c$ depending only on $n$ and $\gamma$ such that 
    \[\Lambda(\rho) \geq \min\{c(1+\psi_\rho(\rho_1)^2)^{-1}, \Lambda(\rho_0)\},\]
for each $\rho \in [\rho_0, \rho_1]$. 
\end{lemma}
\begin{proof}
We have
    \[\rho B_\rho = \bigg(\frac{\rho^2}{2a^2} - B\bigg) (1+\psi_\rho^2)f(1,B) + B >  (1-Mf(1,B))B,\]
where $M := (1+\psi_\rho(\rho_1)^2)$. This implies that $B_\rho > 0$ whenever $B < F(M^{-1}, 1)$. Therefore,
    \[\min_{\rho_0 \leq \rho \leq \rho_1} B \geq \min\{B(\rho_0), F(M^{-1}, 1)\}\]
and hence
    \[\min_{\rho_0 \leq \rho \leq \rho_1} f(1,B) \geq \min\{f(1,B(\rho_0)), f(1,F(M^{-1}, 1))\}.\]
There is a constant $c$ depending only on $n$ and $\gamma$ such that 
    \[f(1,F(M^{-1}, 1)) \geq cM^{-1},\]
which gives the claim when we insert $\Lambda = f(1,B)$. 
\end{proof}

Lemma~\ref{f bound} can be employed to prove higher derivative bounds for solutions to \eqref{shrinker psi}, via the following result. For our purposes a $C^3$-estimate will suffice, but a bound on the $C^k$-norm can be proven along the same lines. 

\begin{lemma}\label{psi C^3}
Let $\psi$ be a solution of \eqref{shrinker psi} on the interval $[\rho_0, \rho_1]$ such that $\psi_\rho > 0$ and $\psi_{\rho\rho} > 0$. Define
    \[\bar \Lambda := \sup_{\rho_0 \leq \rho \leq \rho_1} \Lambda(\rho) < \infty.\] 
We then have 
    \[\sup_{\rho_0 \leq \rho \leq \rho_1} |\psi_{\rho\rho\rho}(\rho)| \leq Ca^{-2}\rho_1 + C(1+\rho_0^{-1})\bigg(1+ \sup_{\rho_0 \leq \rho \leq \rho_1}\frac{\psi_\rho}{\rho}\bigg)\]
where $C$ depends only on $n$, $\gamma$, $\bar \Lambda$, and $\|\psi\|_{C^2([\rho_0, \rho_1])}$. 
\end{lemma}
\begin{proof}
Differentiating \ref{shrinker psi implicit}, we obtain 
    \begin{align*}
    \psi_{\rho\rho\rho}&=2\psi_\rho \, \psi_{\rho \rho} \, \,f\left( \frac{\psi_\rho}{\rho}, \, \frac{1}{2} + \frac{1}{2a^2} ( \rho \psi_\rho - \psi) \right)\\
    &+ (1+ \psi_\rho^2) \frac{\partial f}{\partial y}\bigg(\frac{\psi_\rho}{\rho}, \frac{1}{2} + \frac{1}{2a^2}(\rho\psi_\rho - \psi)\bigg) \bigg(\frac{\psi_{\rho\rho}}{\rho} - \frac{\psi_\rho}{\rho^2}\bigg)\\
    &+ (1+ \psi_\rho^2) \frac{\partial f}{\partial z} \bigg(\frac{\psi_\rho}{\rho}, \frac{1}{2} + \frac{1}{2a^2}(\rho\psi_\rho - \psi)\bigg) \frac{1}{2a^2}\rho \psi_{\rho \rho}.
    \end{align*}
Recalling the notation $B := \rho \psi_\rho^{-1}(\frac{1}{2} + \frac{1}{2a^2}(\rho \psi_\rho - \psi))$, and given that 
    \[F(\Lambda,1) = F(f(1,B), 1) = B,\]
our hypotheses ensure that 
    \[\bigg(\frac{\psi_\rho}{\rho}, \frac{1}{2} + \frac{1}{2a^2}(\rho\psi_\rho - \psi)\bigg) \in U',\]
where $U' := \{(y,z) \in U: F(0, 1) < z/y \leq F(\bar \Lambda,1)\}$. Using the fact that $f$ is homogeneous of degree one, it is straightforward to show that its first derivatives are bounded by some constant $C$ in the cone $U'$, where $C = C(n,\gamma,\bar \Lambda)$. The claim follows. 
\end{proof}

\subsection{Local existence} As mentioned above, to prove Proposition~\ref{prop: sol psi on (0,R)}, we will solve a sequence of initial value problems and take a limit. To do so we will need appropriate sub- and supersolutions. 

The following two lemmas establish that $w = \frac{\theta}{4F(1,1)}\rho^2$ is a lower barrier for solutions of \eqref{shrinker psi implicit} when $\theta \in (\frac{F(1,1)}{Q}, 1)$. 

\begin{lemma}\label{barriers for implicit shrinker}
Fix a constant $\frac{F(1,1)}{Q}<\theta <1$  and let $w = \frac{\theta}{4F(1,1)}\rho^2$. We then have
    \[w_{\rho\rho} < (1+w_\rho^2)f\bigg(\frac{w_\rho}{\rho}, \frac{1}{2}\bigg)\]
for all $\rho \geq 0$.
\end{lemma}
\begin{proof}
Since
    \[F(0,1) \frac{w_\rho}{\rho} = \frac{F(0,1)}{2F(1,1)} < \frac{1}{2} < \frac{Q}{2F(1,1)} = Q \frac{w_\rho}{\rho},\]
we have $(\frac{w_\rho}{\rho}, \frac{1}{2}) \in U$ for all $\rho \geq 0$. Observe that
    \[F\bigg(f\Big(\frac{w_\rho}{\rho}, \frac{1}{2}\Big), \frac{w_\rho}{\rho}\bigg) = \frac{1}{2}\]
and 
    \[F\bigg(\frac{w_{\rho\rho}}{1+w_\rho^2}, \frac{w_\rho}{\rho}\bigg) = \frac{\theta}{2F(1,1)}F\bigg(\frac{1}{1+w_\rho^2},1\bigg).\]
Since $F$ is increasing in each of its arguments, we see that  
    \[w_{\rho\rho} < (1+w_\rho^2)f\bigg(\frac{w_\rho}{\rho}, \frac{1}{2}\bigg)\]
if and only if 
    \[\frac{\theta}{2F(1,1)}F\bigg(\frac{1}{1+w_\rho^2},1\bigg) < \frac{1}{2},\]
or equivalently
    \[F\bigg(\frac{1}{1+w_\rho^2},1\bigg) < \theta^{-1} F(1,1).\]
The last inequality is always true when $\theta < 1$. 
\end{proof}

\begin{lemma}\label{ODE comparison}
Let $\psi \in C^2([\rho_0, \rho_1))$ be a solution of \eqref{shrinker psi implicit} satisfying $\psi_\rho > 0$ and $\psi_{\rho\rho}>0$. Let $w = \tfrac{\theta}{4F(1,1)}\rho^2$ for some constant $\theta \in (\frac{F(1,1)}{Q}, 1)$. If at $\rho = \rho_0$ we have 
    \[\psi \geq w, \qquad \psi_\rho \geq w_\rho, \qquad \rho\psi_\rho - \psi \geq 0,\]
then each of these inequalities also holds for $\rho \in (\rho_0, \rho_1)$.
\end{lemma}
\begin{proof}
Let $s := \sup\{\rho \in (\rho_0, \rho_1) : \psi_\rho(\rho) > w_\rho(\rho)\}$. At $\rho = \rho_0$ we have 
    \[w_{\rho\rho} < (1+w_\rho^2)f\bigg(\frac{w_\rho}{\rho}, \frac{1}{2}\bigg) \leq (1+\psi_\rho^2)f\bigg(\frac{\psi_\rho}{\rho},\frac{1}{2} + \frac{1}{2a^2}(\rho \psi_\rho - \psi)\bigg) = \psi_{\rho\rho},\]
so since $\psi_\rho(\rho_0) \geq w_\rho(\rho_0)$, we deduce that $s > \rho_0$. With the aim of deriving a contradiction, suppose $s < \rho_1$. We then have 
    \[\psi_\rho(\rho) > w_\rho (\rho)\]
for $\rho < s$, and 
    \[\psi_{\rho}(s) = w_\rho(s), \qquad \psi_{\rho\rho}(s) \leq  w_{\rho\rho}(s).\]
In addition, since 
    \[(\rho \psi_\rho - \psi)_\rho = \rho \psi_{\rho\rho} > 0\]
for $\rho \in [\rho_0, s)$, and $\rho \psi_\rho - \psi \geq 0$ at $\rho = \rho_0$, we have 
    \[\rho \psi_\rho - \psi_\rho \geq 0\]
for all $\rho \in [\rho_0, \rho_1)$. We conclude that, at $\rho = s$,
    \[\psi_{\rho\rho} = (1+\psi_\rho^2)f\bigg(\frac{\psi_\rho}{\rho}, \frac{1}{2} + \frac{1}{2a^2}(\rho \psi_\rho - \psi)\bigg) \geq (1+w_\rho^2)f\bigg(\frac{w_\rho}{\rho}, \frac{1}{2}\bigg) > w_{\rho \rho}.\]
This is a contradiction, so $s = \rho_1$. The claim follows. 
\end{proof}

Next we identify upper barriers for solutions of \eqref{shrinker psi implicit}.

\begin{lemma}\label{ODE comparison upper}
Let $\psi \in C^2([\rho_0, \rho_1))$ be a solution of \eqref{shrinker psi implicit} satisfying $\psi_\rho > 0$ and $\psi_{\rho\rho}>0$. Fix a constant $\Theta > \frac{F(1,1)}{F(0,1)}$. We assume
    \[\rho_1^2 < 2a^2(F(0,1) - F(1,1)\Theta^{-1}),\]
and define $W = \tfrac{\Theta}{4F(1,1)}\rho^2$. If at $\rho = \rho_0$ we have 
    \[\psi \leq W, \qquad \psi_\rho < W_\rho,\]
then these inequalities also hold for $\rho \in (\rho_0, \rho_1)$. 
\end{lemma}
\begin{proof}
First observe that
    \[s := \sup\{\rho \in [\rho_0, \rho_1) : W_\rho(\rho) > \psi_\rho(\rho)\}\]
satisfies $s > \rho_0$. With the aim of deriving a contradiction, suppose that $s < \rho_1$. We then have 
    \[W_\rho(\rho) - \psi_\rho(\rho) > 0\]
for $\rho \in (\rho_0, s)$, and 
    \[W_\rho(s) = \psi_\rho(s).\]
Since $\psi_{\rho\rho} > 0$ we have $\big(\frac{\psi_\rho}{\rho}, \frac{1}{2} + \frac{1}{2a^2}(\rho \psi_\rho - \psi)\big)\in U$ and, in particular,
    \[\frac{1}{2} + \frac{1}{2a^2}(\rho \psi_\rho - \psi) > F(0,1)\frac{\psi_\rho}{\rho}.\]
At $\rho = s$ this yields 
    \[\frac{1}{2} + \frac{1}{2a^2}(s W_\rho(s) - \psi(s)) > F(0,1)\frac{W_\rho(s)}{s},\]
which we rearrange to obtain
    \[s^2 > 2a^2(F(0,1) - F(1,1)\Theta^{-1}).\]
This contradicts our assumption that $\rho_1^2 < 2a^2(F(0,1) - F(1,1)\Theta^{-1})$. The claim follows. 
\end{proof}

We now have all the tools needed to prove Proposition~\ref{prop: sol psi on (0,R)}. 

\begin{proof}[Proof of Proposition~\ref{prop: sol psi on (0,R)}]
Fix a constant $\theta \in (\frac{F(1,1)}{Q}, 1)$ and set $w = \frac{\theta}{4F(1,1)}\rho^2$. Let $\rho_k$ be a decreasing sequence of positive numbers such that $\rho_k \to 0$ as $k \to \infty$. By the Picard--Lindel\"{o}f theorem, for each $k$ there exists some maximal $R_k > \rho_k$ such that \eqref{shrinker psi implicit} admits a solution $\psi^k \in C^2([\rho_k, R_k))$ which is strictly convex and satisfies the initial conditions
    \[\psi^k(\rho_k) = w(\rho_k), \qquad \psi^k_\rho(\rho_k) = w_\rho(\rho_k).\]
By Lemma~\ref{ODE comparison} we have
    \[\psi^k \geq \frac{\theta}{4F(1,1)}\rho^2, \qquad \psi^k_\rho \geq \frac{\theta}{2F(1,1)}\rho, \qquad \rho \psi^k_\rho - \psi^k \geq 0\]
for every $\rho \in [\rho_k, R_k)$. Moreover, by Lemma~\ref{ODE comparison upper}, given a constant $\Theta > \frac{F(1,1)}{F(0,1)}$ we have 
    \[\psi^k \leq \frac{\Theta}{4F(1,1)}\rho^2, \qquad \psi^k_\rho \leq \frac{\Theta}{2F(1,1)}\rho,\]
for every $\rho \in [\rho_k, R_k)$ such that $\rho^2 < 2a^2(F(0,1)-F(1,1)\Theta^{-1})$. 

At $\rho = \rho_k$ we have 
    \[F(\Lambda^k, 1) = \frac{\rho}{\psi^k_\rho}\bigg(\frac{1}{2} + \frac{1}{2a^2}(\rho\psi^k_\rho - \psi^k)\bigg) = \frac{F(1,1)}{\theta} + \frac{\rho_k^2}{4a^2} \to \frac{F(1,1)}{\theta}\]
as $k\to\infty$. Since $F(1,1) < \frac{F(1,1)}{\theta} < Q$, we conclude that $\lim_{k \to \infty}\Lambda^k(\rho_k)$ is some finite number in $(1,\infty)$ which depends only on $n$, $\gamma$ and $\theta$. Therefore, for sufficiently large $k$, Lemma~\ref{f bound} yields the bound
    \[\Lambda^k(\rho) \leq C\]
for every $\rho \in [\rho_k, R_k)$ satisfying $\rho < \sqrt{2F(0,1)}a$, where $C = C(n,\gamma, \theta)$. Consequently, we have 
    \[\psi^k_{\rho\rho} \leq C (1+|\psi^k_\rho|^2) \frac{\psi^k_\rho}{\rho} \leq C \frac{\Theta}{2F(1,1)}\bigg(1+\frac{\Theta^2}{4F(1,1)^2}\rho^2\bigg)\]
for every $\rho \in [\rho_k, R_k)$ such that $\rho^2 < 2a^2(F(0,1) - F(1,1)\Theta^{-1})$. In addition, Lemma~\ref{f bound lower} gives
    \[\Lambda^k(\rho) \geq c\bigg(1 + \frac{\Theta^2}{4F(1,1)^2}\rho^2\bigg)^{-1},\]
and hence 
    \[\psi_{\rho\rho}^k \geq \frac{c\hspace{0.3mm}\theta}{2F(1,1)}\bigg(1 + \frac{\Theta^2}{4F(1,1)^2}\rho^2\bigg)^{-1}\]
for every $\rho \in [\rho_k, R_k)$ such that $\rho^2 < 2a^2(F(0,1) - F(1,1)\Theta^{-1})$, where $c$ depends only on $n$ and $\gamma$. 

Suppose $R_k < \sqrt{2F(0,1)}a$. By the estimates above, when $k$ is sufficiently large we have
    \[\sup_{\rho_k \leq \rho < R_k} |\psi^k| + |\psi^k_\rho| + |\psi^k_{\rho\rho}| < \infty, \qquad \sup_{\rho_k \leq \rho < R_k} \Lambda^k < \infty, \qquad \inf_{\rho_k \leq \rho < R_k} \psi^k_{\rho\rho} > 0.\]
Therefore, the Picard-Lindel\"{o}f theorem can be used to extend $\psi^k$ beyond $R_k$ so that it is still a strictly convex $C^2$ solution of \eqref{shrinker psi implicit}. This contradicts maximality, so we must have $R_k \geq \sqrt{2F(0,1)}a$.

We now take a limit of the solutions $\psi^k$ as $k \to \infty$. We have shown that $\|\psi^k\|_{C^2(I)}$ is bounded independently of $k$ for every compact subinterval $I \subset (0,\sqrt{2F(0,1)}a)$, and using Lemma~\ref{psi C^3} we can bound $\|\psi^k\|_{C^3(I)}$ in terms of $\|\psi^k\|_{C^2(I)}$. Therefore, by a standard diagonal argument, $\psi^k$ subconverges to some limiting function $\psi$ in $C^2_{\loc}((0, \sqrt{2F(0,1)}a))$. The function $\psi$ is strictly convex, and solves \eqref{shrinker psi implicit}, so it is smooth. Moreover, we have
    \[\psi \geq \frac{\theta}{4F(1,1)} \rho^2, \qquad \psi_\rho \geq \frac{\theta}{2F(1,1)}\rho\]
for $\rho \in (0, \sqrt{2F(0,1)}a)$, and 
    \[\psi \leq \frac{\Theta}{4F(1,1)}\rho^2, \qquad \psi_\rho \leq \frac{\Theta}{2F(1,1)}\rho\]
for $\rho \in (0, \sqrt{2(F(0,1)-F(1,1)\Theta^{-1})}a)$. In particular, $(\psi, \psi_\rho) \to 0$ as $\rho \to 0$.

It remains to show that $\psi$ can be extended smoothly to $\rho = 0$. Let $\beta(\rho) = \frac{\psi_\rho}{\rho}$. We have 
    \begin{align*}
    \rho \beta_{\rho \rho}  = (\rho \beta)_{\rho\rho} - 2\beta_\rho = \psi_{\rho \rho \rho} - 2\beta_\rho
    \end{align*}
and hence
    \begin{align*}
    \rho \beta_{\rho \rho}&=2\psi_\rho \, \psi_{\rho \rho} \, \,f\left( \frac{\psi_\rho}{\rho}, \, \frac{1}{2} + \frac{1}{2a^2} ( \rho \psi_\rho - \psi) \right)\\
    &+ (1+ \psi_\rho^2) \frac{\partial f}{\partial y}\bigg(\frac{\psi_\rho}{\rho}, \frac{1}{2} + \frac{1}{2a^2}(\rho\psi_\rho - \psi)\bigg) \beta_\rho\\
    &+ (1+ \psi_\rho^2) \frac{\partial f}{\partial z} \bigg(\frac{\psi_\rho}{\rho}, \frac{1}{2} + \frac{1}{2a^2}(\rho\psi_\rho - \psi)\bigg) \frac{1}{2a^2}\rho \psi_{\rho \rho} - 2\beta_\rho. 
    \end{align*}
The right-hand side is positive at any point where $\beta_\rho < 0$. Therefore, $\beta_\rho$ is either negative for all sufficiently small $\rho$ or positive for all sufficiently small $\rho$. In particular, $\beta$ converges as $\rho \to 0$, and hence $\psi$ extends to a $C^2$ function on the interval $[0,\sqrt{2F(0,1)}a)$. Using \eqref{shrinker psi} we conclude that $\psi_{\rho\rho}(0) = \frac{1}{2F(1,1)}$.

We thus have that the function
    \[u(x) := a - a^{-1}\psi(a|x|)\]
is of class $C^2$ in $B^{n-1}_{\sqrt{2F(0,1)}}(0)$. Moreover, $u$ solves the shrinker equation,
    \[\gamma\bigg(-\bigg(\delta^{ik} - \frac{D^i u D^k u}{1+|Du|^2}\bigg)\frac{D_k u D_j u}{\sqrt{1+|Du|^2}}\bigg)\bigg) = \frac{1}{2} \, (x,u) \cdot \frac{(-Du, 1)}{\sqrt{1+|Du|^2}}.\]
Standard theory for fully nonlinear elliptic PDE now implies that $u$ is of class $C^\infty$ (see for example Lemma~17.16 in \cite{GilbargTrudinger}). In particular, $\psi$ extends to a function in $C^\infty([0,\sqrt{2F(0,1)}a))$. This completes the proof. 
\end{proof}

Next we extend the solutions constructed in Proposition~\ref{prop: sol psi on (0,R)} to the origin. 

\subsection{Global behaviour} For each $a > 0$, let $\psi_a$ be a solution to \eqref{shrinker psi} as in Proposition~\ref{prop: sol psi on (0,R)}. We will typically suppress the dependence on $a$ and simply write $\psi = \psi_a$. We define 
    \[h(r) := a - a^{-1} \psi(ar)\]
and write $v(z)$ for the inverse of $h$, so that $r = v(h(r))$. Note that since $\psi$ is defined for $\rho \in [0, \sqrt{2F(0,1)}a)$ we have $0 \leq v < \sqrt{2F(0,1)}$. For each $a > 0$ the function $v$ is positive, strictly decreasing, strictly concave, and satisfies
    \begin{equation*}\label{ODEv}
        \gamma\bigg(-\frac{v_{zz}}{1+v_z^2}, \frac{1}{v}, \dots, \frac{1}{v}\bigg) = \frac{v}{2} - \frac{z}{2} v_z, \qquad v(a) = 0.
    \end{equation*}

We are interested in the asymptotic behaviour of $v$ as $a \to \infty$. Following \cite{ADS}, we prove upper and lower bounds for the quantity
\begin{align*}
    w(z):=z\frac{d}{dz}\log(2F(0,1)-v^2)=\frac{-2zvv_z}{2F(0,1)-v^2},
\end{align*}
and then integrate to obtain estimates for $v$. 

\begin{lemma}\label{w lower bound}
We have $w(z) > 2$. 
\end{lemma}
\begin{proof}
We first show that $w(z) \to 2 \tfrac{F(1,1)}{F(0,1)}$ as $z \to a$. Observe that when $z = h(\rho/a)$ we have 
    \begin{align*}
        w(z)=-\frac{zv(z)v_z(z)}{F(0,1)-\frac{v(z)^2}{2}}=\frac{\left(1-\frac{\psi(\rho)}{a^2}\right)}{F(0,1)-\frac{\rho^2}{2a^2}}\frac{\rho}{\psi_\rho(\rho)}.
    \end{align*}
Using $\psi_{\rho\rho}(0) = \frac{1}{2F(1,1)}$, we conclude that 
    \begin{align*}
      \lim_{z\to a} w =\frac{1}{F(0,1)} \lim_{\rho \to 0} \frac{\rho}{\psi_\rho} = \frac{1}{F(0,1)} \frac{1}{\psi_{\rho\rho}(0)} = 2 \frac{F(1,1)}{F(0,1)}.
    \end{align*}
    
Since $2 \tfrac{F(1,1)}{F(0,1)} > 2$, if the claim is false then there is a value $z_1$ such that $w(z_1) = 2$ and $w_z(z_1) \geq 0$. But we have 
    \[zw_z= w - w^2\bigg(\frac{1}{2} + \frac{F(0,1)}{v^2}\bigg)-(1+v_z^2)\frac{2z}{v_z(z)} f\bigg(\frac{1}{v}, \frac{1}{2}(v-zv_z)\bigg),\]
and $w(z_1) = 2$ implies that $\frac{1}{2}(v-zv_z)=\tfrac{F(0,1)}{v}$, which in turn gives
    \[f\bigg(\frac{1}{v}, \frac{1}{2}(v-zv_z)\bigg)=\frac{1}{v}f(1, F(0,1))=0\]
at $z = z_1$. It follows that
    \[w_z=-\frac{4F(0,1)}{zv^2}<0\]
at $z = z_1$, which is a contradiction. 
\end{proof}

Lemma~\ref{w lower bound} immediately implies the following estimate for $v$. 

\begin{lemma}
We have 
    \[v(z)^2 \geq 2F(0,1)\left(1-\frac{z^2}{a^2}\right).\]
\end{lemma}
\begin{proof}
The estimate $w > 2$ says exactly that 
    \[\frac{d}{dz}\log(2F(0,1)-v^2) > \frac{2}{z}.\]
Integrating this inequality from $z$ to $a$ gives the claim. 
\end{proof}

Next we derive an upper bound for $w$, using the following inequality.

\begin{lemma}\label{w evolution}
There is a constant $c >0$ depending only on $n$ and $\gamma$ such that
    \begin{align}\label{eq: w_supersolution}
        zw_z \geq w-w^2\bigg(\frac{1}{2}+\frac{F(0,1)}{v^2}\bigg) + cz^2(1+v_z^2) (w-2).
    \end{align}
\end{lemma}
\begin{proof}
We first note that
    \[-\frac{wz}{v_z}f\left(\frac{1}{v}, \frac{1}{2}(v-zv_z)\right) = \frac{2z^2}{2F(0,1) - v^2} f\bigg(1, \frac{2F(0,1) - v^2}{2} \bigg(\frac{w}{2} - 1\bigg) + F(0,1)\bigg).\]
To estimate the right-hand side we write
    \[f(1, b + F(0,1)) = f(1, b + F(0,1)) - f(1,F(0,1)) = \int_0^b \frac{\partial f}{\partial z}(1, t + F(0,1)) \, dt\]
and observe that since $\gamma$ is one-homogeneous,
    \begin{align*}
        \inf_{t \in [0, \infty)} \frac{\partial f}{\partial z}(1, t + F(0,1)) &= \inf_{t \in [0, \infty)} \dot\gamma^1(t,1,\dots,1)^{-1}\\
        &= \inf_{t \in [0,\infty)} \dot\gamma^1\bigg(\frac{t}{n-1+t}, \frac{1}{n-1+t},\dots,\frac{1}{n-1+t}\bigg)^{-1}.
    \end{align*}
The last infimum is over a compact subset of $\Gamma$ (the segment connecting $(0,\frac{1}{n-1},\dots,\frac{1}{n-1})$ with $(1, 0, \dots, 0)$), so 
    \[c:=\inf_{t \in [0, \infty)} \frac{\partial f}{\partial z}(1, t + F(0,1))\]
is a positive constant depending only on $n$ and $\gamma$, and we have 
    \[f\bigg(1, \frac{2F(0,1) - v^2}{2} \bigg(\frac{w}{2} - 1\bigg) + F(0,1)\bigg) \geq c \frac{2F(0,1) - v^2}{2} \bigg(\frac{w}{2} - 1\bigg).\]
Therefore,
    \[-\frac{wz}{v_z}f\left(\frac{1}{v}, \frac{1}{2}(v-zv_z)\right) \geq c z^2 (w - 2),\]
and the claim follows when we combine this inequality with
    \[zw_z= w - w^2\bigg(\frac{1}{2} + \frac{F(0,1)}{v^2}\bigg)-(1+v_z^2)\frac{2z}{v_z} f\bigg(\frac{1}{v}, \frac{1}{2}(v-zv_z)\bigg)\]
and $w > 2$. 
\end{proof}

By \cite{rengaswami2023bowl}, there is a unique strictly convex solution $\zeta : \mathbb{R}_+ \to \mathbb{R}_+$ to the equation
    \[\zeta_{\rho\rho} = (1+\zeta_\rho^2)f\bigg(\frac{\zeta_\rho}{\rho}, \frac{1}{2}\bigg)\]
which satisfies $(\zeta, \zeta_\rho) \to 0$ as $\rho \to 0$. This is the profile function of the bowl soliton which translates with speed $1/2$. We will make use of the following asymptotic expansions for $\zeta$ and $\zeta_\rho$, which are derived in Appendix~\ref{translator asymptotics}:
    \[\zeta(M) = \frac{1}{4F(0,1)} M^2 - 2\dot F^{1}(0,1) \log M + o(\log(M))\]
and     
    \[\zeta_\rho(M) = \frac{1}{2F(0,1)}M - 2\dot F^1(0,1)M^{-1} + o(M^{-1})\]
as $M \to \infty$, where $\dot F^1(0,1) = \dot \gamma^1(0,1,\dots,1)$. 

\begin{lemma}\label{shrinker to translator}
    As $a \to \infty$ we have $\psi \to \zeta$ in $C^{\infty}_{\loc}(\mathbb{R}_+)$. In particular, if $M$ is sufficiently large,
        \[\psi(M) \to \frac{1}{4F(0,1)} M^2 - \dot F^1(0,1) \log M + o(\log(M))\]
    and 
        \[\psi_{\rho}(M) \to \frac{1}{2F(0,1)}M - 2\dot F^1(0,1) M^{-1} + o(M^{-1})\]
    as $a \to \infty$. 
\end{lemma}
\begin{proof}
The claim follows if we can show that every sequence $a_k \to \infty$ admits a subsequence $a_{k_i}$ along which $\psi_{a_{k_i}} \to \zeta$ in $C^\infty_{\loc}(\mathbb{R}_+)$. So consider an arbitrary sequence $a_k \to \infty$. Recall that there is a positive constant $C$ depending only on $n$ and $\gamma$ such that $C^{-1} \rho \leq (\psi_{a_k})_\rho \leq C\rho$ holds in $(0, \sqrt{F(0,1)}a_k)$. Moreover, as in the proof of Proposition~\ref{prop: sol psi on (0,R)}, we have $(\psi_{a_k})_{\rho\rho}(\rho) \to \frac{1}{2F(1,1)}$ as $\rho \to 0$ for every $k$. Using Lemma~\ref{f bound}, Lemma~\ref{f bound lower} and Lemma~\ref{psi C^3}, we conclude that 
    \[\sup_k \|\psi_{a_k}\|_{C^3(I)} < \infty\]
for every compact interval $I \subset \mathbb{R}_+$. Therefore, after passing to a subsequence, we have that $\psi_{a_k}$ converges to some limit $\bar \psi$ in $C^{2,\alpha}_{\loc}(\mathbb{R}_+)$. The function $\bar \psi$ is strictly convex by Lemma~\ref{f bound lower}, and satisfies 
        \[\bar \psi_{\rho\rho} = (1+\bar \psi_\rho^2) f\bigg(\frac{\bar \psi_\rho}{\rho}, \frac{1}{2}\bigg).\]
In particular, $\bar \psi$ is smooth. Moreover, $\bar \psi$ satisfies the initial conditions $(\bar \psi, \bar \psi_\rho) \to 0$ as $\rho \to 0$. But $\zeta$ is the unique function with these properties, so $\bar \psi = \zeta$, as required. 
\end{proof}

Let us fix a large positive constant $M$ and set $z_{M,a} = a - a^{-1}\psi(M)$. By Lemma~\ref{shrinker to translator}, if $M$ is sufficiently large we have $a^{-1}z_{M,a} \to 1$ as $a \to \infty$. The following lemma concerns the asymptotic behaviour of $w(z_{M,a})$ as $a \to \infty$. (Note that since $z_{M,a} \sim a$, this is really a statement about the behaviour of our self-shrinking solutions near their tips.) 

\begin{lemma}
    If $M$ is sufficiently large then we have
        \[w(z_{M,a}) \to \frac{1}{F(0,1)} \frac{M}{\zeta_\rho(M)} = 2 + 8F(0,1)^2M^{-2} + O(M^{-4})\]
    as $a \to \infty$. 
\end{lemma}
\begin{proof}
    We first note that $z_{M,a} = h(M/a)$ and hence $v(z_{M,a}) = M/a$. It follows that 
        \[z_{M,a} \, v(z_{M,a}) \to M, \qquad v(z_{M,a})^2 \to 0\]
    as $a \to \infty$. Also, 
        \[v_z(z_{M,a}) = - \psi_\rho(M)^{-1} \to -\zeta_\rho(M)^{-1}\]
    as $a \to \infty$. Combining these facts we see that
        \[w(z_{M,a}) = \frac{-2z_{M,a} v(z_{M,a}) v_z(z_{M,a})}{2F(0,1)-v(z_{M,a})^2} \to \frac{1}{F(0,1)} \frac{M}{\zeta_\rho(M)}\]
    as $a \to \infty$. We conclude using the expansion
        \[\frac{1}{\zeta_\rho(M)} = \frac{1}{\frac{1}{2F(0,1)}M - 2\dot F^1(0,1) M^{-1} + O(M^{-2})} = \frac{2F(0,1)}{M} + 8F(0,1)^3M^{-3} + O(M^{-4}).\]
\end{proof}

We define 
    \[w_1(z) = \frac{1}{z^2} + \frac{1}{a^2 - z^2}\]
and set $\bar w = 2 + K w_1$, where $K := \max\{1, 6F(0,1), 17c^{-1}\}$ and $c$ is the constant appearing in Lemma~\ref{w evolution}. 

\begin{lemma}\label{w bdy cond}
    If $M$ is sufficiently large, and $a$ is sufficiently large depending on $M$, then we have
        \[w(z_{M,a}) < \bar w(z_{M,a}).\]
\end{lemma}
\begin{proof}
    Recalling that $z_{M,a} = a - a^{-1}\psi(M)$, we observe that
        \[a^2 - z_{M,a}^2 \to 2 \zeta(M)\]
    as $a \to \infty$, and hence 
        \[\bar w(z_{M,a}) \to 2 + \frac{K}{2\zeta(M)} = 2 + 2F(0,1)KM^{-2} + O(M^{-4}) \,\log(M)\]
    as $a \to \infty$. On the other hand,
        \[w(z_{M,a}) \to 2 + 8F(0,1)^2M^{-2} + O(M^{-4}).\]
    Since $K > 4F(0,1)$, the claim follows. 
\end{proof}

\begin{proposition}\label{w comparison}
    If $M$ is sufficiently large, and $a$ is sufficiently large depending on $M$, then we have 
        \[w(z) \leq \bar w(z)\]
    for $\sqrt{K} < z < z_{M,a}$.
\end{proposition}
\begin{proof}
 We claim that 
        \[z\overline{w}_z < \bar{w}-\overline{w}^2\left(\frac{1}{2}+\frac{F(0,1)}{v^2}\right) + cv^2(1+v_z^2) (\overline{w}-2)\]
for $\sqrt{K} < z < z_{M,a}$. 

If $a$ is sufficiently large then we have
    \[z \bar w_z = -\frac{2K}{z^2} + \frac{2Kz^2}{(a^2 - z^2)^2} \leq \frac{2K}{a^2 - z_{M,a}^2} z^2 w_1 \leq \frac{2K}{\zeta(M)} z^2 w_1.\]
If $M$ is large enough then we have $2K < \zeta(M)$, and so conclude that
    \[z \bar w_z < z^2 w_1,\]
for $z < z_{M,a}$ and all sufficiently large $a$.

Next we use $\bar w = 2 + Kw_1$ to obtain 
    \[cz^2(1+v_z^2) (\overline{w}-2) \geq c K z^2 w_1,\]
and use $v^2 \geq 2F(0,1)(1-z^2/a^2)$ to obtain
    \[\frac{1}{2} + \frac{F(0,1)}{v^2} \leq \frac{1}{2} + \frac{a^2}{2(a^2 - z^2)} \leq z^2 w_1.\]
Together these inequalities give
    \[\overline{w} - \bigg(\frac{1}{2} +\frac{F(0,1)}{v^2}\bigg)  \overline{w}^2 + cz^2(1+v_z^2) (\overline{w}-2) \geq -z^2 \bar w^2 w_1 + c K z^2 w_1,\]
and so since, for large $M$, we have
    \[\bar w = 2 + \frac{K}{z} + \frac{K}{a^2 - z^2} \leq 2 + \frac{K}{z^2} + \frac{K}{\zeta(M)} \leq 4\]
for $\sqrt{K} \leq z \leq z_{M,a}$, we may conclude that
    \[\overline{w} - \bigg(\frac{1}{2} +\frac{F(0,1)}{v^2}\bigg)  \overline{w}^2 + cz^2(1+v_z^2) (\overline{w}-2) \geq -16z^2 w_1 + c K z^2 w_1 \geq z^2 w_1\]
for $\sqrt{K} \leq z \leq z_{M,a}$. 

Putting all of this together, we see that
    \[z\overline{w}_z < \bar{w}-\overline{w}^2\left(\frac{1}{2}+\frac{F(0,1)}{v^2}\right) + cv^2(1+v_z^2) (\overline{w}-2)\]
for $\sqrt{K} \leq z \leq z_{M,a}$. In the previous lemma we showed that, if $a$ is sufficiently large, we have 
    \[w(z_{M,a}) < \bar w(z_{M,a}).\]
The claim now follows from a simple comparison principle argument. 
\end{proof}

\begin{proposition}\label{w upper bound}
    If $M$ is sufficiently large, and $a$ is sufficiently large depending on $M$, then we have
        \[v^2 \leq 2F(0,1)\bigg(1 - \bigg(1-C\frac{\log a}{a^2}\bigg) \frac{z^2 - K/2}{a^2}\bigg)\]
    for $\sqrt{K} < z < z_{M,a}$, where $C$ depends only on $K$ and $M$. 
\end{proposition}
\begin{proof}
We use $w = z\tfrac{d}{dz}\log(2F(0,1) - v^2)$ and the estimate in Lemma~\ref{w comparison} to obtain
    \begin{align*}
        \frac{\log(2F(0,1) - M^2/a^2)}{\log(2F(0,1) - v(z)^2)} &= \int_z^{z_{M,a}} \frac{w(\eta)}{\eta} \, d\eta\\
        &\leq \log\bigg(\frac{z_{M,a}^2}{z^2}\bigg) + \frac{K}{2z^2} + \frac{K}{2a^2}\log\bigg(\frac{z_{M,a}^2(a^2 - z^2)}{z^2(a^2 - z_{M,a}^2)}\bigg)\\
        &\leq \log\bigg(\frac{z_{M,a}^2}{z^2}\bigg) + \frac{K}{2z^2} + \frac{C}{a^2}\log a - \frac{K}{2a^2}\log (z^2(a^2 - z_{M,a}^2)),
    \end{align*}
where $C$ depends only on $K$ and $M$. If $M$ is so large that $\zeta(M) > 1/2$ then we have 
    \[a^2 - z_{M,a}^2 = 2\psi(M) - a^{-2}\psi(M)^2 > 1\]
for all sufficiently large $a$, and hence the last term on the right is negative for $z > \sqrt{K} >1$. We thus have 
    \[\frac{\log(2F(0,1) - M^2/a^2)}{\log(2F(0,1) - v(z)^2)} \leq \log\bigg(\frac{z_{M,a}^2}{z^2}\bigg) + \frac{K}{2z^2} + \frac{C}{a^2}\log a\]
for $\sqrt{K} < z \leq z_{M,a}$ whenever $a$ is sufficiently large. We rearrange and use $e^{-x} \geq 1 - x$ to obtain
    \[2F(0,1) - v^2 \geq \bigg(2F(0,1) - C\frac{\log a}{a^2}\bigg) \frac{z^2 - K/2}{a^2}\]
for $\sqrt{K} < z < z_{M,a}$, where $C$ depends only on $M$ and $K$. This implies the claim. 
\end{proof}

We can finally prove the main result of this section, Proposition~\ref{barrier}.

\begin{proof}[Proof of Proposition~\ref{barrier}]
Let $M$ be large enough and fixed so that Proposition~\ref{w upper bound} applies. It then suffices to take $\Psi_a(z) = v(z)$ and $L_0 = \sqrt{K}+1$. 
\end{proof}


\section{Sharp asymptotics for the rescaled flow}\label{asymptotics section}

Let $M_t = \partial\Omega_t$, $t \in (-\infty,0]$, be a strictly convex ancient $G$-flow which is noncollapsing and uniformly two-convex. In addition, we assume that $M_t$ noncompact. We define a family of rescaled hypersurfaces $\bar M_{\tau} = e^{\tau/2} M_{-e^{-\tau}}$. These move with pointwise normal speed $-(G-\tfrac{1}{2}\langle x, \nu\rangle)$. Proposition~\ref{blow-down} implies that, as $\tau \to -\infty$, $\bar M_\tau$ converges in $C^{\infty}_{\loc}(\mathbb{R}^{n+1})$ to a cylinder $\Sigma$ of radius $\sigma := \sqrt{2\gamma(0,1,\dots,1)}$ whose axis passes through the origin. Without loss of generality we may assume the axis of $\Sigma$ is the $x_{1}$-axis, which we equip with the coordinate $z$. That is,
    \[\Sigma := \{(z,\sigma\theta) : z \in \mathbb{R}, \; \theta \in S^{n-1}\}.\]

Recall that the asymptotic cone of $M_t$ (and hence of $\bar M_\tau$) is constant, by Lemma~\ref{constant cone}, and of dimension 1, by Lemma~\ref{neck-cap}. Given that $\bar M_\tau \to \Sigma$, the asymptotic cone is a subset of the $x_1$-axis. Since $M_t$ is strictly convex, we may assume without loss of generality that the asymptotic cone is $\{(z,0):z\leq 0\}$. For each $\tau$ we define 
    \[\bar z(\tau):= \sup\{z \in \mathbb{R} : (z,0) \in e^{\tau/2}\Omega_{-e^{-\tau}}\}.\]
Observe that $\bar z(\tau) \to \infty$ as $\tau \to -\infty$. We may define a function $u(z,\theta,\tau)$ such that $\bar M_\tau$ coincides with the set of points $(z,\sigma\theta) + u(z,\theta,\tau)\theta$ in the region of space where $z < \bar z(\tau)$. Since $\bar M_\tau$ is strictly convex and noncompact we have
    \[u > 0, \qquad \frac{\partial u}{\partial z} < 0, \qquad \frac{\partial^2 u}{\partial z^2}<0.\]

In this section we establish the following sharp asymptotic estimate for $u$.

\begin{proposition}\label{asymptotic for u}
    For every $L<\infty$ we have 
        \[\sup_{|z| \leq L} |u(\cdot,\tau)| \leq O(e^{\tau/2})\]
    as $\tau \to -\infty$.     
\end{proposition}

In other words, there is a positive constant $K$ such that
    \[M_t \cap \{|z| \leq L (-t)^{1/2}\} \subset \{(z,y): ||y| - (-2\gamma(0,1,\dots1) t)^{1/2}| \leq K \}\]
whenever $-t$ is sufficiently large. 

We first prove that the maximum of $|u|$ for $|z| \leq L$ controls $u$ and its derivatives in a subset which grows as $\tau \to -\infty$. To obtain these `inner-outer' estimates, we employ the hypersurfaces $\Sigma_a$ constructed in Section~\ref{barrier section} as barriers. We then use the inner-outer estimates to conduct a spectral analysis of the linearised equation for $u$ about $0$ as $\tau \to -\infty$. To do so we must cut off in space. Both linearising and cutting off introduce error terms, but we are be able to show that these decay rapidly as $\tau \to -\infty$. Proposition~\ref{asymptotic for u} is then proven by showing that the spectral decomposition is dominated by unstable modes.

We take inspiration from \cite{ADS,Brendle_Choi_a,Brendle_Choi_b}, which concern ancient solutions to the mean curvature flow, and Brendle's work on ancient solutions to the Ricci flow \cite{Brendle} (see also \cite{brendle2023rotational}). Similar ideas also played a role in Colding and Minicozzi's work on uniqueness of cylindrical tangent flows \cite{CM}. Our analysis is the first of its kind for a fully nonlinear geometric flow. The main new difficulty which we need to overcome in this setting is the presence of error terms depending on the Hessian of $u$. For an overview of the analysis in this section we refer back to the introduction of the paper. 

\subsection{Inner-outer estimates} For each sufficiently large $a >0$ we write 
    \[\Psi_a : [L_0, a] \to \mathbb{R}\]
for the function referred to in Proposition~\ref{barrier}. Since $u(\cdot,\tau) \to 0$ on compact subsets as $\tau \to -\infty$, for every $a>0$ the inequality 
    \[\Psi_a(z) < \sqrt{2\gamma(0,1,\dots,1)} + u(z,\theta,\tau)\]
holds for $L_0 \leq z \leq a$ and sufficiently large $-\tau$. Since $\Psi_a$ is the profile of a self-similarly shrinking $G$-flow, which is a stationary solution to the rescaled flow, the parabolic maximum principle implies the following statement. 

\begin{lemma}\label{comparison for rescaled flow}
Suppose $z_0 \geq L_0$ and $\tau_0 < 0$ are such that 
    \[\Psi_a(z_0) < \sqrt{2\gamma(0,1,\dots,1)} + u(z_0,\theta,\tau)\]
for every $\tau \leq \tau_0$. We then have 
    \[\Psi_a(z) \leq \sqrt{2\gamma(0,1,\dots,1)} + u(z,\theta,\tau)\]
for $z \in [z_0, a]$ and $\tau \leq \tau_0$. 
\end{lemma}

Let $L > L_0$ be an arbitrary large constant. For each integer $k \geq 0$ we define
    \[\delta_k := \sup_{\tau \leq -k} \sup_{|z| \leq L} |u(\cdot,\tau)|.\]
Since $\bar M_\tau$ is strictly convex we have $\delta_k > 0$. Since  $\delta_k \to 0$ as $k \to \infty$, after passing to a subsequence we may assume $\delta_k \leq 10^{-6}$. We fix a small constant $r > 0$. Choosing $r = 10^{-4}$ will prove to be sufficient.

In this section we write $\nabla$ for the Levi-Civita connection on $\mathbb{R}\times S^{n-1}$. We use $C$ to denote a positive constant which is independent of $k$. The value of $C$ may increase as we proceed. 

\begin{lemma}\label{u in growing set}
When $k$ is sufficiently large we have the estimates 
    \[|u| + \bigg|\frac{\partial u}{\partial z}\bigg| \leq C\delta_k^{1-2r}, \qquad |\nabla^{S^{n-1}} u| \leq C\delta_k^{1/2-2r}\]
for $|z| \leq 8\delta_{k}^{-r}$ and $\tau \leq -k$. Moreover, we have 
    \[|u| + \bigg|\frac{\partial u}{\partial z}\bigg| \leq C\delta_k\]
for $|z| \leq L+2$ and $\tau \leq -k$. 
\end{lemma}
\begin{proof}
Recall from Proposition~\ref{barrier} that when $a$ is large we have            \[\Psi_a(z) \leq \sqrt{2\gamma(0,1,\dots,1)\bigg(1-\bigg(1-C\frac{\log a}{a^2}\bigg)\frac{z^2 - C}{a^2}\bigg)}\]
for $z \in [L_0, L]$, where $C = C(n,\gamma)$. Let $z_0 = L_0 + \sqrt{C}$ and $a = \Lambda \delta_k^{-1/2}$, where 
    \[\Lambda^2 := \sqrt{\frac{\gamma(0,1,\dots,1)}{8}}(L_0^2 + 2L_0\sqrt{C}).\]
We may assume without loss of generality that $L > z_0$. The definition of $z_0$ ensures that 
    \[\sqrt{2\gamma(0,1,\dots,1)} - \delta_k > \Psi_a (z_0)\]
whenever $k$ is sufficiently large. Since $|u| \leq \delta_k$ for $|z| \leq L$ and $\tau \leq -k$,
    \[\sqrt{2\gamma(0,1,\dots,1)} + u(z_0, \theta, \tau) \geq \sqrt{2\gamma(0,1,\dots,1)} - \delta_k > \Psi_a(z_0)\]
for $\tau \leq -k$. Using Lemma~\ref{comparison for rescaled flow} and Proposition~\ref{barrier} we obtain 
    \[\sqrt{2\gamma(0,1,\dots,1)} + u(z,\theta,\tau) \geq \Psi_a(z) \geq \sqrt{2\gamma(0,1,\dots,1)\bigg(1- \frac{z^2}{a^2}\bigg)}\]
for $z \in [z_0, a]$ and $\tau \leq -k$. In particular, for all sufficiently large $k$, we have
    \[u \geq - C\delta_k^{1-2r}\]
for $z \in [z_0, 9\delta_{k}^{-r}]$ and $\tau \leq - k$. Moreover, 
    \[u \geq - C\delta_k\]
for $z \in [L+2, L+3]$ and $\tau \leq -k$.

Since $\frac{\partial u}{\partial z}< 0$ we conclude that 
    \[- C\delta_k^{1-2r} \leq u(z,\theta,\tau) \leq u(0,\theta,\tau) \leq \delta_k\]
for $z \in [0, 9\delta_{k}^{-r}]$ and $\tau \leq - k$. The mean value theorem then gives
    \[\inf_{8\delta_{k}^{-r} \leq  z \leq 9\delta_{k}^{-r}} \frac{\partial u}{\partial z}(z,\theta,\tau) \geq -C\delta_k^{1-r}\]
for $\tau \leq - k$, and so since $\frac{\partial u}{\partial z}<0$ and $\frac{\partial^2 u}{\partial z^2} < 0$, 
    \[\bigg|\frac{\partial u}{\partial z}\bigg| \leq C\delta_k^{1-r}\]
for $z \leq 8\delta_k^{-r}$ and $\tau \leq -k$. Using this inequality and the fundamental theorem of calculus we obtain
    \[|u| \leq C\delta_k^{1-2r}\]
for $|z| \leq 8\delta_k^{-r}$ and $\tau \leq -k$.

Similarly, since $u \geq -C\delta_k$ for $z \in [L+2, L+3]$, we have 
    \[\bigg|\frac{\partial u}{\partial z}\bigg| \leq C\delta_k\]
for $z \leq L+2$ and $\tau \leq -k$. It follows that
    \[|u| + \Big|\frac{\partial u}{\partial z}\Big| \leq C\delta_k\]
for $|z| \leq L+2$ and $\tau \leq -k$. 
    
Since $\bar M_\tau$ is convex, a straightforward geometric argument using the bound $|u| \leq C\delta_k^{1-2r}$ gives 
    \[|\nabla^{S^{n-1}} u | \leq C \delta_k^{1/2-r}\]
for $|z| \leq 8\delta_k^{-r}$ and $\tau \leq -k$. This completes the proof. 
\end{proof}

Next we obtain a stronger bound for $\nabla^{S^{n-1}} u$, and establish an estimate for the Hessian of $u$, in the region where $|z| \leq 5\delta_k^{-r}$ and $\tau \leq -k$.

\begin{lemma}\label{u derivatives in growing set}
When $k$ is sufficiently large we have $|\nabla u| \leq C\delta_k^{1-6r}$ and $|\nabla^2 u|\leq C \delta_k^{1-6r}$ for $|z| \leq 5\delta_k^{-r}$ and $\tau \leq -k$. 
\end{lemma}
\begin{proof}
Consider a time $\tau_0 \leq -k$, and set $t_0 = -e^{-\tau_0}$. Lemma~\ref{u in growing set} implies that 
    \[M_{t_0} \cap \{|z| \leq 8\delta_k^{-r} (-t_0)^{1/2}\} \subset \{(z,y): ||y|-(-2\gamma(0,1,\dots,1)t_0)^{1/2}| \leq \delta_k^{1-2r}(-t_0)^{1/2}\}.\]
In particular, if $k$ is large and $|z| \leq 7\delta_k^{-r}(-t_0)^{1/2}$, then
    \[B((z,0), \tfrac{99}{100}(-2\gamma(0,1,\dots,1)t_0)^{1/2}) \subset \Omega_{t_0}.\]
Using the interior curvature estimate of Proposition~\ref{interior curvature}, we conclude that 
    \[(-t_0)^{1/2}|A_{M_{t_0}}| \leq C\]
at every point in $M_{t_0} \cap \{|z| \leq 7\delta_k^{-r}\sqrt{-t_0}\}$. Moreover, the noncollapsing property implies
    \[(-t_0)^{1/2}G_{M_{t_0}} \geq \frac{1}{C}\]
at every point in $M_{t_0} \cap \{|z| \leq 7\delta_k^{-r}(-t_0)^{1/2}\}$
It follows that $|A_{\bar M_{\tau_0}}| \leq C$ and $G_{\bar M_{\tau_0}} \geq 1/C$ on $\bar M_{\tau_0} \cap \{|z| \leq 7\delta_k^{-r}\}$. Since $\tau_0 \leq -k$ was chosen arbitrarily we conclude that $|A_{\bar M_{\tau}}| \leq C$ and $G_{\bar M_\tau} \geq 1/C$ for $|z| \leq 7\delta_k^{-r}$ and $\tau \leq -k$ whenever $k$ is sufficiently large. 

Using $|A| \leq C$, a straightforward computation shows that when $k$ is sufficiently large we have 
        \[|\nabla^2 u| \leq C\]
for $|z| \leq 7\delta_k^{-r}$ and $\tau \leq -k$. Recall that the hypersurfaces $\bar M_\tau$ move with normal speed $-(G(x,\tau)-\tfrac{1}{2}\langle x, \nu(x,\tau))$. This implies that $u$ solves a fully nonlinear parabolic PDE which is convex or concave in $\nabla^2 u$ (depending on whether $\gamma$ is convex or concave) and uniformly parabolic by Lemma~\ref{lemma uniform parabolicity}. We also have
    \[|u_\tau| + |\nabla^2 u| + |\nabla u| + |u| \leq C\]
for $|z| \leq 7\delta_k^{-r}$ and $\tau \leq -k$. Therefore, the parabolic H\"{o}lder estimate for convex/concave fully nonlinear parabolic equations due to \cite{Krylov82} and the parabolic Schauder estimates imply that
        \[|\nabla^m u| \leq C(m)\]
for each integer $m \geq 2$ in the region where $|z| \leq 7\delta_k^{-r}-m$ and $\tau \leq -k$. In this step we also need the bound $\bar G \geq 1/C$, since this implies bounds on the higher derivatives of $\gamma$.

Given a smooth function $\varphi(z,\theta)$ which is compactly supported in the region $|z - z_0| < 2$, we have the interpolation inequality
    \[\sup_{|z-z_0| \leq 2} |\nabla \varphi| \leq C(m)\bigg(\sup_{|z-z_0| \leq 2} |\varphi|\bigg)^{1-1/m}\bigg(\sup_{|z-z_0| \leq 2} |\nabla^m \varphi|\bigg)^{1/m}.\]
We apply this with $\varphi(z,\theta) = \eta(z) u(z,\theta,\tau)$, where $\eta$ is an appropriate cutoff function, and use $|\nabla^j u| \leq C(m)$ for $j \leq m$ to obtain
    \[\sup_{|z-z_0| \leq 1} |\nabla u| \leq C(m)\sup_{|z-z_0| \leq 3} |u| + C(m)\bigg(\sup_{|z-z_0| \leq 3} |u|\bigg)^{1-1/m}\bigg(\sup_{|z-z_0| \leq 3} |\nabla^m u|\bigg)^{1/m}.\]
Fix an integer $m_0 \geq 1/r$. Consider a point $|z_0| \leq 7\delta_k^{-r}-m_0-3$ and a time $\tau \leq -k$. Applying the last inequality with $m = m_0$, and using Lemma~\ref{u in growing set} and $|\nabla^j u| \leq C(m_0)$ for $j \leq m_0$, we obtain
    \[\sup_{|z-z_0| \leq 1} |\nabla u|(z,\theta,\tau) \leq C\delta_k^{1-2r} + C\delta_k^{(1-r)(1-2r)} \leq C\delta_k^{1-4r}.\]
It follows that for $k$ sufficiently large we have $|\nabla u| \leq C\delta_k^{1-4r}$ for $|z| \leq 6\delta_k^{-r}$.

Now consider a point $|z_0| \leq 6\delta_k^{-r} - m_0 - 3$ and a time $\tau \leq -k$. The same kind argument as above yields 
    \begin{align*}
    \sup_{|z-z_0| \leq 1} &|\nabla^2 u|(z,\theta,\tau)\\
    &\leq \sup_{|z-z_0| \leq 3} |\nabla u| + \bigg(\sup_{|z-z_0| \leq 3} |\nabla u |\bigg)^{1-1/m_0}\bigg(\sup_{|z-z_0| \leq 4} |\nabla^{m_0+1} u|\bigg)^{1/m_0}.
    \end{align*}
Inserting $|\nabla u| \leq C\delta_k^{1-4r}$ on the right-hand side, we conclude that when $k$ is sufficiently large, $|\nabla^2 u| \leq C\delta_k^{1-6r}$ for $|z| \leq 5\delta_k^{-r}$ and $\tau \leq -k$. 
\end{proof}

We have shown that $u$ and its first two spatial derivatives are small in the region where $|z| \leq 5\delta_k^{-r}$ and $\tau \leq -k$. It follows that $u$ solves a linear parabolic equation in this region, up to quadratic error terms. 

\begin{lemma}\label{linearised rescaled flow}
When $k$ is sufficiently large we have 
    \[|u_\tau - \mathcal Lu| \leq C(|u|^2 + |\nabla u|^2 + |\nabla^2 u|^2)\]
for $|z| \leq 5\delta_k^{-r}$ and $\tau \leq -k$, where 
    \[\mathcal Lu := a\frac{\partial^2 u}{\partial z^2} + \frac{1}{2(n-1)} \Delta_{S^{n-1}} u  - \frac{1}{2} z \frac{\partial u}{\partial z}+ u\]
and $a := \dot \gamma^1(0,1,\dots,1)$.
\end{lemma}
\begin{proof}
Recall the notation $\sigma = \sqrt{2\gamma(0,1,\dots,1)}$. Given $(z,\theta) \in \mathbb{R}\times S^{n-1}$, let us write $x = (z,\sigma\theta)$ for the corresponding point in $\Sigma$ and $y = x + u(z,\theta,\tau)\theta$ for the corresponding point in $\bar M_\tau$. We have 
    \begin{align*}
        u_\tau (z,\theta,\tau) & = -\frac{1}{\nu_\Sigma(x) \cdot \nu(y)} \, G(y,\tau)+\frac{1}{2}\frac{1}{\nu_\Sigma(x) \cdot \nu(y)}  y \cdot \nu(y) .
    \end{align*}
For $\tau \leq -k$ and $|z| \leq 5\delta_k^{-r}$, Lemma~\ref{Lemma: expansions} gives
    \[ \nu(y)= \nu_\Sigma(x) - \nabla u(z,\theta,\tau) + O(|u|^2 + |\nabla u|^2).\]
We thus have 
    \begin{align*}
        u_\tau (z,\theta,\tau) & = -G(y,\tau) - \frac{1}{2}x \cdot \nabla u(z,\theta,\tau) + \frac{1}{2}u(z,\theta,\tau)  \\
        & + \frac{1}{2}x\cdot\nu_\Sigma(x) + O(|u|^2 + |\nabla u|^2).
    \end{align*}
Lemma~\ref{expansion for G} now gives 
    \begin{align*}
        u_\tau & = a\frac{\partial^2 u}{\partial z^2} + \frac{1}{2(n-1)}\Delta_{S^{n-1}} u  - \frac{1}{2}x \cdot \nabla u + u\\
        & + \frac{1}{2} x \cdot \nu_\Sigma  -G_\Sigma(x) + O(|u|^2 + |\nabla u|^2 + |\nabla^2 u|^2),
    \end{align*}
for $\tau \leq -k$ and $|z| \leq 5\delta_k^{-r}$. The claim follows upon insertion of the identities
    \[x \cdot \nabla u = z \frac{\partial u}{\partial z} \qquad \text{ and} \qquad \frac{1}{2} x \cdot \nu_\Sigma -G_\Sigma = 0.\] 
\end{proof} 

Using Lemma~\ref{linearised rescaled flow}, we obtain the following improved estimate for $\nabla u$ in the region $|z| \leq L+1$.

\begin{lemma}\label{gradient in small set}
When $k$ is sufficiently large we have $|\nabla u| \leq C\delta_k$ for $|z| \leq L+1$ and $\tau \leq -k$.
\end{lemma}
\begin{proof}
Let us define 
    \[\tilde{\mathcal L} u := a\frac{\partial^2 u}{\partial z^2} + \frac{1}{2(n-1)} \Delta_{S^{n-1}} u.\]
We then have
    \[|u_\tau - \tilde{\mathcal L} u| \leq \frac{1}{2}\bigg|z\frac{\partial u}{\partial z}\bigg| + |u| + C(|u|^2 + |\nabla u|^2 + |\nabla^2 u|^2)\]
for $|z| \leq 5\delta_k^{-r}$ and $\tau \leq -k$. Using Lemma~\ref{u in growing set} and Lemma~\ref{u derivatives in growing set} we conclude that 
    \[|u_\tau - \tilde{\mathcal L} u| \leq C \delta_k\]
for $|z| \leq L+2$ and $\tau \leq -k$. Given any $(z_0, \theta_0, \tau_0)$ such that $|z_0| \leq L+1$ and $\tau \leq -k$, there exists a small $d = d(n)$ such that
    \begin{align*}\label{interior gradient}
    |\nabla u|(z_0, \theta_0, \tau_0) \leq C d^{-1} \sup_{\tau_0 - d^2 \leq \tau \leq \tau_0} \sup_{|z-z_0|\leq d} |u| + C d \sup_{\tau_0 - d^2 \leq \tau \leq \tau_0} \sup_{|z-z_0|\leq d} |u_\tau - \tilde{\mathcal L}u|.
    \end{align*}
This is a standard estimate in the theory of linear parabolic PDE. Since $d$ depends only on $n$, we obtain
\[|\nabla u|(z_0,\theta_0,\tau_0) \leq C \delta_k.\]
Since $z_0$ was chosen arbitrarily in the region $|z| \leq L+1$, we conclude that $|\nabla u| \leq C\delta_k$ for $|z| \leq L+1$ and $\tau \leq -k$. 
\end{proof}

\subsection{Spectral analysis} Let us denote by $\mathcal H$ the space of functions $v : \mathbb{R}\times S^{n-1} \to \mathbb{R}$ such that 
    \[\|v\|_{\mathcal H}^2 := \int_{\mathbb{R}\times S^{n-1}} e^{-z^2/4a} v^2\,dzd\theta < \infty.\]
The operator $\mathcal L$ is symmetric with respect to the inner product $\langle \cdot , \cdot \rangle_{\mathcal H}$. Therefore, there exists an orthogonal basis for $\mathcal H$ consisting of eigenfunctions of $\mathcal{L}$. Let us write $H_k$ for the $k$-th Hermite polynomial, and fix a basis $Y_l^m$ of eigenfunctions for $\Delta_{S^{n-1}}$ such that
    \[\Delta_{S^{n-1}} Y_l^m = - l(l+n-2) Y_l^m.\]
A basis of eigenfunctions for $\mathcal L$ is then given by $H_k(z/2\sqrt{a})Y_l^m(\theta)$, where $k$ and $l$ range over the nonnegative integers. We have 
    \[\bigg(a \frac{\partial^2}{\partial z^2} - \frac{1}{2} z \frac{\partial}{\partial z}\bigg) H_k\bigg(\frac{z}{2\sqrt{a}}\bigg)= -\frac{k}{2}.\]
Therefore, the eigenvalues of $\mathcal L$ are given by $1-\frac{k}{2} - \frac{l(l+n-2)}{2(n-1)}$, where $k$ and $l$ are nonnegative integers. Up to scaling, the eigenfunctions with positive eigenvalues are of the form $1$, $z/\sqrt{a}$ and $\theta^i$, $1 \leq i \leq n$. The eigenfunctions with eigenvalue equal to zero are of the form $z^2/a - 1$ and $z \theta^i/\sqrt{a}$, $1 \leq i \leq n$. Let $P_+$, $P_0$ and $P_-$ denote the orthogonal projection operators onto the positive, zero and negative eigenspaces. These projections satisfy
    \begin{align}\label{eigenvalue bounds}
        &\langle \mathcal L P_+ v, P_+v \rangle_{\mathcal H} \geq \frac{1}{2} \|P_+ v\|_{\mathcal H}^2,\notag\\
        &\langle \mathcal L P_0 v, P_0 v \rangle_{\mathcal H} = 0,\notag\\
        &\langle \mathcal L P_- v, P_-v\rangle_{\mathcal H} \leq -\frac{1}{2} \|P_- v\|_{\mathcal H}^2.
    \end{align}
We are going to show that $P_+ u$ dominates $P_0 u$ and $P_- u$ in the region where $|z| \leq \delta_k^{-r}$ and $\tau \leq -k$. To this end, let $\xi:\mathbb{R} \to \mathbb{R}$ be a smooth cutoff function such that $\xi(s) = 1$ for $s \in [-1/2,1/2]$, $\xi(s) = 0$ for $|s| \geq 1$, and  $s \xi'(s) \leq 0$ for $s \in \mathbb{R}$. We define
    \begin{align*}
        \gamma_j &:= \sup_{\tau\in[-j-1,-j]} \int_{\mathbb{R}\times S^{n-1}} e^{-z^2/4a} |u(z,\theta,\tau) \chi(\delta_j^r z)|^2 \, dzd\theta\\
        \gamma_j^+ &:= \sup_{\tau\in[-j-1,-j]} \int_{\mathbb{R}\times S^{n-1}} e^{-z^2/4a} |P_+(u(z,\theta,\tau)\chi(\delta_j^r z))|^2 \, dzd\theta \\
        \gamma_j^0 &:= \sup_{\tau\in[-j-1,-j]} \int_{\mathbb{R}\times S^{n-1}} e^{-z^2/4a} |P_0(u(z,\theta,\tau)\chi(\delta_j^r z))|^2\, dzd\theta\\
        \gamma_j^- &:= \sup_{\tau\in[-j-1,-j]} \int_{\mathbb{R}\times S^{n-1}} e^{-z^2/4a} |P_-(u(z,\theta,\tau)\chi(\delta_j^r z))|^2 \, dzd\theta.
    \end{align*}
We have $C^{-1} \gamma_j \leq \gamma_j^+ + \gamma_j^0 + \gamma_j^- \leq C\gamma_j$, and Lemma~\ref{u in growing set} gives $\gamma_j \leq C \delta_j^{2-4r}$. In particular, $\gamma_j \to 0$ as $j \to \infty$.  

The following result is similar to \cite[Lemma~3.12]{Brendle}.

\begin{lemma}\label{system for modes}
We have 
    \begin{align*}
        \gamma_{j+1}^+ &\leq e^{-1}\gamma_j^+ + C(\gamma_j + \gamma_{j+1})^{1/2}\varepsilon_j^{1/2} + C\exp(-\delta_j^{-2r}/64a),\\
        | \gamma_{j+1}^0 - \gamma_j^0|&\leq C(\gamma_j + \gamma_{j+1})^{1/2}\varepsilon_j^{1/2} + C\exp(-\delta_j^{-2r}/64a),\\
        \gamma_{j+1}^-&\geq e\gamma_j^- - C(\gamma_j + \gamma_{j+1})^{1/2}\varepsilon_j^{1/2} - C\exp(-\delta_j^{-2r}/64a),
    \end{align*}
where 
    \[\varepsilon_j := \sup_{\tau \in [-j-2,-j]} \int_{|z|\leq\delta_j^{-r}} e^{-z^2/4a}|u_\tau - \mathcal L u|^2.\]
\end{lemma}
\begin{proof}
Let us write $\hat u(z,\theta,\tau) = u(z,\theta,\tau) \, \chi(\delta_j^r z)$. We first note that 
    \[\sup_{\tau \in [-j-2,-j]}\int_{\mathbb{R}\times S^{n-1}} e^{-z^2/4a} |\hat u|^2  \leq \gamma_j + \gamma_{j+1}.\]
Combining this with \eqref{eigenvalue bounds} and using H\"{o}lder's inequality we obtain
    \begin{align*}
        \frac{d}{d\tau} \int e^{-z^2/4a}|P_+\hat u|^2 &\geq \int e^{-z^2/4a} |P_+ \hat u|^2 - 2(\gamma_j + \gamma_{j+1})^{1/2}\bigg(\int e^{-z^2/4a}|\hat u_\tau - \mathcal L \hat u|^2 \bigg)^\frac{1}{2},\\
        \bigg| \frac{d}{d\tau} \int e^{-z^2/4a}|P_0\hat u|^2 \bigg|&\leq 2(\gamma_j + \gamma_{j+1})^{1/2}\bigg(\int e^{-z^2/4a}|\hat u_\tau - \mathcal L \hat u|^2 \bigg)^\frac{1}{2},\\
        \frac{d}{d\tau} \int e^{-z^2/4a}|P_-\hat u|^2&\leq -\int e^{-z^2/4a} |P_- \hat u|^2 + 2(\gamma_j + \gamma_{j+1})^{1/2}\bigg(\int e^{-z^2/4a}|\hat u_\tau - \mathcal L \hat u|^2 \bigg)^\frac{1}{2}
    \end{align*}
for $\tau \in [-j-2,-j]$. 

Since $\chi(s)$ is constant for $|s| \leq 1/2$ and $|s| \geq 1$, and $|\chi(s)|\leq 1$ for $s \in \mathbb{R}$, we have 
    \begin{align*}
        \int e^{-z^2/4a}|\hat u_\tau - \mathcal L \hat u|^2 &\leq C\int_{|z| \leq \delta_j^{-r}} e^{-z^2/4a} |u_\tau - \mathcal Lu|^2 + C\int_{\delta_j^{-r}/2 \leq |z| \leq \delta_j^{-r}} e^{-z^2/4a} (|u|^2 + |\nabla u|^2)\\
        &\leq C\varepsilon_j + C\exp(-\delta_j^{-2r}/32a)
    \end{align*}
for $\tau \in [-j-2,-j]$. 
It follows that 
    \begin{align*}
        \frac{d}{d\tau} \int e^{-z^2/4a}|P_+\hat u|^2 &\geq \int e^{-z^2/4a} |P_+ \hat u|^2 - C(\gamma_j + \gamma_{j+1})^{1/2}\varepsilon_j^{1/2} - C\exp(-\delta_j^{-2r}/64a),\\
        \bigg| \frac{d}{d\tau} \int e^{-z^2/4a}|P_0\hat u|^2 \bigg|&\leq 2(\gamma_j + \gamma_{j+1})^{1/2}\varepsilon_j^{1/2} + C\exp(-\delta_j^{-2r}/64a),\\
        \frac{d}{d\tau} \int e^{-z^2/4a}|P_-\hat u|^2&\leq -\int e^{-z^2/4a} |P_- \hat u|^2 + C(\gamma_j + \gamma_{j+1})^{1/2}\varepsilon_j^{1/2} + C\exp(-\delta_j^{-2r}/64a)
    \end{align*}
for $\tau \in [-j-2,-j]$. Integrating over $[\tau - 1, \tau]$ we obtain
    \begin{align*}
        &\int e^{-z^2/4a}|P_+\hat u(\cdot,\tau-1)|^2\\
        &\leq e^{-1}\int e^{-z^2/4a} |P_+ \hat u(\cdot,\tau)|^2 + C(\gamma_j + \gamma_{j+1})^{1/2}\varepsilon_j^{1/2} + C\exp(-\delta_j^{-2r}/64a),\\
        &\bigg| \int e^{-z^2/4a}|P_0\hat u(\cdot,\tau - 1)|^2 - \int e^{-z^2/4a}|P_0\hat u(\cdot,\tau)|^2 \bigg|\\
        &\leq C(\gamma_j + \gamma_{j+1})^{1/2}\varepsilon_j^{1/2} + C\exp(-\delta_j^{-2r}/64a),\\
        &\int e^{-z^2/4a}|P_-\hat u(\cdot,\tau-1)|^2\\
        &\geq e\int e^{-z^2/4a} |P_- \hat u(\cdot,\tau)|^2 - C(\gamma_j + \gamma_{j+1})^{1/2}\varepsilon_j^{1/2} - C\exp(-\delta_j^{-2r}/64a).
    \end{align*}
for $\tau \in [-j-1,-j]$. 
    
Define $\tilde u(z,\tau) = u(z,\tau)\chi(\delta_{j+1}^{r}z)$. Observing that 
    \begin{align*}
        \int e^{-z^2/4a} |\tilde u(z,\tau-1)&-\hat u(z,\tau-1)|^2\leq C\exp(-\delta_j^{-2r}/64a)
    \end{align*}
for all $\tau \in [-j-1, j]$, we conclude that 
    \begin{align*}
        &\int e^{-z^2/4a}|P_+\tilde u(\cdot,\tau-1)|^2 - e^{-1}\int e^{-z^2/4a} |P_+ \hat u(\cdot,\tau)|^2\\
        &\hspace{2cm}\leq C(\gamma_j + \gamma_{j+1})^{1/2}\varepsilon_j^{1/2} + C\exp(-\delta_j^{-2r}/64a),\\
        &\bigg| \int e^{-z^2/4a}|P_0\tilde u(\cdot,\tau - 1)|^2 - \int e^{-z^2/4a}|P_0\hat u(\cdot,\tau)|^2 \bigg|\\
        &\hspace{2cm}\leq C(\gamma_j + \gamma_{j+1})^{1/2}\varepsilon_j^{1/2} + C\exp(-\delta_j^{-2r}/64a),\\
        &\int e^{-z^2/4a}|P_-\tilde u(\cdot,\tau-1)|^2 - e\int e^{-z^2/4a} |P_- \hat u(\cdot,\tau)|^2\\
        &\hspace{2cm}\geq - C(\gamma_j + \gamma_{j+1})^{1/2}\varepsilon_j^{1/2} - C\exp(-\delta_j^{-2r}/64a).
    \end{align*}
for $\tau \in [-j-1,-1]$. Taking the supremum over $\tau \in [-j-1, -j]$, we arrive at the claim.
\end{proof}

We now define
    \[\Gamma_k := \sup_{j \geq k} \gamma_j, \qquad \Gamma_k^+ := \sup_{j\geq k} \gamma_j^+, \qquad \Gamma_k^0 := \sup_{j\geq k} \gamma_j^0, \qquad \Gamma_k^- := \sup_{j\geq k} \gamma_j^-.\]
We have $C^{-1}\Gamma_k \leq \Gamma_k^+ + \Gamma_k^0 + \Gamma_k^- \leq C\Gamma_k$ and $\Gamma_k \to 0$ as $k \to \infty$. Since 
    \begin{align*}
    \varepsilon_j &= \sup_{\tau \in [-j-2,-j]}\int_{|z|\leq \delta_j^{-r}} e^{-z^2/4a} |u_\tau - \mathcal Lu|^2\\
    &\leq C \sup_{\tau \in [-j-2,-j]} \Big(\sup_{|z| \leq \delta_j^{-r}} |u|^4 + |\nabla u|^4 + |\nabla^2 u|^4\Big),
    \end{align*}
as a consequence of Lemma~\ref{u in growing set} and Lemma~\ref{u derivatives in growing set} we have 
    \[\varepsilon_j \leq C \delta_j^{4-24r}\]
whenever $j$ is sufficiently large. In particular, 
    \[E_k := \sup_{j \geq k} \varepsilon_j\]
satisfies $E_k \leq CD_k^{4-24r}$, where $D_k := \sup_{j \geq k} \delta_j$. Note that $D_k \to 0$ as $k \to \infty$. Lemma~\ref{system for modes} immediately implies
    \begin{align*}
        \Gamma_{k+1}^+ &\leq e^{-1} \Gamma_k^+ + C \Gamma_k^{1/2} E_k^{1/2} + C\exp(-\delta_k^{-2r}/64a),\\
        |\Gamma_{k+1}^0 - \Gamma_k^0|&\leq C\Gamma_k^{1/2}E_k^{1/2} + C\exp(-\delta_k^{-2r}/64a),\\
        \Gamma_{k+1}^- &\geq e \Gamma_k^- -C\Gamma_k^{1/2}E_k^{1/2} - C\exp(-\delta_k^{-2r}/64a).
    \end{align*}  
    
In the next lemma we improve our estimate for the error term $E_k$. 

\begin{lemma}\label{error estimates}
For sufficiently large $k$ we have $\delta_k \leq C \Gamma_k^{1/2}$ and 
\[E_k \leq C D_k^{2-24r}\Gamma_k.\]
\end{lemma}
\begin{proof}
Since 
    \[\sup_{|z| \leq L} |u(z,\theta,\tau)| \to 0\]
as $\tau \to -\infty$, for each $k$ there is a point $(z_0, \theta_0, \tau_0)$ such that $|z_0| \leq L$, $\tau_0 \leq -k$, and 
    \[|u(z_0, \theta_0, \tau_0)| = \sup_{\tau \leq -k} \sup_{|z| \leq L} |u(z,\theta,\tau)| = \delta_k.\]
Lemma~\ref{gradient in small set} provides a constant $C_0$ such that $|\nabla u| \leq C_0 \delta_k$ for $|z| \leq L+1$ and $\tau \leq -k$. Using this estimate we obtain 
    \[\inf\Big\{|u(z,\theta,\tau_0)| : \dist_{\mathbb{R}\times S^{n-1}} ((z,\theta), (z_0,\theta_0)) \leq  \tfrac{1}{100}\min\{1,C_0^{-1}\}\Big\} \geq \frac{1}{C} \delta_k.\]
It follows that 
    \[\delta_k^2 \leq C\int_{|z| \leq L+1} |u|^2(\cdot,\tau_0) \,dzd\theta.\]
There exists an integer $j \geq k$ such that $\tau_0 \in [-j-1, -j]$. For this $j$, as a result of the last inequality we have 
    \[\delta_k^2 \leq C\sup_{\tau \in [-j-1,-j]} \int_{\mathbb{R}\times S^{n-1}} e^{-z^2/4a} |u(z,\theta,\tau) \chi(\delta_j^r z)|^2 \, dzd\theta = C\gamma_j,\]
at least if $k$ is sufficiently large. Therefore, 
    \[\delta_k^2 \leq C\sup_{j \geq k} \gamma_j \leq C \Gamma_k\]
for every sufficiently large $k$. Combining this inequality with $E_k \leq CD_k^{4-24r}$, and assuming $k$ is sufficiently large, we obtain
    \[E_k \leq C D_k^{4-24r} = CD_k^{2-24r}\sup_{j\geq k} \delta_j^2 \leq C D_k^{2-24r} \sup_{j \geq k}\Gamma_j \leq  CD_k^{2-24r} \Gamma_k.\]
Here we have used the fact that $\Gamma_k$ is nonincreasing in $k$. This completes the proof.
\end{proof}

Lemma~\ref{error estimates} tells us that $\delta_k \leq C \Gamma_k^{1/2}$ and $E_k \leq C D_k \Gamma_k$. Moreover, for sufficiently large $k$ we have
    \[\exp(-\delta_k^{-2r}/64a) \leq \delta_k^{5/2} \leq C \delta_k^{1/2} \Gamma_k \leq CD_k^{1/2}\Gamma_k .\]
Consequently, for large $k$ we have 
    \begin{align*}
        \Gamma_{k+1}^+ &\leq e^{-1} \Gamma_k^+ + CD_k^{1/2}\Gamma_k\\
        |\Gamma_{k+1}^0 - \Gamma_k^0|&\leq CD_k^{1/2} \Gamma_k \\
        \Gamma_{k+1}^- &\geq e \Gamma_k^- - CD_k^{1/2}\Gamma_k,
    \end{align*}   
and in particular 
    \begin{align*}
        \Gamma_{k+1}^+ &\leq e^{-1} \Gamma_k^+ o(1)\Gamma_k\\
        |\Gamma_{k+1}^0 - \Gamma_k^0|&\leq o(1)\Gamma_k \\
        \Gamma_{k+1}^- &\geq e \Gamma_k^- - o(1)\Gamma_k
    \end{align*}   
as $k \to \infty$. This system of inequalities implies that the decomposition $\Gamma_k = \Gamma_k^+ + \Gamma_k^0 + \Gamma_k^-$ is dominated either by $\Gamma_k^+$ or $\Gamma_k^0$ as $k \to \infty$. For the proof of this statement, which is a variation on an argument by Merle and Zaag \cite{Merle_Zaag}, we refer to Lemma~3.1 in \cite{Brendle}.

\begin{lemma}\label{Merle Zaag}
As $k \to \infty$ we either have $\Gamma_k^0 + \Gamma_k^- \leq o(1) \Gamma_k^+$ or $\Gamma_k^+ + \Gamma_k^- \leq o(1) \Gamma_k^0$.
\end{lemma}

We now establish that $\Gamma_k^0 + \Gamma_k^- \leq o(1) \Gamma_k^+$. The argument makes use of the fact that $\bar M_\tau$ is noncompact---recall that this allowed us to assume $\tfrac{\partial u}{\partial z} < 0$.

\begin{proposition}\label{positive mode dominates}
As $k\to\infty$ we have $\Gamma_k^0 + \Gamma_k^- \leq o(1) \Gamma_k^+$.
\end{proposition}
\begin{proof}
By Lemma~\ref{Merle Zaag}, if the claim is false then $\Gamma_k^+ + \Gamma_k^- \leq o(1) \Gamma_k^0$. Suppose this is the case. Fix a sequence $\tau_k \leq -k$ such that 
    \[\hat u_k(z,\theta) := \chi(\delta_k^rz) u(z,\theta,\tau_k)\]
satisfies
    \[\|\hat u_k\|_{\mathcal H} = \Gamma_k^{1/2}.\]

Since $\Gamma_k^+ + \Gamma_k^- \leq o(1) \Gamma_k^0$, the sequence $\hat u_k/\|\hat u_k\|_{\mathcal H}$ converges to a limit $\hat u(z,\theta)$ with respect to the norm $\|\cdot\|_{\mathcal H}$. Passing to a further subsequence, we may assume that $\hat u_k/\|\hat u_k\|_{\mathcal H} \to \hat u$ pointwise almost everywhere in $\mathbb{R}\times S^{n-1}$. Since $\frac{\partial u}{\partial z} < 0$ we conclude that for each fixed $\theta$ the function $z \mapsto \hat u(z,\theta)$ is nonincreasing. On the other hand, the limit $\hat u$ lies in the zero-eigenspace of $\mathcal L$, which is to say that
    \[\hat u(z,\theta) = \alpha \bigg(\frac{z^2}{a} -2\bigg) + \beta_i  
 \frac{z}{\sqrt{a}} \theta^i,\]
for some constants $\alpha$ and $\beta_i$. The right-hand side is monotone in $z$ if and only if $\alpha = 0$ and $\beta_i = 0$ for each $1 \leq i \leq n$. With this we have reached a contradiction, since $\|\hat u\|_{\mathcal H} = 1$.
\end{proof}

We are now able to prove Proposition~\ref{asymptotic for u}.

\begin{proof}[Proof of Proposition~\ref{asymptotic for u}]
Using Proposition~\ref{positive mode dominates} we conclude that
    \[\Gamma_{k+1}^+ \leq e^{-1} \Gamma_k^+ + CD_k^{1/2}\Gamma_k^+ \leq e^{-1/2} \Gamma_k^+\]
whenever $k$ is sufficiently large. Iterating this estimate we find that $\Gamma_k^+ \leq O(e^{-k/2})$, and hence $\Gamma_k \leq O(e^{-k/2})$. Using the estimate 
    \[D_k = \sup_{j \geq k} \delta_j \leq C\sup_{j \geq k} \Gamma_j^{1/2} \leq C\Gamma_k^{1/2}\]
we obtain $D_k \leq O(e^{-k/4})$ and hence
    \[\Gamma_{k+1}^+ \leq e^{-1} \Gamma_k^+ + Ce^{-k/8}\Gamma_k^+.\]
Iterating this estimate we obtain $\Gamma_k^+ \leq O(e^{-k})$ and hence $\Gamma_k \leq O(e^{-k})$.

We have 
    \[\sup_{\tau \leq -k} \sup_{|z| \leq L} |u| = \delta_k.\]
Using $\delta_k \leq C\Gamma_k^{1/2}$ we conclude that 
    \[\sup_{\tau \leq -k} \sup_{|z| \leq L} |u| \leq O(e^{-k/2}).\]
The claim follows.
\end{proof}

We conclude this section with an important consequence of Proposition~\ref{asymptotic for u}---namely, that $\max_{M_t} G$ is bounded from below by a positive constant which is independent of time. 

A straightforward computation using Proposition~\ref{barrier} shows that if $L > L_0$ is sufficiently large depending on $n$ and $\gamma$, and $a$ is sufficiently large depending on $L$, then we have 
    \begin{equation}\label{Psi decreases}
        \Psi_a(L) - \Psi_a(L-1) < -a^{-2}.
    \end{equation}
Moreover, Proposition~\ref{asymptotic for u} implies there is a positive constant $K = K(L)$ such that
    \[M_t \cap \{|z| \leq L (-t)^{1/2}\} \subset \{(z,y): ||y| - (-2\gamma(0,1,\dots1) t)^{1/2}| \leq K\}\]
whenever $-t$ is sufficiently large. 

\begin{proposition}\label{curvature at tip}
For $K$ as above we have $\max_{M_t} G \geq -1/8K$ for every $t \leq 0$.
\end{proposition}
\begin{proof}
The proof is a direct adaptation of \cite[Proposition~3.3]{Brendle_Choi_a}. We make use of the hypersurfaces 
    \[\Sigma_a = \{(y,z) : y = \Psi_a(z)\theta, \, \theta \in S^{n-1}, \, L_0 \leq z \leq a\}\]
constructed in Section~\ref{barrier section}. The hypersurfaces
    \begin{align*}
        \Sigma_{a,t} &:= (-t)^{1/2}\Sigma_a - (0, Ka^2)\\
        &= \{(y,z) : y=(-t)^{1/2} \Psi_a((-t)^{-1/2}(z + Ka^2)) \theta, \, \theta \in S^{n-1},\\
        &\hspace{6cm} (-t)^{1/2} L_0-Ka^2\leq z\leq (-t)^{1/2}a- Ka^2\}
    \end{align*}
form a $G$-flow. As $t \to -\infty$ the rescaled hypersurfaces $(-t)^{-1/2}M_t$ converge in $C^{\infty}_{\loc}$ to the cylinder $\{|y|^2 = 2\gamma(0,1,\dots,1)\}$, and the hypersurfaces 
\[(-t)^{-1/2}(\Sigma_{a,t}\cap\{z \geq (-t)^{1/2}L_0\})\]
converge to a compact subset of $\{|y|^2 < 2\gamma(0,1,\dots,1)\}$. So when $-t$ is sufficiently large depending on $a$ we have
    \[\Sigma_{a,t}\cap\{z \geq (-t)^{1/2}L_0\} \subset \Omega_t,\]
where $\Omega_t$ is the open region bounded by $M_t$.

Since $\Psi_a(z)$ is concave, as a consequence of \eqref{Psi decreases} we have
    \begin{align*}
    \Psi_a(L + Ka^2 (-t)^{-1/2}) 
    &\leq \Psi_a(L) + (\Psi_a(L) - \Psi_a(L-1)) Ka^2(-t)^{-1/2}\\
    &< (2\gamma(0,1,\dots,1))^{1/2} - K(-t)^{-1/2}
    \end{align*}
for $-t \geq 4 K^2 a^2$ and $a$ sufficiently large. On the other hand, 
    \[\{(y,z) : |y| \leq (-2\gamma(0,1,\dots,1)t)^{1/2} - K, \, z = (-t)^{1/2} L\} \subset \Omega_t \cap \{z = (-t)^{1/2}L\}\]
when $-t$ is sufficiently large. It follows that for $-t \geq 4K^2 a^2$ and $a$ sufficiently large we have
    \[\Sigma_{a,t} \cap \{z = (-t)^{1/2}L\} \subset \Omega_t \cap \{z = (-t)^{1/2}L\}.\]
Therefore, the parabolic maximum principle implies that
    \[\Sigma_{a,t} \cap \{z \geq (-t)^{1/2}L\} \subset \Omega_t\]
for $-t\geq 4K^2 a^2$ and $a$ sufficiently large. When $t = -4K^2 a^2$ the tip of $\Sigma_{a,t}$ is located at $(0,Ka^2) = (0,-t/4K)$. From this we conclude that $M_t \cap \{z \geq -t/4K\}$ is nonempty for all sufficiently large $-t$.

It follows that $\sup_{M_t} G \geq - 1/8K$ at some sequence of times tending to $-\infty$. On the other hand, the differential Harnack inequality implies that $G$ is strictly increasing in time at every point. The claim follows.
\end{proof}


\section{The Neck Improvement Theorem}\label{neck improvement section}

In this section we generalise the Neck Improvement Theorem for solutions of mean curvature flow \cite{Brendle_Choi_b} to hypersurfaces moving with speed $G$. 

To begin with we recall the following definition from \cite{Brendle_Choi_b}.

\begin{definition}
A set of vector fields $\mathcal{K}=\lbrace K_\alpha : 1\leq \alpha\leq \tfrac{1}{2}n(n-1)\rbrace$ on $\mathbb{R}^{n+1}$ is a normalised set of rotation vector fields if there exists an orthonormal basis $\{J_\alpha : 1\leq \alpha\leq \tfrac{1}{2}n(n-1)\}$ of $\operatorname{so}(n) \subset \operatorname{so}(n+1)$, a matrix $S\in O(n+1)$ and a point $q\in\mathbb{R}^{n+1}$ such that
\begin{align*}
     K_\alpha(x)=SJ_\alpha S^{-1}(x-q).
\end{align*}
\end{definition}

Next we define a notion of quantitative almost-rotational symmetry.

\begin{definition}\label{symmetric point} Consider a $G$-flow $M_t$. A point $\bar x \in M_{\bar t}$ is called $\varepsilon$-symmetric if there exists a normalised set of rotation vector fields $\mathcal{K} = \{K_\alpha : 1 \leq \alpha \leq \tfrac{1}{2}n(n-1)\}$ such that 
    \[\max_\alpha |K_\alpha| \, G \leq 10\gamma(0,1,\dots,1)\]
at the point $(\bar{x},\bar{t})$ and 
    \[\max_\alpha |\langle K_\alpha,\nu \rangle| \, G \leq \varepsilon\]
in the neighbourhood $\hat{\mathcal{P}}(\bar{x},\bar{t},10,100)$.
\end{definition}

We now establish the Neck Improvement Theorem for $G$-flows. In fact, the proof reduces to the proof for mean curvature flow, since up to a change of variables the linearised $G$-flow over a shrinking neck is identical to the linearised mean curvature flow over a shrinking neck. We compute the linearisation and then refer to \cite{Brendle_Choi_b} for the rest of the argument. 

\begin{theorem}
Let $M_t$ be a uniformly parabolic $G$-flow. There exists a large constant $L$ (depending on $n$, $\gamma$ and $\dist(\lambda/|\lambda|, \partial \Gamma)$), and a small constant $\varepsilon_1$ depending on $L$, with the following property. Let $(\bar x, \bar t)$ be a point in spacetime such that every point in the parabolic neighbourhood $\hat {\mathcal P}(\bar x, \bar t, L, L^2)$ lies at the center of an $\varepsilon_1$-neck and is $\varepsilon$-symmetric for some $\varepsilon \leq \varepsilon_1$. Then $(\bar x, \bar t)$ is $\tfrac{\varepsilon}{2}$-symmetric. 
\end{theorem}
\begin{proof}
Without loss of generality, we assume $\bar t = -1$. After parabolically rescaling we can arrange that $G(\bar x, -1) = (\gamma(0,1,\dots,1)/2)^{1/2}$. Since every point in the parabolic neighbourhood $\hat {\mathcal P}(\bar x, \bar t, L, L^2)$ lies at the center of an $\varepsilon_1$-neck, we can then approximate $M_t$ by the cylinder 
    \[S^{n-1}((-2\gamma(0,1,\dots,1)t)^{1/2}) \times \mathbb{R}\]
in $\hat {\mathcal P}(\bar x, \bar t, L, L^2)$, up to errors which are bounded by $C(L)\varepsilon_1$ in the $C^{10}$-norm.

For every $(x_0,t_0) \in \hat {\mathcal P}(\bar x, -1, L,L^2)$ there exists a normalised set of rotation vector fields $\{K_\alpha^{(x_0,t_0)}: 1 \leq \alpha \leq \frac{1}{2}n(n-1)\}$ such that $\max_\alpha|K_\alpha|G\leq 10\gamma(0,1,\dots,1)$ at $(x_0,t_0)$ and 
    \[\max_\alpha|\langle K_\alpha^{(x_0,t_0)}, \nu\rangle|G \leq \varepsilon\]
in $\hat{\mathcal P}(x_0, t_0, 10, 100)$. As in \cite[Theorem~4.4]{Brendle_Choi_b}, we may assume without loss of generality that 
\begin{align}\label{eq: bounded sup rot VFs}
   \sup_{B(0,10\gamma(0,1\dots,1)L)}\max_\alpha|K^{(\bar x, -1)}_\alpha-K^{(x_0,t_0)}_\alpha|\leq C(L)\varepsilon 
\end{align}
for each $(x_0,t_0) \in \hat{\mathcal P}(\bar x, -1, L,L^2)$. This implies that 
    \[|\langle K^{(\bar x, -1)}, \nu \rangle| \leq C(L)\varepsilon\]
at each point in $\hat{\mathcal P}(\bar x, -1, L,L^2)$. We may also assume that the common axis of rotation of the vector fields $K_\alpha^{(\bar x, -1)}$ is the $x_{n+1}$-axis (that is, $K_\alpha = J_\alpha x$ for some orthonormal basis $\{J_\alpha : 1 \leq \alpha \leq \tfrac{1}{2}n(n-1)\}$ of $\operatorname{so}(n) \subset \operatorname{so}(n+1)$) and that $\bar x$ lies in the hyperplane $\{x_{n+1} = 0\}$. Let us define $\bar K_\alpha := K^{(\bar x, -1)}_\alpha$.

Consider $M_t$ as a graph over the $x_{n+1}$-axis, which we equip with the coordinate $z$, so that
\begin{align*}
    \Big\lbrace \big( r(\theta, z,t)\theta,z\big):\theta\in S^{n-1}, \, z\in\left[-\tfrac{3}{4}L, \tfrac{3}{4}L\right]\Big\rbrace \subset M_t.
\end{align*}
Exactly as in \cite{Brendle_Choi_b}, we have  
    \begin{equation}\label{div r}
        |\langle \bar K_\alpha, \nu\rangle + \Div_{S^{n-1}}(r(\theta,z,t)J_\alpha\theta)| \leq C(L)\varepsilon_1\varepsilon
    \end{equation}
in $\hat{\mathcal P}(\bar x, -1, L,L^2)$.

Fix $\alpha$ and set $\bar u_\alpha := \langle \overline{K}_\alpha, \nu \rangle$. With respect to any orthonormal frame on $M_t$, $\bar u_\alpha$ satisfies
    \[\frac{\partial \bar u_\alpha}{\partial t} = \dot \gamma^{ij}(\nabla_i \nabla_j \bar u_\alpha +  A_i^kA_{kj} \bar u_\alpha).\]
Define $a:=\dot\gamma^1(0,1,\dots,1)$. By Lemma~\ref{trace expansion} we have
    \begin{align*}
    \bigg|\dot\gamma^{ij}\nabla_i \nabla_j \bar u_\alpha - a\frac{\partial^2 \bar u_\alpha}{\partial z^2} - \frac{1}{(-2(n-1)t)}\Delta_{S^{n-1}} \bar u_\alpha\bigg| \leq C(L)\varepsilon_1(|\nabla \bar u_\alpha| + |\nabla^2 \bar u_\alpha|).
    \end{align*}
Standard interior estimates for parabolic equations imply $|\nabla \bar u_\alpha| + |\nabla^2 \bar u_\alpha| \leq C(L)\varepsilon$, and hence 
    \[\bigg|\dot\gamma^{ij}\nabla_i \nabla_j \bar u_\alpha - a\frac{\partial^2 \bar u_\alpha}{\partial z^2} - \frac{1}{(-2(n-1)t)}\Delta_{S^{n-1}} \bar u_\alpha\bigg| \leq C(L)\varepsilon_1\varepsilon\]
for $|z| \leq \frac{L}{4}$ and $t \in [-\frac{L^2}{16}, -1]$. Lemma~\ref{Lemma: expansions} and Lemma~\ref{trace expansion} imply
    \[\bigg|\dot\gamma^{ij} A_i^kA_{kj} \bar u_\alpha - \frac{1}{(-2t)}\bar u_\alpha\bigg| \leq C(L)\varepsilon_1 \varepsilon.\]
Putting this all together, we obtain 
    \[\bigg|\frac{\partial \bar u_\alpha}{\partial t} - a\frac{\partial^2 \bar u_\alpha}{\partial z^2} - \frac{1}{(-2(n-1)t)}\Delta_{S^{n-1}} \bar u_\alpha - \frac{1}{(-2t)}\bar u_\alpha\bigg| \leq C(L)\varepsilon_1\varepsilon\]
for $|z| \leq \frac{L}{4}$ and $t \in [-\frac{L^2}{16}, -1]$. 

Let us define $u_\alpha(\theta, z, t) := \bar u_\alpha (\theta, a^{-1/2}z, t)$. We then have 
    \begin{equation}\label{linearised u_alpha}
        \bigg|\frac{\partial u_\alpha}{\partial t} - \frac{\partial^2 u_\alpha}{\partial z^2} - \frac{1}{(-2(n-1)t)}\Delta_{S^{n-1}} u_\alpha - \frac{1}{(-2t)} u_\alpha\bigg| \leq C(L)\varepsilon_1\varepsilon
    \end{equation}
for $|z| \leq \tfrac{a^{1/2}}{4} L$ and $t \in [-\frac{L^2}{16}, -1]$. In addition, \eqref{div r} gives
    \begin{equation}\label{div u_alpha}
        |u_\alpha(\theta,z,t) + \Div_{S^{n-1}}(r(\theta,z,t)J_\alpha\theta)| \leq C(L)\varepsilon_1\varepsilon
    \end{equation}
for $|z| \leq a^{1/2} L$ and $t \in [-L^2,-1]$.

Notice that the linear operator on the left-hand side of \eqref{linearised u_alpha} is completely independent of the speed $\gamma$. We may therefore appeal to Step~5 of \cite[Theorem~4.4]{Brendle_Choi_b} to conclude that, for each $1 \leq \alpha \leq \tfrac{1}{2}n(n-1)$, there are vectors $A_\alpha$ and $\tilde B_\alpha$ such that $|A_\alpha| \leq C(L)\varepsilon$, $|\tilde B_\alpha|\leq C(L)\varepsilon$ and 
    \[|u_\alpha(\theta,z,t) - \langle A_\alpha, \theta\rangle - \langle \tilde B_\alpha, \theta\rangle z| \leq C L^{-1/(n-1)}\varepsilon + C(L)\varepsilon_1\varepsilon\]
for $|z| \leq 20\gamma(0,1,\dots,1)a^{1/2}$ and $t \in [-400\gamma(0,1,\dots,1)^2,-1]$. Undoing the rescaling $z \to a^{-1/2} z$ and setting $B_\alpha := a^{1/2}\tilde B_\alpha$, we obtain
    \[|\langle \bar K_\alpha,\nu \rangle - \langle A_\alpha, \theta\rangle - \langle B_\alpha, \theta\rangle z| \leq C L^{-1/(n-1)}\varepsilon + C(L)\varepsilon_1\varepsilon\]
for $|z| \leq 20\gamma(0,1,\dots,1)$ and $t \in [-400\gamma(0,1,\dots,1)^2,-1]$. Step~7 and Step~8 of \cite{Brendle_Choi_b} show that if $L$ is sufficiently large and $\varepsilon_1$ is sufficiently small (depending on $L$), then $(\bar x, -1)$ is $\tfrac{\varepsilon}{2}$-symmetric. This completes the proof. 
\end{proof}


\section{Rotational symmetry}\label{rotational symmetry section}

Let $M_t = \partial\Omega_t$, $t \in (-\infty, 0]$, be a strictly convex ancient $G$-flow which is noncollapsing and uniformly two-convex. Suppose in addition that $M_t$ is noncompact. In this section we prove that $M_t$ is rotationally symmetric. The arguments follow \cite[Section~5]{Brendle_Choi_a} and \cite[Section~5]{Brendle_Choi_b}. We make use of the Neck Improvement Theorem and the lower bound for the maximum of $G$ over $M_t$ established in Proposition~\ref{curvature at tip}. 

Throughout this section $C$ denotes some large constant which is independent of $t$.

\begin{lemma}
For sufficiently large $-t$ the maximum of $G$ over $M_t$ is attained at some unique point $p_t$. The Hessian of $G$ is negative-definite at $p_t$, and hence $p_t$ is smooth as a function of $t$. 
\end{lemma}
\begin{proof}
When $-t$ is sufficiently large, $M_t \cap B(0,(-10\gamma(0,1,\dots,1)t)^{1/2})$ is close to the cylinder $\mathbb{R} \times S^{n-1}((-2\gamma(0,1,\dots,1)t)^{1/2})$. As a consequence, in the complement of $B(0,(-10\gamma(0,1,\dots,1)t)^{1/2})$, $M_t$ consists of two connected components. One of these is compact and the other is noncompact. By Lemma~\ref{neck-cap}, on the noncompact component we have $G \leq C(-t)^{-1/2}$. Since we have shown (in Proposition~\ref{curvature at tip}) that $\max_{M_t} G$ is bounded from below by a positive constant, it follows that $\max_{M_t} G$ is attained by some $p_t$ in the compact component. 

Now consider an arbitrary sequence $t_k \to \infty$. After passing to a subsequence, the sequence of shifted flows $M_{t+t_k} - p_{t_k}$ converges in $C^\infty_{\loc}$ to an ancient $G$-flow. Moreover, on the limit, the spacetime maximum of $G$ is attained at $(0,0)$. Therefore, by Proposition~\ref{differential harnack}, the limit is a translating soliton, and hence is the bowl soliton by \cite{Bourni--Langford} (cf.~\cite{Haslhofer_Bowl}). The claim follows.
\end{proof}

Let $\varepsilon_1$ and $L$ be the constants in the Neck Improvement Theorem. Given that $\max_{M_t} G$ is bounded from below, Proposition~\ref{lower curvature} implies that $G(x,t)^{-1} = o(|x - p_t|)$ uniformly in $t$. Therefore, by Lemma~\ref{neck-cap}, there is a large constant $\Lambda$ with the following property: every point $x \in M_t$ satisfying $|x - p_t| \geq \Lambda$ lies at the center of an $\varepsilon_1$-neck and satisfies $G(x,t)|x-p_t| \geq 10^6 \gamma(0,1,\dots,1) L$.

\begin{lemma}\label{tip to neck growing}
There is a time $T < 0$ with the following property. If $\bar t\leq T$ and $\bar x \in M_{\bar t}$ is such that $|\bar x - p_{\bar t}| \geq \Lambda$, then $|\bar x - p_t| \geq |\bar x - p_{\bar t}|$ for all $t \leq \bar t$.  
\end{lemma}
\begin{proof}
We have seen that, when $-t$ is large, $M_t$ is extremely close to the bowl soliton near $p_t$. In particular, $\frac{d}{dt} p_t$ is almost parallel to $-\nu(p_t, t)$. It follows that there is a time $T \leq 0$ such that  
    \[\bigg\langle x - p_t, \frac{d}{dt} p_t\bigg\rangle > 0\]
whenever $t \leq T$, $x \in M_t \cup \Omega_t$ and $|x - p_t| \geq \Lambda$. Consequently, 
    \[\frac{d}{dt}|x - p_t| = -\bigg\langle\frac{x-p_t}{|x-p_t|}, \frac{d}{dt}p_t \bigg\rangle < 0\]
whenever $t \leq T$, $x \in M_t \cup \Omega_t$ and $|x - p_t| \geq \Lambda$. 

Consider a time $\bar t \leq T$ and a point $\bar x \in M_{\bar t}$ such that $|\bar x - p_{\bar t}| \geq \Lambda$. We claim that $|\bar x - p_t| \geq |\bar x - p_{\bar t}|$ for all $t \leq \bar t$. If this is not the case then 
    \[\tilde t := \sup\{t \leq \bar t : |\bar x - p_t| < |\bar x - p_{\bar t}|\}\]
is finite. We have $\tilde t < \bar t$ and $|\bar x - p_t| \geq |\bar x - p_{\bar t}| \geq \Lambda$ for $t \in [\tilde t, \bar t]$. But our choice of $T$ ensures that $\frac{d}{dt}|\bar x - p_t|<0$ for $t \in [\tilde t, \bar t]$, and hence $|\bar x - p_{\tilde t}| > |\bar x - p_{\bar t}|$. This contradicts the definition of $\tilde t$, so $T$ has the desired property. 
\end{proof}

Next we employ the Neck Improvement Theorem to show that the end of $M_t$ becomes increasingly symmetric at large distances.

\begin{proposition}\label{end symmetric}
If $t \leq T$, and $x \in M_t$ satisfies $|x - p_t| \geq 2^{j/400}\Lambda$, then $(x,t)$ is $2^{-j}\varepsilon_1$-symmetric. 
\end{proposition}
\begin{proof}
We argue by induction on $j$. The assertion is true for $j = 0$, by virtue of the fact that if $t \leq T$ and $|x - p_t|\geq\Lambda$, then $(x,t)$ lies at the center of an $\varepsilon_1$-neck. Suppose that $j \geq 1$ and the assertion is true for $j - 1$. If the assertion fails for $j$ then there is a time $\bar t \leq T$ and a point $\bar x \in M_{\bar t}$ such that $|\bar x - p_{\bar t}|\geq 2^{j/400}\Lambda$ and yet $(\bar x, \bar t)$ fails to be $2^{-j}\varepsilon_1$-symmetric. Appealing to the Neck Improvement Theorem, we conclude that there is a point $(x,t) \in \hat{\mathcal P}(\bar x, \bar t, L, L^2)$ such that either $(x,t)$ is not $2^{-j+1}\varepsilon_1$-symmetric, or else $(x,t)$ does not lie at the center of an $\varepsilon_1$-neck. In the former case the induction hypothesis implies $|x - p_t| \leq 2^{\frac{j-1}{400}}\Lambda$, and in the latter we have $|x - p_t| \leq \Lambda$. So in both cases $|x - p_t| \leq 2^{\frac{j-1}{400}}\Lambda$. Since $ t \leq \bar t \leq T$, Lemma~\ref{tip to neck growing} implies that $|\bar x - p_{\bar t}| \leq |\bar x - p_{t}|$. Combining all of this we obtain
    \begin{align*}
        |\bar x - p_{\bar t}| &\leq |\bar x - p_{t}|\\
        &\leq |x - p_t| + |\bar x - x|\\
        &\leq 2^{\frac{j-1}{400}}\Lambda + 10\gamma(0,1,\dots,1)LG(\bar x, \bar t)^{-1}\\
        &\leq 2^{-\frac{1}{400}}|\bar x - p_{\bar t}| + 10^{-5}|\bar x - p_{\bar t}|\\
        &<|\bar x - p_{\bar t}|.
    \end{align*}
This is a contradiction, so the assertion holds for $j$, and the claim follows by induction. 
\end{proof}

We now proceed with the main result of this section.

\begin{theorem}\label{rot symmetry}
For each $t \leq 0$ the hypersurface $M_t$ is rotationally symmetric. 
\end{theorem}
\begin{proof}
Fix an arbitrary time $\bar t \leq T$. For each $j$, let 
    \[\Omega^{(j)} = \{(x,t) : t \leq \bar t, \; x \in M_t, \; |x - p_t| \leq 2^{\frac{j}{400}}\Lambda\}.\]
Let $Q$ be a large constant to be chosen later. By Proposition~\ref{lower curvature}, when $j$ is sufficiently we have 
    \[G(x,t) \geq Q \Lambda |x - p_t|^{-1} \geq 2^{-j/400}Q\]
for each $(x,t) \in \Omega^{(j)}$. 

By Proposition~\ref{end symmetric}, every point in $\partial \Omega^{(j)}$ is $2^{-j}\varepsilon_1$-symmetric. That is, for each $(x,t) \in \partial \Omega^{(j)}$ there exists a normalised set of rotation vector fields  
    \[\mathcal K^{(x,t)} = \{K_\alpha^{(x,t)}:1 \leq \alpha \leq \tfrac{1}{2}n(n-1)\}\]
such that $\max_\alpha |K_\alpha^{(x,t)}|G \leq 10\gamma(0,1,\dots,1)$ at $(x,t)$ and $\max_\alpha |\langle K_\alpha^{(x,t)}, \nu\rangle| G \leq 2^{-j}\varepsilon_1$ in $\hat {\mathcal P}(x,t,10,100)$. Arguing as in \cite[Theorem~5.4]{Brendle_Choi_b}, we find that for each $j$ there exists a normalised set of rotation vector fields $\mathcal K^{(j)} = \{K_\alpha^{(j)}:1 \leq \alpha \leq \frac{1}{2}n(n-1)\}$ with the following property. If $(x,t) \in \partial \Omega^{(j)}$ is such that $\bar t - 2^{j/100} \leq t \leq \bar t$, then
    \[\inf_{\omega \in O(\frac{1}{2}n(n-1))} \max_\alpha \bigg|K_\alpha^{(j)} - \sum_{\beta = 1}^{\frac{1}{2}n(n-1)}\omega_{\alpha\beta}K_{\beta}^{(x,t)}\bigg| \leq C2^{-j/2}\]
at the point $(x,t)$. It follows that if $(x,t) \in \partial \Omega^{(j)}$ is such that $\bar t - 2^{j/100} \leq t \leq \bar t$ then 
    \[\max_\alpha |\langle K_\alpha^{(j)}, \nu\rangle| \leq C2^{-j/2}\]
at $(x,t)$. Finally, for $\bar t - 2^{j/100} \leq t \leq \bar t$ we have 
    \[|p_t - p_{\bar t}| \leq C2^{j/100},\]
and hence 
    \[\max_\alpha |\langle K_\alpha^{(j)}, \nu\rangle| \leq C2^{j/100}\]
for every $(x,t) \in \Omega^{(j)}$ such that $\bar t - 2^{j/100} \leq t \leq \bar t$. 

For each $1 \leq \alpha \leq \frac{1}{2}n(n-1)$ we define a function $f^{(j)}_\alpha : \Omega^{(j)} \to \mathbb{R}$ by 
    \[f^{(j)}_\alpha := \exp(-2^{-j/200}(\bar t - t)) \frac{\langle K_\alpha^{(j)}, \nu\rangle}{G - 2^{-j/400}}.\]
When $j$ is large each of these functions satisfies
    \[|f^{(j)}_\alpha(x,t)| \leq C2^{-j/4},\]
for every $(x,t) \in \partial \Omega^{(j)}$ such that $\bar t - 2^{j/100} \leq t \leq \bar t$, and for every $(x,t) \in \Omega^{(j)}$ such that $t = \bar t - 2^{j/100}$. Moreover, we have 
    \begin{align*}(\partial_t -\dot\gamma^{pq}\nabla_p \nabla_q) f^{(j)}_\alpha &= \frac{2}{G - 2^{-j/400}}\dot\gamma^{pq}\nabla_p G\nabla_q f^{(j)}_\alpha \\
    &- 2^{-j/400}\bigg(\frac{\dot\gamma^{pq}A^2_{pq}}{G - 2^{-j/400}} - 2^{-j/400}\bigg)f^{(j)}_\alpha.
    \end{align*}
We estimate the final term on the right by 
    \[\frac{\dot\gamma^{pq}A^2_{pq}}{G - 2^{-j/400}} - 2^{-j/400} \geq \frac{1}{C} \frac{G^2}{G - 2^{-j/400}} - 2^{-j/400} \geq \frac{1}{C}G - 2^{-j/400}.\]
Choosing $Q > C$ ensures that the right-hand side is positive when $j$ is large. We may then apply the the parabolic maximum principle to conclude that 
    \[|f^{(j)}_\alpha(x,t)| \leq C2^{-j/4}\]
for every $(x,t) \in \Omega^{(j)}$ such that $\bar t - 2^{j/100} \leq t \leq \bar t$. 

As $j \to \infty$, $\mathcal K^{(j)}$ converges to a normalised set of rotation vector fields $\bar{\mathcal K}$ such that $\langle \bar K_\alpha, \nu\rangle$ vanishes identically on $M_t$ for every $1\leq\alpha\leq\frac{1}{2}n(n-1)$ and $t \leq \bar t$. In particular, $M_{\bar t}$ is rotationally symmetric. Since $\bar t \leq T$ was chosen arbitrarily, we conclude that $M_t$ is rotationally symmetric for $t \leq T$.

It is now straightforward to show that $M_t$ is rotationally symmetric for all $t \leq 0$. The quantity $u_\alpha := \langle \bar K_\alpha, \nu \rangle$ satisfies
    \[\partial_t u_\alpha = \dot\gamma^{ij}\nabla_i \nabla_j u_\alpha + \dot \gamma^{ij}A_i^kA_{kj}u_\alpha.\]
Moreover, the function $h(x,t) = e^{2Ct}(|x|^2 + 1)$ satisfies
    \[\partial_t h = \dot\gamma^{ij}\nabla_i \nabla_j h + 2Ch -2e^{2Ct}\dot\gamma^{pq}g_{pq} > \dot\gamma^{ij}\nabla_i \nabla_j h + \dot \gamma^{ij}A_i^kA_{kj} h\]
for $t \geq T$ if $C$ is sufficiently large. By the maximum principle, $\sup_{M_t} \tfrac{|u_\alpha|}{h}$ is nonincreasing for $t \leq 0$. Since $\langle \bar K_\alpha, \nu\rangle = 0$ vanishes at $t =T$, we have $\langle \bar K_\alpha, \nu\rangle = 0$ on $M_t$ for all $t \leq 0$.
\end{proof}


\section{Uniqueness of ancient solutions with rotational symmetry}\label{translating section}

Let $M_t$, $t \in (-\infty,0]$, be a strictly convex ancient $G$-flow which is noncollapsing and uniformly two-convex. We assume in addition that $M_t$ is noncompact and rotationally symmetric about the $x_{n+1}$-axis in $\mathbb{R}^{n+1}$. Therefore, there is a convex function $f(r,t)$ such that $M_t$ is the set of points $(r\theta, f(r,t))$ for $\theta \in S^{n-1}$. The function $f$ satisfies 
    \[f_t = G\langle -e_{n+1}, \nu\rangle^{-1}\]
and hence 
    \[f_t = \gamma\bigg(\frac{f_{rr}}{1+f_r^2}, \frac{f_r}{r}, \dots, \frac{f_r}{r}\bigg).\]
We can also express $M_t$ in terms of the radius function $r(z,t)$, which is determined by 
    \[f(r(z,t), t) = z.\]
We have 
    \[r_t = - \gamma\bigg(-\frac{r_{zz}}{1+r_z^2}, \frac{1}{r}, \dots, \frac{1}{r}\bigg).\]
Since $M_t$ is strictly convex we may assume that
    \[r > 0, \qquad r_z > 0, \qquad r_t < 0 \qquad \text{and} \qquad r_{zz} < 0.\]
Without loss of generality, we may assume $r(0,0) = 0$.

We adapt the arguments in Section~6 of \cite{Brendle_Choi_a} (see also Section~6 of \cite{Brendle_Choi_b}) to prove that $M_t$ moves by translation. It follows that $M_t$ is the bowl soliton. 

\begin{proposition}\label{symmetry implies translating}
The hypersurfaces $M_t$ move by translation.
\end{proposition}

Let $q_t = (0, f(0,t))$ denote the tip of $M_t$, and set $G_{\operatorname{tip}}(t)= G(q_t,t)$. As a consequence of the differential Harnack inequality we know that $G_{\operatorname{tip}}$ is nondecreasing with time, and hence 
    \[\mathcal G := \lim_{t \to -\infty} G_{\operatorname{tip}}\]
is well defined. Proposition~\ref{curvature at tip} implies that $\mathcal G > 0$.

Next we establish that $f_t(r,t)$ is a nondecreasing function of time. 

\begin{lemma}\label{f_t increasing in time}
    We have $f_{tt}(r,t) \geq 0$ everywhere.
\end{lemma}
\begin{proof}
Proposition~\ref{differential harnack} asserts that
    \[\frac{\partial}{\partial t}G + 2\langle V, \nabla G\rangle + A(V,V) \geq 0\]
for every vector field $V$ on $M_t$. Setting $V = G\langle -e_{n+1}, \nu\rangle^{-1} e_{n+1}^{\top}$, where $\top$ denotes the component tangent to $M_t$, we obtain 
    \[\bigg(\frac{\partial}{\partial t} + V^i\nabla_i\bigg)(G\langle -e_{n+1}, \nu \rangle^{-1}) \geq 0.\]
On the other hand
    \[f_{tt} = \bigg(\frac{\partial}{\partial t} + V^i\nabla_i\bigg)(G\langle -e_{n+1}, \nu \rangle^{-1}),\]
so this implies the claim.
\end{proof}

We conclude that $f_t(r,t)$ is bounded from below by $\mathcal G$. 

\begin{lemma}\label{speed limit tip}
We have $f_t(r,t) \geq \mathcal G$ everywhere. Moreover, for each $r_0 > 0$, 
    \[\lim_{t\to-\infty} \sup_{r \leq r_0} f_t(r,t) = \mathcal G.\]
\end{lemma}
\begin{proof}
Fix a sequence $t_k \to -\infty$ and consider the sequence of shifted solutions $M_t^k = M_{t+t_k} - q_{t_k}$. By Proposition~\ref{compactness} this sequence converges to a rotationally symmetric limiting solution $\bar M_t$ which is defined for all $t \in (-\infty,\infty)$ and is strictly convex. Moreover, the speed of $\bar M_t$ at its tip is constant in time and equal to $\mathcal G$. Therefore, by Proposition~\ref{differential harnack}, the hypersurfaces $\bar M_t$ move by translation. We conclude that 
    \[\lim_{k \to \infty} \sup_{r \leq r_0} |f_t(r,t_k) - \mathcal G| = 0,\]
and since $f_{tt}(r,t) \geq 0$, it follows that $f_t(r,t) \geq \mathcal G$ at every point in spacetime.  
\end{proof}

Next we show that $f_t(r,t)$ is nondecreasing as a function of $r$. 

\begin{lemma}\label{f_t increasing in space}
We have $f_{tr}(r,t)\geq 0$ everywhere. 
\end{lemma}
\begin{proof}
Fix a time $t_0 \leq 0$ and a radius $r_0 > 0$ at which $f(r_0,t_0)$ is defined. We consider the parabolic region
    \[Q_T := \{(x,t) \in \mathbb{R}^n \times [T, t_0] : x_1^2 + \dots + x_{n-1}^2 \leq r_0^2\}.\]
Using the evolution equations
    \[\frac{\partial}{\partial t} G = \dot \gamma^{ij} \nabla_i \nabla_j G + \dot \gamma^{ij} A_i^kA_{kj} G\]
and 
    \[\frac{\partial}{\partial t} \langle -e_{n+1}, \nu \rangle = \dot \gamma^{ij} \nabla_i \nabla_j \langle -e_{n+1}, \nu \rangle + \dot \gamma^{ij} A_i^kA_{kj} \langle -e_{n+1}, \nu \rangle\]
we conclude that 
    \[\sup_{Q_T} G \langle -e_{n+1}, \nu \rangle^{-1} = \max\bigg\{ \sup_{r \leq r_0, \; t = T} G  \langle -e_{n+1}, \nu \rangle^{-1}, \sup_{r = r_0, \; T \leq t \leq t_0} G  \langle -e_{n+1}, \nu \rangle^{-1} \bigg\}.\]
Since $f_t = G  \langle -e_{n+1}, \nu \rangle^{-1}$ it follows that 
    \begin{align*}
        \sup_{Q_T} f_t = \max\bigg\{ \sup_{r \leq r_0, \; t = T} f_t, \sup_{r = r_0, \; T \leq t \leq t_0} f_t \bigg\}\leq \max\bigg\{\sup_{r \leq r_0, \; t = T} f_t, \,f_t(r_0,t_0) \bigg\},
    \end{align*}
where we have made use of Lemma~\ref{f_t increasing in time}. Sending $T \to -\infty$ and appealing to Lemma~\ref{speed limit tip} we obtain 
    \[f_t(0, t_0) \leq \max \{\mathcal G, f_t(r_0,t_0) \} \leq f_t(r_0, t_0).\]
Since $r_0$ was arbitrary this completes the proof. 
\end{proof}

There exists a small constant $\varepsilon_0$ such that $q_t$ does not lie at the center of an $\varepsilon_0$-neck. Since $G$ is at least $\mathcal G$ at the point $q_t$, by Lemma~\ref{neck-cap}, there exists a decreasing function $\Lambda:(0,\varepsilon_0] \to \mathbb{R}$ such that if $x \in M_t$ satisfies $|x - q_t| \geq \Lambda(\varepsilon)$ then $x$ lies at the center of an $\varepsilon$-neck. 

Recall the notation $F(0,1) = \gamma(0,1,\dots,1)$.

\begin{lemma}\label{first derivative bound on a neck}
On every $\varepsilon_0$-neck we have $r r_z = r/f_r \leq 4
(F(0,1)+C_0\varepsilon_0)\mathcal G^{-1}$, where $C_0$ depends only on $n$ and $\gamma$. 
\end{lemma}
\begin{proof}
On an $\varepsilon_0$-neck we have $\frac{1}{f_r} = r_z \leq \varepsilon_0$. Moreover, the principal curvature in the radial direction is bounded from above by $\varepsilon_0/r$. This gives
    \[\frac{f_{rr}}{(1+f_r^2)^{3/2}} \leq \frac{\varepsilon_0}{r},\]
and hence 
    \[\frac{rf_{rr}}{(1+f_r^2)f_r} \leq \varepsilon_0(1+\varepsilon_0^2)^{1/2}.\]
Using Lemma~\ref{speed limit tip} we obtain
    \[\mathcal G \leq f_t = \gamma\bigg(\frac{f_{rr}}{1+f_r^2}, \frac{f_r}{r}, \dots, \frac{f_r}{r}\bigg) = \gamma\bigg(\frac{rf_{rr}}{(1+f_r^2)f_r}, 1, \dots, 1\bigg)\frac{f_r}{r} \leq (F(0,1) + C_0\varepsilon_0)\frac{f_r}{r}.\]
\end{proof}

The following lemma establishes scaling-invariant derivative bounds for $r(z,t)$ on any $\varepsilon_0$-neck. 

\begin{lemma}\label{higher derivative bounds on a neck}
There is a constant $C_1 > 1$ such that on every $\varepsilon_0$-neck with $r \geq 1$ we have $r^m |\tfrac{d^m}{dz^m} r| \leq C_1$, $m = 1, 2, 3$.  
\end{lemma}
\begin{proof}
For $m = 1$ we may appeal to Lemma~\ref{first derivative bound on a neck}. 

Set $u = \langle - e_{n+1}, \nu\rangle$. Since $f_t = G u^{-1}$, Lemma~\ref{speed limit tip} implies $u \leq G \mathcal{G}^{-1}$. Suppose $(x,t)$ lies at the center of an $\varepsilon_0$-neck. We may then use the evolution equation
    \[\frac{\partial}{\partial t} u = \dot\gamma^{ij}\nabla_i\nabla_j u + \dot\gamma^{ij}A_i^kA_{kj} u\]
and standard interior estimates for linear parabolic equations to conclude that 
    \[|\nabla^m u|^2 \leq C G^{2m+2}\]
at $(x,t)$. 

The claim follows from this estimate by writing the left-hand side in terms of $r$ and rearranging. Indeed, we have $u = r_z(1+r_z^2)^{-1/2}$, and the induced metric is 
    \[g = (1+r_z^2)dz^2 + r^2 g_{S^{n-1}}.\]
Using these formulae the rest of the argument is straightforward. 
\end{proof}

Next we prove a sharp estimate for $r_{zz}(z,t)$ which is valid when $r(z,t)$ is sufficiently large. 

\begin{lemma}\label{improved second derivative bound}
    Let $C_2 = 2 + 2\Lambda(\varepsilon_0) + \tfrac{64}{9}(F(0,1)+C_0\varepsilon_0)^2\mathcal G^{-2}$. If $r(z,t) \geq C_2$ then 
        \[|(rr_z)_z(z,t)| \leq C_3 r(z,t)^{-3/2} \qquad \text{and} \qquad 0 \leq - r_{zz}(z,t) \leq C r(z,t)^{-5/2}.\]
\end{lemma}
\begin{proof}
Let us fix a point $(\bar r,\bar t)$ such that $\bar r \geq C_2$ and let $\bar z = f(\bar r, \bar t)$. We then have that every point $x = (r\theta, f(r,t))$  in $M_t$ with $r \geq \bar r$ lies at the center of an $\varepsilon_0$-neck. 

Using $f_r \geq (F(0,1)+C_0\varepsilon_0)^{-1} \mathcal{G} r$ we estimate
    \[\bar z - f\bigg(\frac{\bar r}{2}, \bar t\bigg) = \int_{\bar r/2}^{\bar r} f_r(r,\bar t)\,dr \geq \frac{3}{8}(F(0,1)+C_0\varepsilon_0)^{-1} \mathcal{G} \bar r^2.\]
Since $\bar r \geq \tfrac{64}{9}(F(0,1)+C_0\varepsilon_0)^2\mathcal G^{-2}$ this implies that $f(\tfrac{\bar r}{2}, \bar t) \leq \bar z - \bar r^{3/2}$. In other words, $r(z,\bar t) \geq \bar r/2$ for $z \geq \bar z - \bar r^{3/2}$. Since $r(z,t)$ is nonincreasing in time it follows that $r(z,t) \geq \bar r/2$ for $t \leq \bar t$ and $z \geq \bar z - \bar r^{3/2}$. In particular, $r \geq \bar r/2$ holds in the set $Q:= [\bar z - \bar r^{3/2}, \bar z + \bar r^{3/2}] \times [\bar t - \bar r^3, \bar t]$, and hence every point $(x,t)$ with $(x_{n+1},t) \in Q$ lies at the center of an $\varepsilon_0$-neck. 

We compute 
    \begin{align*}
        (rr_z)_t &= -r_z \gamma\bigg(-\frac{r_{zz}}{1+r_z^2}, \frac{1}{r}, \dots, \frac{1}{r}\bigg) + \dot \gamma^1\bigg(-\frac{r_{zz}}{1+r_z^2}, \frac{1}{r}, \dots, \frac{1}{r}\bigg)\bigg(\frac{r r_{zzz}}{1+r_z^2} - \frac{2rr_z r_{zz}^2}{(1+r_z^2)^2}\bigg)\\
        & + \frac{r_z}{r}\sum_{i=2}^n\dot \gamma^i\bigg(-\frac{r_{zz}}{1+r_z^2}, \frac{1}{r}, \dots, \frac{1}{r}\bigg).
    \end{align*}
Inserting 
    \begin{align*}
        \gamma\bigg(-\frac{r_{zz}}{1+r_z^2}, \frac{1}{r}, \dots, \frac{1}{r}\bigg) &= - \dot\gamma^1\bigg(-\frac{r_{zz}}{1+r_z^2}, \frac{1}{r}, \dots, \frac{1}{r}\bigg)\frac{r_{zz}}{1+r_z^2}\\
        &+ \frac{1}{r}\sum_{i=2}^n \dot\gamma^i\bigg(-\frac{r_{zz}}{1+r_z^2}, \frac{1}{r}, \dots, \frac{1}{r}\bigg)
        \end{align*}
we obtain 
    \begin{align*}
        (rr_z)_t &= \dot \gamma^1\bigg(-\frac{r_{zz}}{1+r_z^2}, \frac{1}{r}, \dots, \frac{1}{r}\bigg)\bigg(\frac{r_z r_{zz}}{1+r_z^2} + \frac{r r_{zzz}}{1+r_z^2} - \frac{2rr_z r_{zz}^2}{(1+r_z^2)^2}\bigg)\\
        &= \dot \gamma^1\bigg(-\frac{r_{zz}}{1+r_z^2}, \frac{1}{r}, \dots, \frac{1}{r}\bigg)\bigg( (rr_z)_{zz} - \frac{(2+3r_z^2)r_z r_{zz} + r r_z^2 r_{zzz}}{1+r_z^2} - \frac{2rr_z r_{zz}^2}{(1+r_z^2)^2}\bigg).
    \end{align*}

Let us define a smooth function $a = a(z,t)$ by
    \[a(z,t) = \dot \gamma^1\bigg(-\frac{r_{zz}(z,t)}{1+r_z(z,t)^2}, \frac{1}{r(z,t)}, \dots, \frac{1}{r(z,t)}\bigg),\]
so that we may write
    \[(rr_z)_t - a (rr_z)_{zz} =  - a \frac{(2+3r_z^2)r_z r_{zz} + r r_z^2 r_{zzz}}{1+r_z^2} -a \frac{2rr_z r_{zz}^2}{(1+r_z^2)^2}.\]
We have $C^{-1} \leq a \leq C$, so using Lemma~\ref{higher derivative bounds on a neck} we find that 
    \[\sup_Q |rr_z| \leq C \qquad \text{ and } \qquad \sup_Q |(rr_z)_t - a(rr_z)_{zz}| \leq C\bar r^{-3}.\]
Standard interior estimates for linear parabolic equations then give 
    \[|(rr_z)_z| \leq  C \bar r^{-3/2} \sup_{Q} |rr_z| + C \bar r^{3/2} \sup_Q|(rr_z)_t - a(rr_z)_{zz}|\]
at $(\bar z, \bar t)$. So we conclude that 
    \[|r_{zz}| \leq C \bar r^{-5/2}\]
at $(\bar z, \bar t)$. 
\end{proof}

For each $z<0$ we define $\mathcal T(z)$ such that $r(z,t) > 0$ for $t < \mathcal T(z)$ and 
    \[\lim_{t \to \mathcal T(z)} r(z,t) = 0.\] 

\begin{lemma}\label{r^2 goes like -t}
Whenever $z < 0$ and $t < \mathcal T(z)$ we have
    \[2F(0,1)(\mathcal T(z) - t) \leq r(z,t)^2.\]
If in addition $r(z,t) > 2C_2$ then we also have
    \[r(z,t)^2 \leq 2F(0,1)(\mathcal T(z) - t) + C(\mathcal T(z) - t)^{1/4} + C_2^2.\]
\end{lemma}
\begin{proof}
Let us fix a point $(\bar z, \bar t)$ such that $\bar z < 0$. Since $r_{zz} < 0$ we have
    \[\frac{d}{dt}(r(\bar z,t)^2 + 2F(0,1)t) = -2\gamma\bigg(-\frac{r(\bar z,t) r_{zz}(\bar z, t)}{1+r_z(\bar z, t)^2}, 1, \dots, 1\bigg) + 2F(0,1) < 0,\]
and hence   
    \[r(\bar z, \bar t)^2 \geq 2F(0,1)(\mathcal T(\bar z) - \bar t).\]

From now on we assume $r(\bar z, \bar t) > 2C_2$. Let $\bar t < \tilde t < \mathcal T(\bar z)$ denote the time at which $r(\bar z, \tilde t) = C_2$. We have 
    \begin{align*}
    2F(0,1)&\geq 2\gamma\bigg(-\frac{r(\bar z,t) r_{zz}(\bar z, t)}{1+r_z(\bar z, t)^2}, 1, \dots, 1\bigg) - C|r(\bar z, t) r_{zz}(\bar z, t)|,
    \end{align*}
and hence
    \begin{align*}
    \frac{d}{dt} (r(\bar z, t)^2 + 2F(0,1)t) &= -2\gamma\bigg(-\frac{r(\bar z,t) r_{zz}(\bar z, t)}{1+r_z(\bar z, t)^2}, 1, \dots, 1\bigg)+ 2F(0,1)\\
    &\geq -C|r(\bar z, t) r_{zz}(\bar z, t)|
    \end{align*}
for $\bar t < t < \tilde t$. Lemma~\ref{improved second derivative bound} now gives
    \[\frac{d}{dt} (r(\bar z, t)^2 + 2F(0,1)t) \geq -Cr(\bar z, t)^{-3/2}.\]
Integrating this inequality over $\bar t < t < \tilde t$ gives 
    \begin{align*}
    r(\bar z , \bar t)^2 - r(\bar z, \tilde t)^2 &= \,2F(0,1)(\tilde t - \bar t) - \int_{\bar t}^{\tilde t}\frac{d}{dt} (r(\bar z, t)^2 + 2F(0,1)t)\, dt\\
     &\leq  \,2F(0,1)(\tilde t - \bar t) + C \int_{\bar t}^{\tilde t} r(\bar z, t)^{-3/2} \, dt\\
     &\leq  \,2F(0,1)(\tilde t - \bar t) + C \int_{\bar t}^{\tilde t}(\mathcal T(\bar z) - t)^{-3/4}\,dt\\
    &\leq  \,2F(0,1)(\mathcal T(\bar z) - \bar t) + C(\mathcal T(\bar z) - \bar t)^{1/4}.
    \end{align*}
Inserting $r(\bar z, \tilde t)^2 = C_2^2$ gives the claim.
\end{proof}

Using the estimates for $r(z,t)$ established in Lemma~\ref{r^2 goes like -t}, we can now prove a sharp lower bound for $r(z,t)r_z(z,t)$ at $z = 0$. 

\begin{lemma}\label{sharp lower bound rr_z at z=0}
Let us fix $\delta > 0$ arbitrarily. We have 
    \[r(0,t)r_z(0,t) \geq F(0,1)(\mathcal G^{-1} -\delta)\]
whenever $-t$ is sufficiently large.
\end{lemma}
\begin{proof}
In the following we assume $-t$ is so large that $R = r(0,t) \geq 2C_2$. Recall that $C_2 \geq \Lambda(\varepsilon_0)$. Consequently, every point $(x,t)$ with $x_{n+1} = 0$ lies at the center of an $\varepsilon_0$-neck. This implies $|r(z,t) - R| \leq \varepsilon_0 R$ for $|z| \leq 2R$. Then we have $R \geq C_2$ for $|z| \leq 2R$, so by Lemma~\ref{improved second derivative bound} we have
    \[|(rr_z)_z| \leq C_3 R^{-3/2}.\]
Lemma~\ref{r^2 goes like -t} implies $r(0,t) \to \infty$ as $t \to -\infty$, so we may assume $-t$ is so large that $R^{1/2} \geq 4C_3 \delta^{-1}$. The fundamental theorem of calculus then gives
    \[|r(z,t)r_z(z,t) - r(0,t)r_z(0,t)| \leq 2C_3 R^{-1/2} \leq \delta /2\]
for $|z| \leq 2R$. Lemma~\ref{r^2 goes like -t} gives
    \[r(-R,t)^2 \geq 2F(0,1)(\mathcal T(-R) - t) \qquad \text{and} \qquad r(-2R,t)^2 \geq 2F(0,1)(\mathcal T(-2R) - t).\]
Assuming $-t$ is so large that $r(-2R,t) >2C_2$, the second of these inequalities and Lemma~\ref{r^2 goes like -t} imply
    \begin{align*}
    r(-2R,t)^2 &\leq 2F(0,1)(\mathcal T(-2R) - t) + C(\mathcal T(-2R) - t)^{1/4} + C_2^2\\
    &\leq 2F(0,1)(\mathcal T(-2R) - t)+ Cr(-2R,t)^{1/2} + C_2^2\\
    &\leq 2F(0,1)(\mathcal T(-2R) - t) + CR^{1/2} + C_2^2.
    \end{align*}
We thus have
    \begin{align*}
        r(-R,t)^2 - r(-2R, t)^2 &\geq 2F(0,1)(\mathcal T(-R) - \mathcal T(-2R))- C_4R^{1/2} - C_2^2.
    \end{align*}
On the other hand, Lemma~\ref{speed limit tip} implies
    \[\mathcal T(-R,t) - \mathcal T(-2R,t) \geq (\mathcal G^{-1} 
 - \delta/4)R\]
for sufficiently large $-t$. We may assume $-t$ is so large that 
    \[C_2^2 + C_4R^{1/2} \leq F(0,1) \delta R/2,\]
so that combining these estimates gives
    \begin{align*}
    r(-R,t)^2 - r(-2R, t)^2 &\geq 2F(0,1)(\mathcal G^{-1} 
 - \delta/2)R.
    \end{align*}
Therefore, if $-t$ is sufficiently large,
    \[\sup_{z \in [-2R,-R]} r(z,t)r_z(z,t) \geq F(0,1)(\mathcal G^{-1} 
 - \delta/2).\]
Recalling that $r(0,t)r_z(0,t) \geq r(z,t)r_z(z,t) - \delta/2$ for $|z| \leq 2R$, we finally obtain
    \[r(0,t)r_z(0,t) \geq F(0,1)(\mathcal G^{-1} - \delta)\]
for $-t$ sufficiently large. 
\end{proof}

We recall \cite[Proposition~6.9]{Brendle_Choi_a}. 

\begin{lemma}
We define a smooth function $\psi : (0,\infty)\times(0,\infty) \to \mathbb{R}$ by 
    \[\psi(z,t) = \frac{1}{\sqrt{4\pi t}}\int_0^\infty (e^{-\frac{(z-y)^2}{4t}} - e^{-\frac{(z+y)^2}{4t}})\,dy.\]
Then $\psi$ is a solution to the heat equation $\psi_t = \psi_{zz}$. Moreover, for each $z > 0$ and $t > 0$ we have $\psi_{zz}(z,t) < 0$, and 
    \[\lim_{z \to 0} \psi(z,t) = 0, \qquad \lim_{z \to \infty} \psi(z,t) = 1, \qquad \lim_{t\to 0} \psi(z,t) = 1, \qquad \lim_{t\to \infty} \psi(z,t) = 0.\]
\end{lemma}

Using the function $\psi(z,t)$ to construct an appropriate barrier, we extend the lower bound for $r(0,t)r_z(0,t)$ proven in Lemma~\ref{sharp lower bound rr_z at z=0} to $z \geq 0$. 

\begin{lemma}\label{sharp lower bound rr_z}
Given $\delta > 0$ there is a time $\bar t$ dependig on $\delta$ such that 
    \[r(z,t)r_z(z,t) \geq F(0,1)(\mathcal G^{-1} - \delta)\]
for all $z \geq 0$ and $t \leq \bar t$. 
\end{lemma}
\begin{proof}
By Lemma~\ref{improved second derivative bound}, we have $1 + r r_{zz} \geq 0$ for $r \geq C_1 + C_2$. This implies 
    \[(rr_z)_t = a \frac{(rr_z)_{zz}}{1+r_z^2} - a\frac{2r_zr_{zz}(1 + r_z^2 + rr_{zz})}{(1+r_z^2)^2} \leq a \frac{(rr_z)_{zz}}{1+r_z^2}\]
for $r \geq C_1 + C_2$. By Lemma~\ref{sharp lower bound rr_z} we may choose $\bar t$ so that 
    \[r(0,t)r_z(0,t) \geq F(0,1)(\mathcal G^{-1} - \delta)\]
for every $t \leq \bar t$. Moreover, for a suitable choice of $\bar t$ we can arrange that $r(z,t) \geq C_1 + C_2$ for all $z \geq 0$ and $t \leq \bar t$. For each $s < \bar t$ we define a barrier function $\psi^{\delta, s}(z,t)$ by 
    \[\psi^{\delta,s}(z,t) = F(0,1)(\mathcal G^{-1} - 2\delta - \mathcal G^{-1}\psi(qz,t-s)), \qquad (z,t) \in (0,\infty)\times(s, \bar t],\]
where $q$ is a positive constant. We claim that if $q$ is chosen correctly then $rr_z \geq \psi^{\delta, s}$ for all $z \geq 0$ and $t \in (s,\bar t]$. 

By our choice of $\bar t$, 
    \[r(0,t)r_z(0,t) \geq F(0,1)(\mathcal G^{-1} - \delta) > \limsup_{z\to 0} \psi^{\delta,s}(z,t)\]
for each $t \in (s,\bar t]$. Moreover, 
    \[\liminf_{z\to\infty} r(z,t)r_z(z,t) \geq 0 > \limsup_{z\to\infty}\psi^{\delta,s}(z,t)\]
for each $t \in (s,\bar t]$. Finally, 
    \[r(z,s)r_z(z,s) \geq 0 > \limsup_{t \to s} \psi^{\delta,s}(z,t).\]
Thus, if the inequality $rr_z > \psi^{\delta,s}$ fails, then there is a point $(\tilde z, \tilde t) \in (0,\infty)\times(s,\bar t]$ such that $r(\tilde z, \tilde t)r_z(\tilde z, \tilde t) = \psi^{\delta,s}(\tilde z, \tilde t)$ and $r(\tilde z, t)r_z(\tilde z, t) > \psi^{\delta,s}(\tilde z, t)$ for $t \in (s,\bar t)$. At $(\tilde z, \tilde t)$ we have
    \[a \frac{(\psi^{\delta,s})_{zz}}{1+r_z^2} \leq a \frac{(rr_z)_{zz}}{1+r_z}^2 \leq (rr_z)_t \leq \psi^{\delta,s}_t = q^{-2} \psi^{\delta,s}_{zz}.\]
Since $a \geq C^{-1}$ and $r_z^2 \leq \varepsilon_0$, recalling that $\psi^{\delta,s}_{zz} > 0$, we have a contradiction if $q > 100 + C$.

We thus conclude that $r(z,t)r_z(z,t) > \psi^{\delta,s}(z,t)$ for all $z > 0$ and $t \in (s, \bar t]$, where $s$ can take any value such that $s < \bar t$. Sending $s \to -\infty$, this gives 
    \[r(z,t)r_z(z,t) \geq F(0,1)(\mathcal G^{-1} - 2\delta)\]
for all $z \geq 0$ and $t \leq \bar t$. 
\end{proof}

Lemma~\ref{sharp lower bound rr_z} lets us conclude that, at sufficiently early times, $r(z,t)^2$ grows at least linearly in $z$ as $z \to \infty$. 

\begin{lemma}\label{r goes like sqrt z}
We can find a time $T \in (-\infty, 0]$ such that $r(z,t)^2 \geq \mathcal G^{-1} z$ whenever $z \geq 0$ and $t \leq T$. In particular, if $t \leq T$ then the function $f(r,t)$ is defined for all $r \geq 0$. 
\end{lemma}
\begin{proof}
By Lemma~\ref{sharp lower bound rr_z} we can choose $T$ so that $r(z,t)r_z(z,t) \geq \tfrac{1}{2}\mathcal G^{-1}$ for all $z \geq 0$ and $t \leq T$. The assertion then follows from the fundamental theorem of calculus. 
\end{proof}

With these preparations out of the way, we can show that the limit of $r(z,t)r_z(z,t)$ as $z \to \infty$ is constant in time. 

\begin{lemma}\label{rr_z constant at infinity}
For each $t \leq T$ we have $\lim_{z \to \infty} r(z,t)r_z(z,t) = F(0,1)\mathcal G^{-1}$.
\end{lemma}
\begin{proof}
Lemma~\ref{first derivative bound on a neck} gives 
    \[\limsup_{z\to\infty} r(z,t)r_z(z,t) = F(0,1)\mathcal G^{-1}\] 
for $t \leq 0$. So it suffices to show that 
    \[\liminf_{z \to \infty} r(z,t)r_z(z,t) \geq F(0,1)\mathcal G^{-1}\]
for $t \leq T$. Given any $\delta$, Lemma~\ref{sharp lower bound rr_z} provides a $\bar t \leq T$ such that 
    \[\liminf_{z\to\infty}r(z,\bar t)r_z(z,\bar t) \geq F(0,1)(\mathcal G^{-1} - 2\delta).\]
Moreover, Lemma~\ref{higher derivative bounds on a neck} gives
    \[|(rr_z)_t| \leq a \, \bigg|\frac{r_z r_{zz}}{1+r_z^2} +  \frac{r r_{zzz}}{1+r_z^2} +  \frac{2rr_z r_{zz}^2}{(1+r_z^2)^2}\bigg|  \leq \frac{C}{r^2}\]
for $r \geq C_2$. Using Lemma~\ref{r goes like sqrt z} we obtain
    \[\liminf_{z \to \infty} r(z,t)r_z(z,t) = \liminf_{z \to \infty}r(z,\bar t)r_z(z,\bar t) \geq F(0,1)(\mathcal G^{-1} - 2\delta)\]
for $t \leq T$. Since $\delta$ can be chosen arbitrarily small, this completes the proof. 
\end{proof}

We finally conclude that the hypersurfaces $M_t$ move by translation.

\begin{proof}[Proof of Proposition~\ref{symmetry implies translating}]
Since $rr_z = \tfrac{r}{f_r}$, Proposition~\ref{rr_z constant at infinity} implies 
    \[\lim_{r\to\infty} \frac{f_r}{r} = F(0,1)^{-1}\mathcal G\]
for $t \leq T$. Using this fact, the evolution equation for $f$, and the fact that 
    \[\lim_{r\to\infty} \frac{f_rr}{1+f_r^2} = 0,\]
we obtain 
    \[\lim_{r\to\infty}f_t(r,t) = F(0,1) \cdot \lim_{r\to\infty} \frac{f_r}{r} = \mathcal G\]
for $t \leq T$. Using Lemma~\ref{f_t increasing in space} we find that $f_t(r,t) \leq \mathcal G$ for $r \geq 0$ and $t \leq T$. Therefore, Lemma~\ref{speed limit tip} implies that $f_t(r,t) = \mathcal G$ for $r \geq 0$ and $t \leq T$. That is, the hypersurfaces $M_t$ move by translation with velocity $\mathcal{G}e_{n+1}$ for $t \leq T$. 

It is now straightforward to show that $M_t$ moves by translation for all $t \leq 0$. The quantity $v := G + \langle \mathcal G e_{n+1}, \nu \rangle$ satisfies
    \[\frac{\partial}{\partial t} v = \dot\gamma^{ij}\nabla_i \nabla_j v + \dot \gamma^{ij}A_i^kA_{kj}v.\]
Moreover, the function $h(x,t) = e^{2Ct}(|x|^2 + 1)$ satisfies
    \[\frac{\partial}{\partial t} h > \dot\gamma^{ij}\nabla_i \nabla_j h + \dot \gamma^{ij}A_i^kA_{kj} h\]
for $t \geq T$ if the constant $C$ is sufficiently large. By the maximum principle, $\sup_{M_t} \tfrac{|v|}{h}$ is nonincreasing for $t \leq 0$. Since $v$ vanishes at $t = T$, we have $G = -\langle \mathcal G e_{n+1}, \nu\rangle$ on $M_t$ for all $t \leq 0$. That is, $M_t$ moves by translation with velocity $\mathcal G e_{n+1}$ for $t\leq 0$.
\end{proof}

We now have all the ingredients required to prove Theorem~\ref{main}.

\begin{proof}[Proof of Theorem~\ref{main}]
Let $M_t$, $t \in (-\infty, 0]$, be a convex ancient $G$-flow which is noncollapsing and uniformly two-convex. Suppose also that $M_t$ is noncompact. If $M_t$ fails to be strictly convex, then it is a self-similarly shrinking cylinder by Lemma~\ref{splitting}. If $M_t$ is strictly convex then it is rotationally symmetric by Theorem~\ref{rot symmetry}. Therefore, $M_t$ is a translating soliton by Proposition~\ref{symmetry implies translating}. This completes the proof. 
\end{proof}


\appendix

\section{Asymptotics for the bowl soliton}\label{translator asymptotics}

A $G$-flow $M_t$ moves by translation with velocity $V$ if and only if
    \[G = -\langle V, \nu\rangle.\]
Suppose $V = \frac{1}{2}e_{n+1}$ and $M_0 = \graph u$. Then this equation can be expressed as
    \begin{equation}\label{translator}
    \gamma\bigg(\bigg(\delta^{ik} - \frac{D^i u D^k u}{1+|Du|^2}\bigg)\frac{D_k u D_j u}{\sqrt{1+|Du|^2}}\bigg)\bigg) = \frac{1}{2} \frac{1}{\sqrt{1+|Du|^2}}.
    \end{equation}
In case $u(x) = \zeta(\rho)$ where $\rho = |x|$, we have
    \begin{equation}\label{translator zeta}
    \gamma\bigg(\frac{\zeta_{\rho\rho}}{1+\zeta_\rho^2}, \frac{\zeta_\rho}{\rho}, \dots, \frac{\zeta_\rho}{\rho}\bigg) = \frac{1}{2}.    
    \end{equation}
We use the notation $F(x_1, x_2) = \gamma(x_1, x_2, \dots, x_2)$ and $\dot F^i(x) = \frac{\partial F}{\partial x_i}(x)$. As we explain in Section~\ref{barrier section}, there is a function $f:U \to \mathbb{R}$ such that $F(x_1, x_2) = x_3$ if and only if $x_1 = f(x_2, x_3)$, where 
    \[U = \{(x_2,x_3) \in \mathbb{R}^2_+ : 0 < x_3/x_2 < Q\}\]
and $Q := \lim_{x_1 \to \infty} F(x_1, 1)$. In case $\zeta$ is strictly increasing and strictly convex we may equivalently express \eqref{translator zeta} in the form 
    \begin{equation}\label{translator zeta implicit}
    \zeta_{\rho\rho} = (1+\zeta_\rho^2)f\bigg(\frac{\zeta_\rho}{\rho}, \frac{1}{2}\bigg).    
    \end{equation}

Up to parabolic rescalings and rigid motions of the ambient space there is a unique translating $G$-flow which is convex and rotationally symmetric. We refer to this solution as the bowl soliton. Existence and uniqueness of the bowl soliton were established in \cite{rengaswami2021rotationally}. We note that existence also follows from Proposition~\ref{prop: sol psi on (0,R)}---simply set $a = \infty$ in the statement and proof. Uniqueness is proven by a standard argument---given two bowl solitons with velocity $V = \frac{1}{2} e_{n+1}$, shift one of them vertically until both agree at some $\bar \rho > 0$. The maximum principle then ensures that the two solutions coincide for $\rho \leq \bar \rho$, and so by uniqueness of solutions to ODEs they also coincide for $\rho \geq 1$. 

Let us denote by $M_t$ the bowl soliton which translates with velocity $V = \frac{1}{2}e_{n+1}$ and whose tip lies at the origin. Let $\zeta(\rho)$ denote the profile function of $M_0$. In \cite{rengaswami2021rotationally} it was shown that under the assumption $(0,1,\dots,1) \in \Gamma$ (which we impose throughout this paper), the hypersurface $M_0$ is an entire graph. That is, $\zeta(\rho)$ is defined for all $\rho \geq 0$. This can also be deduced from Lemma~\ref{prop: sol psi on (0,R)}, which implies the bounds
    \[C^{-1} \rho \leq \zeta_\rho \leq C \rho,\]
where $C = C(n,\gamma)$ is a positive constant. Our goal is to establish the following asymptotic expansion for $\zeta(\rho)$ as $\rho \to \infty$.

\begin{proposition}\label{prop: expand}
As $\rho \to \infty$, the function $\zeta$ satisfies
    \begin{equation*}\label{eq: zetaRho_infty}
        \zeta_\rho(\rho) = \frac{1}{2\gamma(0,1,\dots,1)} \rho - 2\dot \gamma^1(0,1,\dots,1) \rho^{-1} + o\left(\rho^{-1}\right).
    \end{equation*}
It follows that as $\rho \to \infty$ we have
    \begin{equation*}\label{eq: zeta_infty}
         \zeta(\rho) = \frac{1}{4\gamma(0,1,\dots,1)} \rho^2 - 2\dot \gamma^1(0,1,\dots,1) \log \rho + o(\log \rho).
    \end{equation*}
\end{proposition}

For the mean curvature flow, Proposition~\ref{prop: expand} was proven in \cite{CSS07}. We refer also to \cite{rengaswami2023bowl}, which contains an independent proof of Proposition~\ref{prop: expand} (in addition to bowl soliton asymptotics for some other classes of flows). 

\begin{proof}[Proof of Proposition~\ref{prop: expand}]
Let $\vartheta = \frac{\zeta_\rho}{\rho}$. We compute
    \[\vartheta_\rho = \frac{1}{\rho}\bigg((1+\rho^2\vartheta^2) f\bigg(\vartheta, \frac{1}{2}\bigg) - \vartheta\bigg).\]
Since $f(\frac{1}{2F(0,1)},\frac{1}{2}) = 0$ and $f$ is strictly decreasing in its first argument, we see that $\vartheta_\rho <0$ whenever $\vartheta > \frac{1}{2F(0,1)}$. On the other hand, if $\vartheta < \frac{1}{2F(0,1)}$ and $\rho$ is large, we have $\vartheta_\rho > 0$. It follows that $\vartheta \to \frac{1}{2F(0,1)}$ as $\rho \to \infty$. Equivalently, we have 
    \[\zeta_\rho = \frac{1}{2F(0,1)}\rho + o(\rho)\]
as $\rho \to \infty$. 

Next we claim that $\zeta_\rho = \frac{1}{2F(0,1)}\rho + o(1)$ as $\rho \to \infty$. Let us define $\xi := \zeta_\rho - \frac{1}{2F(0, 1)}\rho$. We have 
\begin{align*}
    \xi_\rho 
    & = \left( 1 + \xi^2 + \frac{1}{F(0,1)}\,\rho \xi + \frac{1}{4F(0,1)^2}\,\rho^2 \right) \, f\left( \frac{\xi}{\rho}+ \frac{1}{2F(0,1)}, \, \frac{1}{2}\right) - \frac{1}{2F(0,1)}.
\end{align*}

Since $\frac{\xi}{\rho} = o(1)$, we may perform a Taylor expansion of $f$ about the point $(1, F(0,1))$ to obtain
    \begin{align*}
    f\left( 2F(0,1)\frac{\xi}{\rho}+1, F(0,1)\right)&=2F(0,1) \frac{\partial f}{\partial x_2} (1, F(0,1)) \frac{\xi}{\rho} + o(1)\frac{\xi}{\rho}\\
    &= -\frac{2F(0,1)^2}{\dot F^1(0,1)} \frac{\xi}{\rho} + o(1)\frac{\xi}{\rho},
    \end{align*}
and hence
\begin{align*}
    \xi_\rho 
    & = \left(\frac{1}{4F(0,1)^2}\rho^2 + o(\rho)\right) \bigg(-\frac{F(0,1)}{\dot F^1(0,1)} \frac{\xi}{\rho} + o(1)\frac{\xi}{\rho}\bigg) - \frac{1}{2F(0,1)}.
\end{align*}
We conclude that $\xi_\rho < 0$ whenever $\xi >0$ and $\rho$ is large, and $\xi_\rho > 0$ whenever $\xi<0$ and $\rho$ is large. It follows that $\xi \to 0$ as $\rho \to \infty$. 

Now consider the function $\lambda := \rho \xi$. Since $\xi = o(1)$ we have $\lambda = o(\rho)$ as $\rho \to \infty$. We claim that $\lambda \to -2\dot F^1(0,1)$ as $\rho \to \infty$. To begin with we compute
\begin{align*}
    \lambda_\rho 
    &  = \frac{\rho}{2F(0, 1)}\left( 1 + \frac{\lambda^2}{\rho^2} + \frac{\lambda}{F(0, 1)} + \frac{\rho^2}{4F(0, 1)^2} \right)f\left( 2F(0, 1) \frac{\lambda}{\rho^2} + 1, \, F(0, 1)\right)\\
    &- \frac{\rho}{2F(0, 1)} + \frac{\lambda}{\rho},
\end{align*}
and so conclude that 
\begin{align*}
    \lambda_\rho  &= \frac{\rho}{2F(0, 1)}\left( \frac{\rho^2}{4F(0, 1)^2} + o(\rho^2)\right)f\left( 2F(0, 1) \frac{\lambda}{\rho^2} + 1, \, F(0, 1)\right) - \frac{\rho}{2F(0, 1)} + o(1).
\end{align*}
as $\rho \to \infty$. Performing a Taylor expansion as above, we obtain 
\begin{align*}
    \lambda_\rho  &= \frac{\rho}{2F(0, 1)}\left( \frac{\rho^2}{4F(0, 1)^2} + o(\rho^2)\right)\bigg(-\frac{2F(0,1)^2}{\dot F^1(0,1)} \frac{\lambda}{\rho^2} + o(1)\frac{\lambda}{\rho^3}\bigg) - \frac{\rho}{2F(0, 1)} + o(1)\\
    &= \left( \frac{\rho^2}{4F(0, 1)^2} + o(\rho^2)\right)\bigg(-\frac{1}{\dot F^1(0,1)} \frac{\lambda}{\rho} + o(1)\frac{\lambda}{\rho^2}\bigg) - \frac{\rho}{2F(0, 1)} + o(1).
\end{align*}
From this we observe that $\lambda_\rho < 0$ whenever $\lambda > -2\dot F^1(0,1)$ and $\rho$ is large, and $\lambda_\rho > 0$ whenever $\lambda < -2\dot F^1(0,1)$ and $\rho$ is large. It follows that $\lambda \to -2\dot F^1(0,1)$ as $\rho \to \infty$. 

In summary, we have shown that 
    \[\rho\bigg(\zeta_\rho - \frac{1}{2F(0,1)}\rho\bigg) = -2\dot F^1(0,1) + o(1)\]
as $\rho \to \infty$. From this the claim follows. 
\end{proof}


\bibliographystyle{alpha}
\bibliography{references}

\end{document}